\documentclass[a4paper,11pt]{article} %

\usepackage{amsmath} %
\usepackage{amsthm} %
\usepackage{amssymb} %
\usepackage{mathtools} %
\usepackage{epsfig} %
\usepackage{psfrag} %
\usepackage[mathscr,mathcal]{eucal} %
\usepackage{enumerate} %
\usepackage{dsfont} %
\usepackage[margin=1in]{geometry} %
\usepackage{graphicx}
\usepackage{comment}
\usepackage[hidelinks]{hyperref}

\theoremstyle{definition}

\newtheorem{theorem}{Theorem}[section]
\newtheorem{lemma}[theorem]{Lemma}
\newtheorem{corollary}[theorem]{Corollary}
\newtheorem{proposition}[theorem]{Proposition}
\newtheorem{claim}[theorem]{Claim}

\newtheorem{definition}[theorem]{Definition}

\newtheorem{remark}[theorem]{Remark}

\numberwithin{equation}{section}

\newcommand{\N}{\mathbb{N}}
\newcommand{\Z}{\mathbb{Z}}

\newcommand{\R}{\mathbb{R}}

\newcommand{\ind}{\mathds{1}}
\newcommand{\veps}{\varepsilon}
\newcommand{\fsubset}{\subset \subset}

\newcommand{\ball}{\mathcal{B}}

\newcommand{\distance}{\mathrm{d}}

\newcommand{\prob}{\mathbb{P}}
\newcommand{\Var}{\mathrm{Var}}
\newcommand{\E}{\mathbb{E}}

\newcommand{\capacity}{\mathrm{cap}}

\newcommand{\wormspace}{W^\lozenge} 
\newcommand{\wormspacesigmaalgebra}{\mathcal{W}^\lozenge}
\newcommand{\lozengeworm}{w^\lozenge}
\newcommand{\fwormspace}{W^\bullet}

\newcommand{\firstlastvis}{\Lambda}
\newcommand{\wormspm}{\boldsymbol{\mathcal{M}}}
\newcommand{\pointmeasure}{M}

\newcommand{\wormsppp}{\mathcal{P}}

\newcommand{\step}{\textbf{s}}
\newcommand{\initial}{\textbf{i}}
\newcommand{\terminal}{\textbf{e}}

\newcommand{\initialinterval}{\textbf{i}}
\newcommand{\terminalinterval}{\textbf{e}}

\newcommand{\growth}{\mathcal{V}}
\newcommand{\Cayley}{\mathsf{Cay}}
\newcommand{\equivalent}{\sim}
\newcommand{\rank}{\mathrm{rk}}
\newcommand{\isoperimetry}{\mathsf{IS}}
\newcommand{\boundary}{\partial}
\newcommand{\subgrp}{\leq}

\newcommand{\isoprof}{\Phi}

\begin{document}

\title{Random interlacement is a factor of i.i.d.}
\author{
	 M\'arton Borb\'enyi\footnote{
		E\"otv\"os Lor\'and University, Mathematics Institute, Department of Computer Science, P\'azmány P\'eter s\'et\'any 1/C, H-1117 Budapest, Hungary.
		marton.borbenyi@gmail.com}, \;
	Bal\'azs R\'ath\footnote{
		Department of Stochastics, Institute of Mathematics, Budapest University of Technology and
Economics, M\H{u}egyetem rkp. 3., H-1111 Budapest, Hungary.; 
		MTA-BME Stochastics Research Group, M\H{u}egyetem rkp. 3., H-1111 Budapest, Hungary;
		Alfr\'ed R\'enyi Institute of Mathematics, Re\'altanoda utca 13-15, 1053 Budapest, Hungary.
		rathb@math.bme.hu}, \; 
	S\'andor Rokob\footnote{
		Department of Stochastics, Institute of Mathematics, Budapest University of Technology and
Economics, M\H{u}egyetem rkp. 3., H-1111 Budapest, Hungary.;
		Alfr\'ed R\'enyi Institute of Mathematics, Re\'altanoda utca 13-15, 1053 Budapest, Hungary.
		roksan@math.bme.hu }
}

\maketitle
	
\begin{abstract} The random interlacement point process (introduced in \cite{S10}, generalized in \cite{T09}) is a Poisson point process on the space of labeled doubly infinite nearest neighbour trajectories modulo time-shift on a transient graph $G$. We show that the random interlacement point process on any transient transitive graph $G$ is a factor of i.i.d., i.e., it can be constructed from a family of i.i.d.\  random variables indexed by vertices of the graph via an equivariant measurable map.
Our proof uses a variant of the soft local time method (introduced in \cite{PT15}) to construct the interlacement point process as the almost sure limit of a sequence of finite-length variants of the model with increasing length.  We also discuss a more direct method of proving that the interlacement point process is a factor of i.i.d.\ which works if and only if $G$ is non-unimodular.

		
\end{abstract}

\section{Introduction}

	
\subsection{Random interlacements}

Random interlacements, introduced in \cite{S10}, describe the local distributional limit of the trace of a random walk on a $d$-dimensional discrete torus $\left( \mathbb{Z} / N \mathbb{Z} \right)^d, d \geq 3$ if we run the random walk up to times comparable to the volume of the torus and let $N \to \infty$, cf.\ \cite{W08}. The notion of random interlacements was generalized to transient weighted graphs in \cite{T09}.
	
Let us give a brief  description of the random interlacement point process, deferring the technical details to Section \ref{subsection:pointmeasure}. Let us denote by $W$ the space of doubly infinite transient nearest neighbour trajectories in $G$.
We say that $w,w' \in W$ are equivalent modulo time-shift 
 if there exists  $k \in \mathbb{Z}$ such that for all $n \in \mathbb{Z}$ we have
	$w(n)=w'(n+k)$. Let us denote by $W^*$ the set of equivalence classes of $W$
with respect to time-shift equivalence. 
The random interlacement point process $\mathcal{Z}=\sum_{i \in I}\delta_{(w^*_i,t_i)}  $ is a Poisson point process (PPP) on  the space $W^* \times \R_+$ of labeled trajectories modulo time-shift with intensity measure $\nu \times \lambda$,  where $\lambda$ denotes the Lebesgue measure on $\R_+$ and
$\nu$ is a $\sigma$-finite measure on $W^*$ that we will precisely define in Section \ref{subsection:pointmeasure}. However, note that the following property characterizes $\nu$: for each finite subset $K$ of the vertex set of $G$, an alternative way of generating a PPP on $W^*$ with the same distribution as the point process of trajectories of $\mathcal{Z}$ that hit $K$ and have a label in the interval $[0,u]$ is as follows (cf.\ Theorem \ref{thm_nu_exists_unique}): independently for each vertex $v$ of $K$, let us start a $\mathrm{POI}(u)$ number of i.i.d.\ doubly infinite random walks from $v$ indexed by $\mathbb{Z}$,  throw away those trajectories
that already visit $K$ at a time indexed by a negative number and take the point process that consists of  the equivalence classes of the remaining trajectories modulo time-shift.

The goal of our paper is to construct the interlacement Poisson point process $\mathcal{Z}$ from a family of i.i.d.\ random variables indexed by the vertex set of $G$ via a measurable map which intertwines the action of the automorphism group $\Gamma$ of $G$. Let us now provide the precise formulation of this this property.

\subsection{Factor of i.i.d.\ property}

If we are given a group $\Gamma$ acting on two sets $\Omega_1$ and $\Omega_2$ then a map $T \, : \, \Omega_1 \rightarrow \Omega_2$ is called $\Gamma$-equivariant if it intertwines the actions of $\Gamma$, i.e., if $T(\varphi(\omega_1) )=\varphi(T(\omega_1))$ holds for any $\omega_1 \in \Omega_1$ and $\varphi \in \Gamma$. In the case when $(\Omega_i, \mathcal{A}_i)$, $i = 1, 2$ are measurable spaces, then a $\Gamma$-equivariant measurable map is called a $\Gamma$-factor. If there is also a measure $\mu$ given on the domain space $(\Omega_1, \mathcal{A}_1)$, then the push-forward measure is called $\Gamma$-factor of $\mu$.

We focus on the case when the domain space $\Omega_1$ is a product space of the form $( \Omega^V, \mathcal{A}^V)$, where $(\Omega, \mathcal{A})$ is a measurable space and $V$ is the (countable) vertex set
of a graph $G$. In our case the group $\Gamma$ is the automorphism group $\mathrm{Aut}(G)$ of the simple graph $G$ with vertex set $V$ and edge set $E$, i.e., $\varphi \in \Gamma$ if and only if $\varphi: V \to V$ is a permutation with the property that  $\{x,y\} \in E$ if and only if $\{\varphi(x), \varphi(y) \} \in E$ for any $x \neq y \in V$.
 Note that in this case $\Gamma$ acts on the product space $\Omega^V$ as
 $\varphi (\underline{\eta}) = ( \eta_{ \varphi^{-1}(x) } )_{x \in X}$, where $\underline{\eta} = ( \eta_{x} )_{x \in V} \in \Omega^V$ and $\varphi \in \Gamma$. We assume that $G$ is transitive, i.e., for any $x,y \in V$ there exists $\varphi \in \Gamma$ such that $\varphi(x)=y$.
 
 In our case the target space $\Omega_2$ is the space $\wormspm(W^* \times \mathbb{R}_+)$ of locally finite point measures on $W^* \times \mathbb{R}_+$ (see Definition \ref{def_space_of_worm_point_measures} for details) and the action of $\Gamma$ extends naturally to  $\wormspm(W^* \times \mathbb{R}_+)$ (see
 Definition \ref{def_auto} for details).

If the probability measure $\mu$ on the domain space is a product measure on $( \Omega^V, \mathcal{A}^V)$, i.e., if $ ( \eta_{x} )_{x \in V}$ are i.i.d.\ then the corresponding factor is called a factor of i.i.d.\ (or f.i.i.d.\ for short).

\subsection{Statements of  results}

Recall that we denote $\Gamma=\mathrm{Aut}(G)$.

\begin{theorem}[Main result]\label{main_thm_intro} Let $G$ denote a locally finite, connected, transitive, transient infinite simple graph.
 There exists a probability space $(\Omega, \mathcal{A}, \vartheta)$ and a measurable map $T \, : \, \Omega^V \rightarrow \wormspm ( W^* \times \mathbb{R}_+ )$  with the following properties.
\begin{enumerate}[(i)]
	\item\label{main_thm_i_intro} If $\underline{\eta}=( \eta_x )_{x\in V}$ are i.i.d.\ with distribution $\vartheta$ then $T ( \underline{\eta} )$ is a PPP on $W^* \times \mathbb{R}_+$ with intensity measure $\nu \times \lambda$.
	\item\label{main_thm_ii_intro} For any  $\varphi\in\Gamma$ we have $T ( \varphi( \underline{\eta} ) ) = \varphi\left( T ( \underline{\eta} ) \right)$.
\end{enumerate}
\end{theorem}
\noindent
In words:  the law of the random interlacement point process $\mathcal{Z}$  on any  locally finite, connected, transitive, transient infinite simple graph $G$ is a factor of i.i.d.

Our proof of Theorem \ref{main_thm_intro} uses an approximation of $\mathcal{Z}$ with a homogeneous PPP of random walk trajectories of length $T$ that we call finite-length interlacements (cf.\ Definition \ref{def_PPP_of_labelled_worms}).
This notion is inspired by that of \cite{B19}, where
 a homogeneous PPP of random walk trajectories with geometric length distribution called finitary random interlacements is introduced. As it turns out, we found the variant of the model with trajectories of fixed length to be more suitable for our purposes. One ingredient of our proof is that finite-length interlacements converge in distribution to random interlacements as $T \to \infty$ 
(cf.\ Lemma \ref{lemma:worms_converge_in_distribution}). Let us note that
similar approximation results have already appeared in the literature, cf.\ 
 \cite[Chapter 4.5]{S12}, \cite[Theorem 3.1]{DRS14}, \cite[Proposition 3.3]{H18},
 \cite[Theorem A.2]{B19}. 
  Let us also note  that the factor of i.i.d\ property is not necessarily inherited by a distributional limit (see e.g.\ Corollary 3.3 of \cite{L17}), hence  Theorem \ref{thm:main} does not follow automatically from Lemma \ref{lemma:worms_converge_in_distribution} and the fact that finite-length interlacements
  is  factor of i.i.d.\ for each $T \in \mathbb{N}_+$.

 The novelty of our paper is that we boost the above-mentioned
 result about convergence in distribution and show that there exists a jointly equivariant realization of finite-length interlacements $\mathcal{Z}^n, n \in \N$ with lengths $T_n=2^n, n \in \N$ on the same probability space which converges almost surely with respect to an appropriate topology (cf.\ Section \ref{subcetion_topology}) on the space of labeled nearest neighbour trajectories.
In order to do so, we employ a variant of the \emph{soft local time method}, which has found many applications in the development of the theory of random interlacements (see e.g.\ \cite{CPW16, CT16, S17}) since its introduction in \cite{PT15}.

Informally, we construct a (partial) matching of the trajectories of finite-length interlacements of length $T$ and ``stitch together' the matched pairs to obtain a point process which is  ``close'' to being finite-length interlacements of length $2T$ (see the introduction of  Section \ref{section_matching_softlocal} for a more detailed description). This almost sure convergence result
is the key to our proof of Theorem \ref{main_thm_intro}.

Our proof of Theorem \ref{main_thm_intro} is somewhat involved and one may wish for a more direct proof. As it turns out, a certain type of direct proof works if and only if the transitive graph $G$ is not unimodular. In order to make this statement precise, we need to introduce some definitions.

A function $f \, : \, V \times V \rightarrow [0, \infty)$ is called a mass transport function if it is invariant under the diagonal action of $\Gamma$, i.e., if  we have  $f(x,y)=f(\varphi (x),\varphi (y))$ for any pair of vertices $x,y \in V$ and for any $\varphi\in\Gamma $.
A transitive graph $G$  is called unimodular if it satisfies the mass-transport principle, i.e., for any vertex $o \in V$ and any mass transport function $f \, : \, V \times V \rightarrow [0, \infty)$  
we have
\begin{equation}
	\label{mass_transport_principle_transitive}
	\sum_{x \in V} f(o,x) = \sum_{x \in V} f(x,o).
\end{equation}
One can think of $f(x,y)$ as an amount of mass that is sent from $x$ to $y$, in which case \eqref{mass_transport_principle_transitive} is just a formal way to state that mass is conserved. 

Let us denote by $\mathcal{W}$ the $\sigma$-algebra on $W$ generated by the coordinate maps.
Note that the action of $\Gamma$ on $W$   extends to the $\sigma$-algebra 
$\mathcal{W}$  in a natural way.
Let us denote by $\pi^*: W \to W^*$ the function which maps to each element of $W$ its equivalence class with respect to time shift equivalence. Let us introduce the natural $\sigma$-algebra $\mathcal{W}^* = \left\{ A \subseteq W^* \, : \, (\pi^*)^{-1}(A) \in \mathcal{W} \right\}$ on $W^*$.

\begin{claim}[A sufficient condition for the interlacement to be a factor of i.i.d.\ in a cheap way]\label{claim_QW_exists_then_fiid_triv}
 If there exists a  measure $Q$ on $(W,\mathcal{W})$ such that
 \begin{enumerate}[(i)]
\item\label{QW_projects_to_nu}  $\nu(A) =Q((\pi^*)^{-1}(A))$ for all $A \in \mathcal{W}^*$ and
\item\label{QW_invariant}  $Q(B)=Q(\varphi(B))$ for all $B \in \mathcal{W}$ and all $\varphi \in \Gamma$
\end{enumerate}
 then 
 \begin{enumerate}[(a)]
\item\label{triv_claim_a} the PPP $\mathcal{Z}^W= \sum_{i \in I} \delta_{(w_i,t_i)} $ on $W \times \mathbb{R}_+$ with intensity measure $Q \times \lambda  $ is a f.i.i.d.,
\item \label{triv_claim_b} $\mathcal{Z}:=  \sum_{i \in I} \delta_{(\pi^*(w_i),t_i)} $ is a PPP with intensity measure $\nu \times \lambda$ and $\mathcal{Z}$ is also a f.i.i.d.
 \end{enumerate}
\end{claim}
\noindent Claim \ref{claim_QW_exists_then_fiid_triv} states that the conclusion of our main result  ``trivially'' holds if there exists a PPP $\mathcal{X}=\sum_{i\in I} \delta_{w_i}$ on $W$ such that the law of  $\mathcal{X}$ is invariant under the action of $\Gamma$ and $\pi^*(\mathcal{X}):= \sum_{i\in I} \delta_{\pi^*(w_i)}$ is a PPP with intensity measure $\nu$.

We will present the (rather short) proof of Claim \ref{claim_QW_exists_then_fiid_triv} in Section \ref{section_unimodular_proofs}. 

\begin{proposition}[On the role of unimodularity]\label{prop_QW_unimod} Assume that $G$ is locally finite, connected, transitive and transient. Then the following conditions are equivalent:
\begin{enumerate}[(A)]
    \item\label{unimod_A}  If $G$ is unimodular.
    \item\label{unimod_B}  There is no measure $Q$ on  $(W,\mathcal{W})$ that satisfies properties \eqref{QW_projects_to_nu}  and \eqref{QW_invariant} of Claim \ref{claim_QW_exists_then_fiid_triv}.
\end{enumerate}
\end{proposition}

We will prove Proposition \ref{prop_QW_unimod} in Section \ref{section_unimodular_proofs}. The idea of the proof that \eqref{unimod_B} implies \eqref{unimod_A} was suggested to us by \'Ad\'am T\'im\'ar. 

In \cite[Remark 1.2]{S10} it is proved that
 \eqref{unimod_B}  holds in the special case when $G=\mathbb{Z}^d, d \geq 3$, however, that argument used that if $G=\mathbb{Z}^d$ then $\mathrm{cap}(K)/|K|$ can be made arbitrarily small (e.g.\ by choosing $K$ to be a ball with big radius). This argument does not generalize to the  setting of Proposition \ref{prop_QW_unimod}, since e.g.\ if $G$ is a $d$-regular infinite tree $\mathbb{T}_d$  then one can show that $\inf_{K \fsubset V} \mathrm{cap}(K)/|K|>0$. However, note that  $\mathbb{T}_d$ is unimodular (cf.\ \cite[Exercise 8.7]{LP16}), thus Proposition \ref{prop_QW_unimod} can be applied to conclude that \eqref{unimod_B} holds for $\mathbb{T}_d$.

\subsection{Related literature}

\subsubsection{Bird's eye view}\label{subsub_birdseye}

The question whether a stationary stochastic process on $\mathbb{Z}$ is a factor of another one traces back to the seminal work of Ornstein \cite{O70} (see \cite{O77} for a more detailed explanation and some related results), who  answered the question: when is an i.i.d.\ process isomorphic to (i.e., an invertible factor of) an other i.i.d.\ process?
Although there are some natural extensions of this result even on $\Z$, such as a description of the existence of similar isomorphisms under more constraints on the factor maps like in \cite{KS77, KS79}, or a construction of a Markov chain as a factor of i.i.d.\ as in \cite{AJR79, R82, AS21}, the research focusing on proving which of the well-known random fields (indexed by more general graphs) arise as a factor of i.i.d.\  took place immediately.

Obviously, the presence (or in some cases the absence) of the factor of i.i.d.\ property depends on the underlying graph as well as the distribution of the random field. As a consequence, we only discuss those  results in detail that are most relevant from the point of view of our results.
However, without providing a fully exhaustive list, some of the examined models of statistical physics are: the Ising model on $\Z^d$ \cite{BS99, MS22, RS22B} and on more general graphs \cite{A92, L17, NSZ22, HS22}; the Potts model on $\Z^d$ \cite{HS00, HS22, ST19, S20B, RS22B} and on general graphs \cite{HJL02, HS22}; the proper $q$-coloring of $\Z^d$ \cite{S20A, RS22A, S20B}; the hard-core model on $\Z^d$ \cite{S20B}; the six-vertex model \cite{RS22B}; the Widom-Rowlinson lattice gas on $\Z^d$ \cite{HS00, S20B}; the Voter model on $\Z^d$ \cite{ST19, SZ22}; and the uniform spanning forests on random rooted almost surely transient graphs \cite{ARS21}.

Let us also note  that many of the papers cited above ask about the presence of a stronger property, i.e.,  when can the examined model be constructed as a finitary factor of i.i.d.
In this strengthening, one also requires that the output variable at a vertex is calculated by looking at almost surely finitely many input variables (where the random number of input variables that need to be inspected can depend on the location of the output variable).
Due to the algorithmic nature of these constructions, the theory of finitary factor of i.i.d.\ processes is a very active area of research in theoretical computer science. A comprehensive list of references of such results can be found in \cite{L17}.

\subsubsection{Generalized divide and color model, voter model}

One model of particular interest for us is a random field called the generalized divide and color model, introduced in \cite{ST19}.
In this model, the vertex set of a graph is partitioned into subsets by a random equivalence relation and then each equivalence class of vertices is given a random color independently of the others.
This general model includes e.g.\ the Ising and Potts models as a special case via the random cluster representation and the extremal shift-invariant stationary distributions of the voter model 
via the coalescing random walk representation.

The authors argue that on $\Z^d$, if the law of the partition is invariant under the translations of $\Z^d$ and  almost surely produces partition sets of finite cardinality, then certain ergodic theoretic properties such as the factor of i.i.d.\ property is inherited from the random equivalence relation to the generalized divide and color model.

Question 7.21 of \cite{ST19} asks whether the extremal stationary distributions of the voter model on  $\Z^d$ ($d \geq 3$)  are factors of i.i.d. This question is affirmatively answered in \cite{SZ22}.
Let us mention that the spine of their argument, that is, producing the stationary distribution in question as an almost sure limit of i.i.d.\ factors using a well-behaved coupling was highly influential for us.

\subsubsection{Ising model on regular trees}

Let us point out that there are some negative results among those proved in the papers on the list given in Section \ref{subsub_birdseye}.
One example of this is the free Ising model on the $d$-regular tree, which cannot be represented as a factor of i.i.d.\ if the inverse temperature $\beta$ satisfies $\tanh \beta > (d-1)^{-1 / 2}$, cf.\ \cite[Corollary 3.2]{L17}.
It is believed that  this result is sharp, i.e., if $\tanh \beta \leq (d-1)^{-1/2}$ then the model is a factor of i.i.d.
Currently this is only proved if $\tanh \beta \leq c(d-1)^{-1/2}$, where $c > 0$ is an absolute constant and $d$ is large, cf.\ \cite{NSZ22}.

\subsubsection{An application: interlacement Aldous-Border algorithm and WUSF}

There are multiple ways of proving that the law of the wired uniform spanning forest (WUSF) on
on a transitive transient graph $G$  is a  factor of i.i.d.:  the case of amenable Cayley graphs follows from \cite[Corollary 7.4]{LT16}, the nonamenable case is part
of the proof of \cite[Proposition 9]{GL09}, the general case of transient transitive graphs is implicit in the proof of \cite[Proposition 5.3]{BLPS01}, and \cite[Theorem 1.4]{ARS21} proves the result in a more general setting.

Our main result can be used to provide yet another proof of this result.
The proof that we propose builds on the generalization of the so-called Aldous-Broder algorithm \cite{A90, B89} to transient graphs, introduced in \cite{H18} under the name of interlacement Aldous-Broder algorithm.

The input of this algorithm is a random interlacement point process $\mathcal{Z}$ (i.e., a PPP on $W^* \times \mathbb{R}_+$ with intensity measure $\nu \times \lambda$) on $G$. The  WUSF is generated by keeping at every vertex only the first entry edge of the random interlacement trajectory with the smallest label. 
In our Theorem  \ref{main_thm_intro} we construct $\mathcal{Z}$ as a factor of i.i.d.\  and the interlacement Aldous-Border algorithm produces a WUSF
from $\mathcal{Z}$ in an equivariant fashion, thus the composition of these two constructions provide an alternative way to show that the WUSF is a factor of i.i.d.

\subsubsection{Random interlacements and amenability}
If $G$ is a locally finite, connected, transitive, transient graph and $\mathcal{Z}=\sum_{i \in I}\delta_{(w^*_i,t_i)}  $ is a PPP on  $W^* \times \R_+$ with intensity measure $\nu \times \lambda$, let us denote by $\mathcal{I}^u$ the set of vertices visited by the trajectories $w^*_i$ satisfying $t_i \leq u$. The random set $\mathcal{I}^u$ of vertices is called the interlacement set at level $u$.  The main result of \cite{TT13} states that  $\mathcal{I}^u$ is almost surely connected for all $u >0$ if and only if $G$ is amenable. This result provides a characterization of amenability using interlacements, somewhat similarly to our Proposition \ref{prop_QW_unimod}, which provides a characterization of unimodularity  using interlacements. Let us note here that every locally finite, connected, transitive, amenable graph is also unimodular by
\cite[Proposition 8.14]{LP16}.

\section{Setup and notation}

In Section \ref{subsec_basic} we  introduce some basic notation. 
In Section \ref{subsec_transitivity} we fix our notation regarding graphs and graph automorphisms.
In Section \ref{sebsection:wstar} we introduce various spaces of nearest neighbour trajectories on $G$. In Section \ref{subsection_rw} we introduce our notation pertaining to random walks and state an important heat kernel estimate (the proof of which is deferred to the Appendix).
In Section \ref{subsection:pointmeasure} we introduce our notation pertaining to PPPs on various spaces of (labeled) nearest neighbour trajectories and define finite-length random interlacements with length $T$ as well as  classical (infinite-length) random interlacements.

\subsection{Basic notation}\label{subsec_basic}

Let $\lambda$ denote the Lebesgue measure on $\mathbb{R}_+$ equipped with the $\sigma$-algebra $\mathcal{B}(\mathbb{R}_+)$ of Borel sets.

Let $\mathbb{N}:=\{0,1,2,\dots\}$ and $\mathbb{N}_+:=\{1,2,\dots\}$. Let $\mathbb{Z}$ denote the set of integers.

If $n \in \mathbb{N}_+$, let $[n]:=\{1,\dots,n\}$.

If $a\in \mathbb{R}$, let us denote by $a_+:= \max\{a,0\}$ the positive part of $a$. 

If $a,b \in \mathbb{R}$, let $a \wedge b:=\min\{a,b\}$ and $a \vee b := \max\{a,b\}$. Note that 
$a- a \wedge b= (a-b)_+$.

If  $(Z, \mathcal{Z})$ is a  measurable space and $A \in \mathcal{Z}$, let us denote by $\mathds{1}[A]$ the indicator of $A$, i.e., 
the function from $Z$ to $\{0,1\}$ which assigns $1$ to elements of $A$ and $0$ to elements of $Z \setminus A$.
If $\mu$ is a measure on $(Z, \mathcal{Z})$,
let us define
the measure $ \mu \, \ind[A]$ on  $(Z, \mathcal{Z})$ by
\begin{equation}
    \label{nota:restricted_measure}
    (\mu \, \ind[A])(B)  :=  \mu(A \cap B), \qquad B \in \mathcal{Z}.
\end{equation}

\subsection{Graphs, transitivity}
\label{subsec_transitivity}

Let $G=(V,E)$ denote an undirected locally finite infinite graph. Let $o$ denote a fixed vertex of the graph, ``the origin''. Since we focus on  graphs that are also transitive and transient, let us recall these notions.

For $x, y \in V$, we denote by $\{x,y\} \in E$ if there is an edge between $x$ and $y$. We say that a bijection $\varphi: V \to V$ is a graph automorphism if
\begin{equation*}
    \{x,y\} \in E \iff \{\varphi(x),\varphi(y)\} \in E.
\end{equation*}
The graph $G$ is called (vertex-)transitive, if the group $\Gamma$ of graph automorphisms of $G$ acts transitively on the vertex set of $G$. Intuitively this means that the graph $G$ looks the same from all of its vertices. As a consequence, such $G$ must be regular, so let us denote its degree by $d$.

We call a graph $G$ transient if the simple random walk on $G$ is transient.

We denote the fact $K$ is a finite subset of $V$ by $K \subset \subset V$.

 Distances with respect to the usual graph metric will be denoted by $\distance(.,.)$.

\subsection{Spaces of trajectories}
\label{sebsection:wstar}
Let $\mathcal{H}$ denote the set of finite or infinite sub-intervals $H$ of $\mathbb{Z}$. 
For any $H \in \mathcal{H}$, let
\begin{align}
\label{transient_nearest_neighb_traj}
W_H :=  \left\{ \;
w: H \to V \, : \,
\parbox{17em}{
	$\forall\ n, n+1 \in H \quad  \{ w(n), w(n+1) \} \in E; \; \\
	\forall \, x \in V \quad \sum_{n \in H} \mathds{1}[w(n)=x]<+\infty $
}
\; \right\}
\end{align}
denote the space of nearest neighbour trajectories indexed by $H$ which visit every vertex $x$ of $G$ only finitely many times. Let us introduce the shorthand notation
\begin{align}
    \label{def_eq_W_Wplus_Wminus}
    & W:=W_{(-\infty,\infty)}; \qquad W_+:= W_{[0,+\infty)}; \qquad W_-:=W_{( -\infty , 0  ] }; \\  \label{def_eq_WT_frormspace}
    & \qquad W_T:=W_{[0,T-1]}, \; T \in \mathbb{N}_+; \qquad \fwormspace:=  \bigcup_{T=1}^{\infty} W_T.
\end{align}

\begin{definition}[Time shift equivalence]\label{time_shift_equivalence}
	Let $\sim$ denote the following equivalence relation on $W$. The trajectories
	$w,w' \in W$  are equivalent if there exists  $k \in \mathbb{Z}$ such that for all $n \in \mathbb{Z}$ we have
	$w(n)=w'(n+k)$, i.e., $w'$ can be obtained from $w$ by a time shift.
	The quotient space $W/\sim$ is denoted by $W^*$.
\end{definition}

\noindent We write
\begin{equation}
    \label{def:canonical_projection}
    \pi^*: W \to W^*
\end{equation}
for the  projection which assigns to a trajectory $w \in W$ its $\sim$-equivalence class $\pi^*(w)\in W^*$.

All of the spaces above can be endowed with a natural $\sigma$-algebra. Indeed, for example in the case of $W$ or $W_+$ we can simply define the $\sigma$-algebras $\mathcal{W}$ or $\mathcal{W}_+$, respectively to be the one generated by the canonical coordinate maps.
Furthermore, using the former we can define a $\sigma$-algebra $\mathcal{W}^*$ in $W^*$ given by the preimages of the map $\pi^*$, i.e., $\mathcal{W}^* = \left\{ A \subseteq W^* \, : \, (\pi^*)^{-1}(A) \in \mathcal{W} \right\}$.

Let us define $\wormspace$ to be the (disjoint) union of the spaces $ W^*$,   $ \fwormspace$,  $W_+$, and $W_-$:
\begin{equation}
\label{wormspaces}
\wormspace:= W^* \cup  \fwormspace  \cup  W_+ \cup W_-
\end{equation}
Let us define the natural sigma-algebra $\wormspacesigmaalgebra$ on $\wormspace$ as follows:
\begin{equation*}
    \wormspacesigmaalgebra=\left\{\, A \subseteq \wormspace \; : \; 
    A \cap W^* \in \mathcal{W}^*, \;
    A \cap  \fwormspace \in \mathcal{W}^\bullet, \; 
    A \cap W_+ \in \mathcal{W}_+, \;
    A \cap W_- \in \mathcal{W}_- \, \right\}.
\end{equation*}

\begin{remark}
We will focus on local convergence (cf.\ Definition \ref{def:local_conv}) of point measures supported on $\fwormspace \times [0,1]$ to a point measure supported on $W^* \times [0,1]$, but we also included $W_+$ and $W_-$ along with $\fwormspace$ and $W^*$ in the definition of $\wormspace$ in \eqref{wormspaces} in order to make the space of point measures on $\wormspace \times [0,1]$
 complete w.r.t.\ our notion of local convergence
(cf.\ Claim \ref{claim_completeness}).
\end{remark}

\begin{definition}[Trajectories that visit a set $K$]\label{worms_that_hit_K}
	For any $K \subset \subset V$, let us denote by $\wormspace(K)$ the subset of $\wormspace$ which consists of those trajectories that visit the set $K$. Let us define $W_H(K)$ for any $H \in \mathcal{H}$, $\fwormspace(K)$ and  $W^*(K)$  analogously.
\end{definition}

If $K \subset \subset V$,  $H \in \mathcal{H}$ and $w\in W_H(K)$, let us define
\begin{align}\label{def_eq_entrance_time}
H_K(w) & := \min\{\, n \in H \, : \, w(n) \in K\, \}, \quad \text{`first entrance time'} \\
\label{eq:lastExitTime}
L_K(w) & := \max\{\, n \in H\, :\, w(n) \in K \, \}, \quad \text{`time of last visit'}
\end{align}
Note that the reason why we can write $\min$ and $\max$ (rather than $\inf$ and $\sup$) in the above definitions is that by \eqref{transient_nearest_neighb_traj} the trajectory $w$ only spends a finite amount of time in $K$, therefore $H_K(w) >-\infty$ and  $L_K(w)<+\infty$.

If $H \in \mathcal{H}$, let $|H|$ denote the cardinality of $H$.
If $w \in W_H$, let $|w|:=|H|$ denote the length of $w$, that is, $w$ performs $|w| - 1$ steps. In particular, if $w \in W_T$ then $|w|=T$.

\subsection{Random walks}\label{subsection_rw}

\begin{definition}[Random walk]\label{def_rw}  $ $

\begin{enumerate}[(i)]
\item For $x \in V$, let $P_x$ denote the law of the simple random walk $(X(n))_{n \in \N}$ on $G$ which starts from $X(0) = x$, and let $E_x$  denote the corresponding expectation. The law $P_x$ can be considered as a probability measure on the measurable space $(W_+, \mathcal{W}_+)$.

\item\label{def_double_rw}  Denote by $P^{\pm}_x$ the law of a doubly infinite simple random walk
$(X(n))_{n \in \Z}$ on $G$ which satisfies $X(0)=x$.
The law  $P^{\pm}_x$ is a probability measure on the measurable space $(W, \mathcal{W})$.
\end{enumerate}
\end{definition}

Let us denote the transition probabilities of a simple random walk on $G$ by
\begin{equation}
	\label{random_walk_transitions}
	p_n(x,y) := P_x(X(n) = y), \qquad x, y \in V, \, n \in \N.
\end{equation}
By the time-reversibility of simple random walk on $G$, we have
\begin{equation}\label{heat_kernel_symmetric}
  p_n(x,y)=p_n(y,x), \qquad x,y \in V,  \quad n \in \mathbb{N}.
\end{equation}

\begin{lemma}[Heat kernel bound]
	\label{lemma:heat_kernel_upper}
	Let $G=(V,E)$ denote a locally finite, connected, transitive, transient simple graph. There exists a constant $C=C(G)$ such that
	\begin{equation}\label{eq:heat_kernel_upper}
	p_n(x,y) \leq C n^{-3/2}, \qquad x,y \in V.
	\end{equation}
\end{lemma}

\noindent
The result stated in  Lemma \ref{lemma:heat_kernel_upper} is part of the mathematical folklore. However, since we did not find it stated in the published literature, 
 we provide a detailed derivation of the statement of Lemma \ref{lemma:heat_kernel_upper} from known results \cite{CSC93,G81,MP05,T85} in the Appendix.

Given $K \subset \subset V$ and $x \in K$, let us define the equilibrium measure $e_K(x)$ of $x$ with respect to $K$ as well as the capacity $\mathrm{cap}(K)$ of $K$ by
\begin{equation}\label{def_eq_equlibrium_measure}
    e_K(x):= P_x(\, X(n) \notin K, \; n=1,2,3,\dots \,) \qquad \text{and} \qquad \mathrm{cap}(K):=\sum_{x \in K} e_K(x).
\end{equation}
For any $K \subset \subset V$, $x \in K$ and $s=0,1,2,\dots$,  let us also define
\begin{equation}\label{equilib_s}
    e_K^{s}(x):=P_x(\, X(n) \notin K \;\; \text{if} \;\; 0< n \leq  s \,).
\end{equation}
Let us note that for any $K \subset \subset V$ and $x \in K$ we have
\begin{equation}\label{equlib_s_conv}
    e_K^{s}(x) \searrow e_K(x), \qquad s \to \infty. 
\end{equation}

\subsection{Point measures}
\label{subsection:pointmeasure}

\begin{definition}[Point measures]
\begin{enumerate}[(i)] $ $

\item If $(\mathcal{S}, \mathcal{F})$ is a measurable space then a $\sigma$-finite point measure $\omega$ on $(\mathcal{S}, \mathcal{F})$ is measure of form $\omega = \sum_{i \in I} \delta_{s_i}$, where $I$ is a finite or countably infinite index set, $s_i \in \mathcal{S}$ for each $i \in I$ and $\delta_{s_i}$ denotes the Dirac measure concentrated on $s_i$. In other words, for each  $A \in \mathcal{F}$, the measure $\omega(A)$ of the set $A$ with respect to the measure $\omega$ is equal to $\omega(A)=\sum_{i \in I} \mathds{1}[s_i \in A]$.

\item Let us denote by $\pointmeasure(\mathcal{S})$ the set of $\sigma$-finite point measures on $\mathcal{S}$.

\item Let $|\omega|$ denote the total mass of $\omega \in \pointmeasure(\mathcal{S}) $, i.e., $|\omega|=\omega(\mathcal{S})$. 

\item If $\omega, \omega' \in M(\mathcal{S}) $ then we say that $\omega \geq \omega'$ if $\omega-\omega' \in  M(\mathcal{S})$.
\end{enumerate}
\end{definition}

\begin{definition}[Locally finite point measures]
	\label{def_space_of_worm_point_measures}
	\begin{enumerate}[(i)]
	$ $

\item	We say that
	$ \omega \in \wormspm(\wormspace)$ if $
	\omega = \sum_{i \in I} \delta_{\lozengeworm_i} \in \pointmeasure(\wormspace)$
		and for any $x \in V$  the number of indices $i \in I$ for which the trajectory $\lozengeworm_i$ hits $x$  is finite, i.e., $\omega\big( \wormspace(\{x\}) \big)<+\infty$.	
	
\item	We say that $\omega  \in \wormspm(\wormspace \times \mathbb{R}_+)$ if $\omega = \sum_{i \in I} \delta_{(\lozengeworm_i,t_i)} \in \pointmeasure(\wormspace \times \mathbb{R}_+)$ and for any $x \in V$ and any $t \in \mathbb{R}_+$ the number of indices $i \in I$ for which $\lozengeworm_i$ hits $x$ and $t_i \leq t$ holds  is finite, i.e., $\omega\big(  \wormspace(\{ x \}) \times [0,t]   \big)<+\infty$.

\item One defines the spaces of point measures $\wormspm(W_T)$, $\wormspm(\fwormspace)$, $\wormspm(V)$ as follows.  If $\mathcal{S}$ is a countable set,
then $ \omega = \sum_{i \in I} \delta_{s_i} \in \wormspm(\mathcal{S})$
if $\omega \in \pointmeasure(\mathcal{S})$ and $\omega( \{s\} )<+\infty$ for any  $s \in \mathcal{S}$.
\end{enumerate}
\end{definition}

\noindent  One can think about the point measure $\omega \in \wormspm(\wormspace)$ as a multiset of finite or infinite trajectories, where $\omega(\{ \lozengeworm \} )$ denotes the number of copies of the trajectory $\lozengeworm \in \wormspace$ contained in $\omega$. In the case of $\omega \in \wormspm ( \wormspace \times \R_+ )$, the trajectories also have a label attached to them.

\begin{definition}[Automorphisms]\label{def_auto} 
\begin{enumerate}[(i)] $ $

\item\label{auto_i} Let $H \in \mathcal{H}$.
Given a graph automorphism $\varphi \in \Gamma$, we define its action on a trajectory $w = (w(n))_{n \in H} \in W_H$  as $\varphi(w) := (\varphi(w(n)))_{n \in H}$. 
\item 
If $w^* \in W^*$, we define $\varphi(w^*):=\pi^*(\varphi(w)) \in W^* $ for any $w \in (\pi^*)^{-1}(w^*)$, which   is unambiguous, since it does not matter which $w \in (\pi^*)^{-1}(w^*)$ we pick.
\item 
We define the action of $\Gamma$ on any point measure of trajectories naturally, e.g.\ if  $\omega = \sum_{i \in I} \delta_{(\lozengeworm_i,t_i)}  \in \wormspm(\wormspace \times \mathbb{R}_+)$  and $\varphi \in \Gamma$, then $\varphi(\omega):=\sum_{i \in I} \delta_{(\varphi(\lozengeworm_i),t_i)} \in  \wormspm(\wormspace \times \mathbb{R}_+) $.
\end{enumerate}
\end{definition}


\begin{definition}[Counting measure]\label{def_counting_measure} If $\mathcal{S}$ is a countable set, let us denote by $\mu^{\mathcal{S}}$ the counting measure on $\mathcal{S}$.
    \end{definition}

\begin{definition}[Homogeneous PPP on $V$]\label{def_PPP_on_V}
Given some $v \in \R_+$, a random element $\mathcal{R}$ of  $\wormspm(V)$ has law $\mathcal{P}_{v,1}$ if $\mathcal{R}$ is a Poisson point process on $V$ with intensity measure 	$v \cdot \mu^V$.
\end{definition}
\noindent
In words, $\mathcal{R} \sim \mathcal{P}_{v,1}$ if and only if the integer-valued random variables $\mathcal{R}(\{x\}), x \in V$ are i.i.d.\ with $\mathrm{POI}(v)$ distribution. Our next definition generalizes the previous one, since $W_T$ with $T=1$ can be identified with $V$.

\begin{definition}[Finite-length random interlacement with length $T$]
	\label{def_PPP_of_worms}
	Given some $v \in \R_+$ and $T \in \N_+$, a random element $\mathcal{X}$ of  $\wormspm(\wormspace)$ has law $\wormsppp_{v,T}$ if $\mathcal{X}$ is a Poisson point process (PPP) on $\wormspace$ with intensity measure 	$\nu_{v,T}$, where 	$\nu_{v,T}$ is defined as
	\begin{equation}
	\label{intensity_worms_PPP}
	\nu_{v,T} := v \cdot d^{1-T} \cdot \mu^{W_T}.
		\end{equation}
\end{definition}

\begin{claim}[Construction of  $\mathcal{P}_{v,T}$ using random walks]\label{claim_worms_random_walks} A cloud $\mathcal{X}$ of trajectories with distribution $\mathcal{P}_{v,T}$ can be generated as follows. Let   $N_x$ denote the number of trajectories starting from $x$. Then $(N_x)_{x \in V}$ are i.i.d.\  with $\mathrm{POI}(v)$ distribution  and given their starting points, the trajectories are conditionally independent, and they are distributed as the first $T-1$ steps of a simple random walk on $G$.
\end{claim}

\begin{definition}[Finite-length random interlacements with labels]
	\label{def_PPP_of_labelled_worms}
	Given some  $T \in \N_+$, a random element $\mathcal{Z}$ of $\wormspm(\wormspace \times [0,1])$ has law $\mathcal{Q}_{ T}$ if $\mathcal{Z}$ is a Poisson point process on $\wormspace \times [0,1]$ with intensity measure
	\begin{equation}
	\label{intensity_worms_PPP_labelled}
	 \frac{1}{T} \cdot d^{1-T} \cdot \mu^{W_T} \times \lambda \ind \left[\, [0,1]\, \right].
	\end{equation}
\end{definition}

\begin{claim}[Construction of $\mathcal{Q}_{ T}$ from $\mathcal{P}_{1/T,T}$]\label{claim_uniform_labels}  A cloud $\mathcal{Z}$ of trajectories with distribution $\mathcal{Q}_{T}$ can be generated as follows.
Let $\mathcal{X}=\sum_{i \in I} \delta_{w_i} \sim \mathcal{P}_{1/T,T}$. Given $\mathcal{X}$, let $U_i, i \in I$ denote conditionally i.i.d.\ random variables with $\mathrm{UNI}[0,1]$ distribution. Then $\mathcal{Z}=\sum_{i \in I} \delta_{(w_i, U_i)} \sim \mathcal{Q}_{T}$.
\end{claim}

In order to define the interlacement point process on a general (transient, weighted) graph, we need to recall the $\sigma$-finite measure $\nu$ on $W^*$ which was introduced in \cite{T09}. Recall the notion of the law  $P^{\pm}_x$  from
Definition \ref{def_rw}\eqref{def_double_rw}. Our next theorem follows from  \cite[Theorem 2.1]{T09}.

\begin{theorem}[Interlacement intensity measure]\label{thm_nu_exists_unique}
    There exists a unique $\sigma$-finite measure $\nu$ on $(W^*, \mathcal{W}^*)$ such that for every $A \in \mathcal{W}^*$ and every finite $K \fsubset V$ we have
    \begin{equation}
        \label{eq:compatibility_of_RI_intensity_measure}
        \nu \left( A \cap W^*(K) \right) =
        Q_K( (\pi^*)^{-1}(A) ),
    \end{equation}
    where the finite measure $Q_K$ on $W$ is defined by
    \begin{equation}\label{def:Q_k_measures}
        Q_K(B)= \sum_{x \in K} P^{\pm}_x(\,  B, \, H_K=0 \,  ), \qquad
        B \in \mathcal{W}.
    \end{equation}

%
\end{theorem}
\noindent
We extend $\nu$ to $\wormspace$ by defining $\nu(\wormspace \setminus W^*) := 0$.


\begin{definition}[Random interlacement point process]
    \label{def:RI_PP}
\begin{enumerate}[(i)] $ $

\item  The random interlacement point process $\mathcal{Z}$ is a random element of $\wormspm(\wormspace \times \R_+)$ which is a Poisson point process on $\wormspace \times \R_+$ with intensity measure  $ \nu \times \lambda$.
\item Let us denote by $\mathcal{Q}_\infty$ the law of $\mathcal{Z}\mathds{1}[\, \wormspace \times [0,1]\, ]$. Thus $\mathcal{Q}_\infty$ is the law of a PPP on $\wormspm(\wormspace \times [0,1])$ with intensity measure $ \nu \times \lambda \mathds{1}[\, [0,1]\, ] $.
 \end{enumerate}
    \end{definition}

\begin{claim}\label{claim_interlacement_invariant}
The laws of the above defined point processes are invariant under the action of $\Gamma$.
\end{claim}
\begin{proof} It is enough to check that the measures $\nu_{v,T}$ and $\nu$ are invariant under $\Gamma$, and this
directly follows from the definitions of $\nu_{v,T}$ 
(cf.\ \eqref{intensity_worms_PPP}) and  $\nu$ (cf.\ \eqref{eq:compatibility_of_RI_intensity_measure}, \eqref{def:Q_k_measures}).
\end{proof}

\section{One the role of (non)unimodularity}
\label{section_unimodular_proofs}

The goal of Section \ref{section_unimodular_proofs} is to prove Claim \ref{claim_QW_exists_then_fiid_triv} and Proposition \ref{prop_QW_unimod}. 

First we deal with a technical issue: our graph $G$ is transitive, but there might be many automorphisms that map the origin $o$ to vertex $x$. We will use i.i.d.\ $\mathrm{UNI}[0,1]$ random variables  $(U_x)_{x \in V}=\underline{U}$  to simultaneously pick for each $x \in V$ our favourite automorphism $\underline{\varphi}_x^{\underline{U}}$ that maps $o$ to $x$, moreover our construction of such a family of automorphisms will be equivariant in the sense that
$\underline{\varphi}^{\gamma(\underline{U})}_{\gamma(x)}(y)=\gamma( \underline{\varphi}^{\underline{U}}_x(y) ) $ holds for any $\gamma \in \Gamma$ and any $x,y\in V$.

\begin{definition}[Chart at a vertex, atlas] $ $
\begin{enumerate}[(i)]
\item     Let $\overline{G}=(\overline{V},\overline{E})$ denote a fixed copy of the graph $G$ with root $o$. Given some $x \in V$, we say that $\underline{\varphi}:\overline{V}\to V$ is a chart at $x$ if $\underline{\varphi}$ is a graph automorphism and $\underline{\varphi}(o)=x$.
\item We call $\underline{\varphi} = (\underline{\varphi}_x)_{x\in V}$ an atlas if  $\underline{\varphi}_x$ is a chart at $x$ for all $x \in V$. Denote by
$\Phi$ the set of atlases.
\item The  group  $\Gamma$  of graph automorphisms acts on $\Phi$ as follows: if $\gamma \in \Gamma$ and $\underline{\varphi} \in \Phi$  then the atlas $\gamma(\underline{\varphi})$ is  defined by \begin{equation}\label{auto_acts_on_atlas}
(\gamma(\underline{\varphi}))_x=\gamma\circ\underline{\varphi}_{\gamma^{-1}(x)}.
\end{equation}
\end{enumerate}
\end{definition}
\noindent One easily checks that if $\underline{\varphi}$ is an atlas then 
indeed $\gamma(\underline{\varphi})$ is also an atlas.

Our next result states that  we can construct an atlas as a factor of i.i.d.
\begin{lemma}[Factor of i.i.d.\ atlas]\label{lemma_atlas_factor_of_iid}
     There is a measurable function $F :[0,1]^V\to \Phi$ such that if $\underline{U}=(U_x)_{x \in V}$ is a family of i.i.d.\ random variables with $\mathrm{UNI}[0,1]$ distribution then 
    we have 
    \begin{equation}\label{chart_gamma_intertwine}
        \mathbb{P}\big[ \, \forall \, \gamma \in \Gamma  \, : \, 
        F(\gamma(\underline{U}) )=\gamma( F(\underline{U}) ) \,
        \big]=1. 
    \end{equation}
\end{lemma}
\begin{proof} Let $(x_n)_{n \in \mathbb{N}}$ denote a well-ordering of $\overline{V}$ satisfying $x_0=o$. We will construct the atlas
$\underline{\varphi}^{\underline{U}} :=F(\underline{U})$ as follows.
Let us fix $x \in V$. We will recursively construct $\underline{\varphi}^{\underline{U}}_x$ by determining the values $x'_n:=\underline{\varphi}^{\underline{U}}_x(x_n),\, n\in \N $ one by one using induction on $n$. We define $x'_0:=x$.
Assuming that we have already defined $x'_0,\dots, x'_n$ for some $n \in \mathbb{N}$, let us denote 
\begin{equation}\Gamma_n:=\{\, \varphi \in \Gamma \, : \, \varphi(x_i)=x'_i, \, i=0,\dots,n \, \}.
\end{equation}
Our induction hypothesis is that $\Gamma_n \neq \emptyset$. Note that this indeed holds for $n=0$ since $G$ is transitive. Given $x'_0,\dots, x'_n$, note that the orbit $O_{n}:= \Gamma_n x_{n+1}$ of
$x_{n+1}$ under $\Gamma_n$ is finite since $G$ is connected and locally finite. Let 
$x'_{n+1}:=\mathrm{argmin}_{y \in O_{n}} U_y$ denote the vertex with the smallest  label among the possible options. Now we see that 
$\Gamma_{n+1} \neq \emptyset$ and we can continue. If we let $n \to \infty$, we obtain a well-defined graph automorphism $\underline{\varphi}^{\underline{U}}_x$ if we let $\underline{\varphi}^{\underline{U}}_x(x_n):= x'_n ,\, n\in \N $, noting that $\underline{\varphi}^{\underline{U}}_x$ is
 a chart at $x$.
We do this for all $x \in V$ to obtain the atlas $F(\underline{U}):=\underline{\varphi}^{\underline{U}}
=(\underline{\varphi}^{\underline{U}}_x )_{x \in V} $.

This construction works if  $(U_x)_{x \in V}$ are all distinct, and this event has probability $1$.
    Then it is straightforward to check by induction on $n$ that for any $\gamma \in \Gamma$, $x \in V$ and $n \in \mathbb{N}$ we have $\underline{\varphi}^{\gamma(\underline{U})}_{\gamma(x)}(x_n)=\gamma( \underline{\varphi}^{\underline{U}}_x(x_n) ) $. Hence we have 
    $\underline{\varphi}^{\gamma(\underline{U})}_{x}(y)=\gamma( \underline{\varphi}^{\underline{U}}_{\gamma^{-1}(x)}(y) ) $ for any $\gamma \in \Gamma$ and any $x,y \in V$,
        i.e., we have $\underline{\varphi}^{\gamma(\underline{U})}=
    \gamma(\underline{\varphi}^{\underline{U}})$ (c.f.\ 
\eqref{auto_acts_on_atlas}),  i.e., \eqref{chart_gamma_intertwine} holds.
\end{proof}

\begin{proof}[Proof of Claim \ref{claim_QW_exists_then_fiid_triv}]
Let us assume that there is a   measure $Q$ on $(W,\mathcal{W})$ such that
properties  \eqref{QW_projects_to_nu} and \eqref{QW_invariant} of  Claim \ref{claim_QW_exists_then_fiid_triv} both hold. 

We begin with the proof of statement \eqref{triv_claim_a} of Claim \ref{claim_QW_exists_then_fiid_triv}.
For any $x \in V$ let 
\begin{equation*}
W^x:=\{\, w \in W \, : \, w(0)=x  \,\}.\end{equation*}
Let $\mathcal{Z}^x, \, x \in V$ denote independent and identically distributed Poisson point processes on $W\times \mathbb{R}_+$ with intensity measure $Q\mathds{1}[W^o] \times \lambda$.
 Let us denote $\underline{\mathcal{Z}}:= (\mathcal{Z}^x)_{x \in V}$.
Let
$\underline{U}=(U_x)_{x \in V}$ denote a family of i.i.d.\ random variables with $\mathrm{UNI}[0,1]$ distribution. Let us define the atlas $\underline{\varphi}^{\underline{U}}:=F(\underline{U})$, where $F$ is defined in Lemma \ref{lemma_atlas_factor_of_iid}. Let us define the point process
\begin{equation}\label{def_of_Z_W}
    \mathcal{Z}^W=  \mathcal{Z}^W(\underline{U}, \underline{\mathcal{Z}}) := \sum_{x \in V} \underline{\varphi}^{\underline{U}}_x(\mathcal{Z}^x). 
\end{equation}
It follows from  \eqref{QW_invariant} that the intensity measure of $\underline{\varphi}^{\underline{U}}_x(\mathcal{Z}^x)$ is  $Q\mathds{1}[W^x] \times \lambda$, thus 
$ \mathcal{Z}^W$ is a PPP with intensity measure $\sum_{x \in V} Q\mathds{1}[W^x] \times \lambda= Q \times \lambda$. It remains to check  that the output $\mathcal{Z}^W $ depends on the i.i.d.\ input $(\underline{U}, \underline{\mathcal{Z}})$  in an equivariant way. Indeed, for any $\gamma \in \Gamma$, we have
\begin{multline*}
  \mathcal{Z}^W\big(\gamma(\underline{U}), \gamma(\underline{\mathcal{Z}})\big)  =\mathcal{Z}^W\big( \gamma(\underline{U}), (\mathcal{Z}^{\gamma^{-1}(x)})_{x \in V}  \big)\stackrel{\eqref{def_of_Z_W} }{=}
  \sum_{x \in V} \underline{\varphi}^{\gamma(\underline{U})}_x\big(\mathcal{Z}^{\gamma^{-1}(x)}\big)\stackrel{(*)}{=}\\ 
  \sum_{x\in V}   \gamma\Big( \underline{\varphi}^{\underline{U}}_{\gamma^{-1}(x)}\big( \mathcal{Z}^{\gamma^{-1}(x)} \big) \Big) \stackrel{(**)}{=} 
   \sum_{y\in V}   \gamma\Big( \underline{\varphi}^{\underline{U}}_{y }\big( \mathcal{Z}^{y } \big) \Big)=
     \gamma\Big(  \sum_{y\in V}  \underline{\varphi}^{\underline{U}}_{y }\big( \mathcal{Z}^{y } \big) \Big) \stackrel{\eqref{def_of_Z_W}}{=} \gamma\left( \mathcal{Z}^W(\underline{U}, \underline{\mathcal{Z}}) \right),
  \end{multline*}
where in $(*)$ we used Lemma \ref{lemma_atlas_factor_of_iid} and \eqref{auto_acts_on_atlas}, and in $(**)$ we changed the variable of summation from $x$ to $y=\gamma^{-1}(x)$.
The proof of \eqref{triv_claim_a} is complete.

If we define $\mathcal{Z}$ as in statement \eqref{triv_claim_b} of Claim \ref{claim_QW_exists_then_fiid_triv} then the proof of \eqref{triv_claim_b}  follows from statement \eqref{triv_claim_a}, property \eqref{QW_projects_to_nu} of Claim \ref{claim_QW_exists_then_fiid_triv}, the mapping property of Poisson point processes
(cf.\ \cite[Section 5.2]{DRS14}) and the fact that $\pi^*(\varphi(w))=\varphi(\pi^*(w))$ holds for any $w \in W$.
    \end{proof}

\begin{proof}[Proof of Proposition \ref{prop_QW_unimod},
\eqref{unimod_A}$\implies$\eqref{unimod_B}] 
We will prove this implication by contradiction.
    Let us suppose that \eqref{unimod_A} holds (i.e., $G$ is unimodular) and the conclusion \eqref{unimod_B}  is false, i.e., let us assume that there exists  a measure $Q$ on $(W,\mathcal{W})$ that satisfies properties \eqref{QW_projects_to_nu} and \eqref{QW_invariant} of Claim \ref{claim_QW_exists_then_fiid_triv}. Let $\mathcal{X} := \sum_{i \in I} \delta_{w_i}$ denote a PPP on $(W, \mathcal{W})$ with intensity measure $Q$. Let us define $w^*_i:=\pi^*(w_i)$ for any $i \in I$ and let $\mathcal{X}^* :=\pi^*(\mathcal{X})= \sum_{i \in I} \delta_{ w^*_i }$.  Note that it follows from   Claim \ref{claim_QW_exists_then_fiid_triv}\eqref{QW_projects_to_nu} and the mapping property of Poisson point processes (cf.\ \cite[Section 5.2]{DRS14}) that $\mathcal{X}^*$ is a PPP on $(W^*, \mathcal{W}^*)$ with intensity measure $\nu$.
    
     For any $n \in \Z$ let us define the point processes $L_n^{\mathcal{X}} \in \wormspm(V)$ as $L_n^{\mathcal{X}} := \sum_{i \in I} \delta_{w_i(n)}$ and the total local time point process $L^{\mathcal{X}} \in \wormspm(V)$ as 
    \begin{equation}\label{sum_L_n}
        L^{\mathcal{X}} := \sum_{n \in \Z} L_n^{\mathcal{X}}.
    \end{equation}
    Let us also define the map: $L: W^* \to \wormspm(V)$ by  $L(w^*):= \sum_{n \in \mathbb{Z}} \delta_{w(n)}$ for any $w \in (\pi^*)^{-1}(w^*) $, noting that this definition is unambiguous, since it does not matter which $w \in (\pi^*)^{-1}(w^*) $ we pick. Let us also define $L^{\mathcal{X}^*}:=\sum_{i \in I} L(w^*_i) $, and note that we have  
    \begin{equation}\label{LLstar_equal}
    L^{\mathcal{X}}=L^{\mathcal{X}^*}.
    \end{equation}
    We will show that the following equalities hold:
    \begin{align}
    \label{expected_total_local_time_at_origin}  
    \E \left[ L^{\mathcal{X}^*}( \{ o \} ) \right]&=1, \\ 
\label{mass_tr_implies}  
\E \left[ L_n^{\mathcal{X}}( \{o\} ) \right] &=  \E \left[ L_0^{\mathcal{X}}( \{o\} ) \right], \quad n \in \mathbb{Z}.
    \end{align}
Note that \eqref{expected_total_local_time_at_origin} together with \eqref{LLstar_equal} implies   $\E \left[ L^{\mathcal{X}}( \{ o \} ) \right]=1$, while \eqref{mass_tr_implies} together with \eqref{sum_L_n}
implies that $\E \left[ L^{\mathcal{X}}( \{ o \} ) \right]$ is  equal to $0$ or  $+\infty$, and we arrived at a contradiction. Consequently, it remains to prove \eqref{expected_total_local_time_at_origin}  and \eqref{mass_tr_implies}.
   
 We start with the proof of \eqref{expected_total_local_time_at_origin}. 
    Given a random walk started form $o$, let $\tau_o$ denote the number of its visits to $o$. Recalling the definition of the equilibrium measure from \eqref{def_eq_equlibrium_measure}, it  follows from the strong Markov property of the random walk that $\tau_o$ has geometric distribution with parameter $e_{\{ o \}}(o)$. Now we can write
    \begin{equation}
        \E \left[ L^{\mathcal{X}^*}( \{ o \}  ) \right] \stackrel{(*)}{=}
        \E \big[\, | \mathcal{X}^* \left( W^*( \{  o \} ) \right) |\, \big] \cdot \E \left[ \tau_o \right] \stackrel{ \eqref{eq:compatibility_of_RI_intensity_measure}, \eqref{def:Q_k_measures} }{ = }\capacity( \{ o \}  ) \cdot \frac{1}{e_{ \{o\} }(o)} \stackrel{\eqref{def_eq_equlibrium_measure}}{=} 1,
    \end{equation}
    where  in $(*)$ we used the properties of PPPs, \eqref{def:Q_k_measures} and the law of total expectation.
    
    Let us now prove \eqref{mass_tr_implies}. For $n \in \Z$ and $x, y \in V$ let us introduce the mass transport function
    \begin{equation}
        f_n(x,y) := \E \left[ \sum_{i \in I} \ind \left[ w_i(0) = x, \, w_i(n) = y \right] \right].
    \end{equation}
    Note that for each $n \in \mathbb{Z}$, the function $f_n(\cdot,\cdot)$ is invariant under the diagonal action of the group $\Gamma$ of automorphisms of $G$, since the intensity measure $Q$ is invariant under $\Gamma$  by   Claim \ref{claim_QW_exists_then_fiid_triv}\eqref{QW_invariant}, which implies that the law of $\mathcal{X}$ is also invariant under $\Gamma$.
    Consequently,  we obtain 
    \begin{equation}
        \E \left[ L_0^{\mathcal{X}}( \{ o \} ) \right] =
        \sum_{x \in V} f_n(o, x) \stackrel{\eqref{mass_transport_principle_transitive} }{=}
        \sum_{x \in V} f_n(x, o) =
        \E \left[ L_n^{ \mathcal{X}}( \{ o \} ) \right], \qquad n \in \Z.
    \end{equation}
 We are done with the proof of the fact that 
 \eqref{unimod_A}$\implies$\eqref{unimod_B}.
\end{proof}
 
\begin{proof}[Proof of Proposition \ref{prop_QW_unimod},
\eqref{unimod_B}$\implies$\eqref{unimod_A}]
Our goal is to show that if $G$ is locally finite, transitive, transient and not unimodular then
 there is a measure $Q$ on  $(W,\mathcal{W})$ that satisfies properties \eqref{QW_projects_to_nu}  and \eqref{QW_invariant} of Claim \ref{claim_QW_exists_then_fiid_triv}.

Let us denote the stabilizer of $x \in V$ by $S(x):=\{\, \varphi \in \Gamma \, : \, \varphi(x)=x \, \}$. Let us denote by $S(x)y:=\{\, \varphi(y) \, : \, \varphi \in S(x)\, \}$ the orbit of $y \in V$ under the action of $S(x)$.

By \cite[Theorem 8.10]{LP16} there exists a function $\mu: V \to (0,+\infty)$  such that 
\begin{equation}\label{def_of_mu_haar}
    \frac{\mu(x)}{\mu(y)}= \frac{|S(x)y | }{|S(y)x|}, \qquad x,y \in V.
\end{equation}
Let us note that we have
\begin{equation}\label{mu_ratio_automorph}
     \frac{\mu(x)}{\mu(y)}=\frac{\mu(\varphi(x))}{\mu(\varphi(y))} \qquad x,y \in V, \, \varphi \in \Gamma,
\end{equation}
since $S(\varphi(x'))\varphi(y')=\varphi(S(x')y' )$ and thus $|S(\varphi(x'))\varphi(y')|=|S(x')y'|$  holds if $x',y' \in V$, $\varphi \in \Gamma$.

For any $x \in V$, let us denote by $N(x)$ the set of neighbours of $x$. Exercise 8.6 of
 \cite{LP16} and our assumptions on $G$ (locally finite, transitive, non-unimodular) together imply that 
 \begin{equation}\label{mu_not_constant}
\text{for any $o \in V$ there exists $x \in N(o)$ such that $|S(o)x | \neq |S(x)o|$.}
 \end{equation}
Let $H \in \mathcal{H}$. For any $w \in W_H$, let us denote 
$\sup^\mu(w):=\sup_{n \in H}\mu(w(n))$. Let us denote 
$W^b_H:=\{\, w \in W_H \, : \,  \sup^\mu(w)=\mu(w(n)) \text{ for some } n \in H    \,\} $.
For  any $w \in W^b_H$, let us denote 
$ \mathrm{argmax}^\mu(w):= \{ \, n \in H\, : \, \mu(w(n))= \sup^\mu(w)   \, \}  $.
Let 
\begin{equation}
\overline{W}_H:= \{\, w \in W^b_H \, : \, |\mathrm{argmax}^\mu(w)|<+\infty   \,\}, \quad  \overline{W}:= \overline{W}_{(-\infty,\infty)}, \quad
 \overline{W}_+:= \overline{W}_{[0,+\infty)}.
\end{equation}
For any  $w \in \overline{W}$, let $n_0(w):= \min \mathrm{argmax}^\mu(w) $ denote the smallest index $n \in \mathbb{Z}$ for which 
$\mu(w(n))=\sup^\mu(w)$. Let us define $T_0: \overline{W} \to \overline{W}$ by $(T_0(w))(n):=w(n-n_0(w))$ for any $n \in \mathbb{Z}$.
Let $\overline{W}^*:=\pi^*(\overline{W})$. Let us define
 $T^*_0: \overline{W}^* \to \overline{W}$ by  $T^*_0(w^*):=T_0(w)$ for any $w \in (\pi^*)^{-1}(w^*)$, noting that this definition is unambiguous since the output $T_0(w)$ is the same for all $w \in (\pi^*)^{-1}(w^*)$. Let us also note that $n_0(w)=n_0(\varphi(w))$ for any $w \in \overline{W}$ by \eqref{mu_ratio_automorph}. This implies that $T^*_0(\varphi(w^*))=\varphi(T^*_0(w^*))$  holds for any $w^* \in \overline{W}^*$ and $\varphi \in \Gamma$, i.e., $T^*_0$  is $\Gamma$-equivariant.
 
 Let  $\mathcal{X}^*=\sum_{i \in I} \delta_{w^*_i}$ denote  a PPP on $(W^*, \mathcal{W}^*)$ with intensity measure $\nu$. We will show that
 \begin{equation}\label{nonunimod_nu_supp_on_overline_W_star}
     \mathbb{P}\left(\,  \, w^*_i \in  \overline{W}^*, \, i \in I   \,\right)=1,
 \end{equation}
 thus $\mathcal{X}:= T^*_0(\mathcal{X}^* )= \sum_{i \in I} \delta_{T^*_0(w^*_i)} \in \wormspm(W)$ is almost surely well-defined. Note that the
 mapping property of Poisson point processes implies that $\mathcal{X}$ is a PPP on $W$. Let us denote by  $Q$ the intensity measure of $\mathcal{X}$. Property \eqref{QW_projects_to_nu} of Claim \ref{claim_QW_exists_then_fiid_triv} holds by the mapping property and the identity $\pi^*(\mathcal{X})=\mathcal{X}^*$, which follows from the simple observation that $\pi^*(T^*_0(w^*))=w^*$ holds for all $w^* \in \overline{W}^*$. In order to prove Property \eqref{QW_invariant} of Claim \ref{claim_QW_exists_then_fiid_triv}, we only need to check that the law of $\mathcal{X}$ is invariant under the action of any $\Gamma$, but this directly follows from the fact that the same holds for the law of $\mathcal{X}^*$ (cf.\ Claim \ref{claim_interlacement_invariant}) and the observation that $T^*_0$  is $\Gamma$-equivariant.
 
 It remains to prove \eqref{nonunimod_nu_supp_on_overline_W_star}. It is enough to prove $\nu(W^* \setminus \overline{W}^*)=0$. By the definition of $\nu$ (cf.\ Theorem \ref{thm_nu_exists_unique}), it is enough to show that $ P^{\pm}_x\left( \overline{W} \right)=1$ for any $x \in V$.
In fact, it is enough to show that $P_x\left( \overline{W}_+ \right)=1$ for any $x \in V$, since the backward path of a doubly infinite random walk starting from $x$ is a time-reversed random walk, and if the forward path as well as the time-reversed backward path is in  $\overline{W}_+$ then the whole doubly infinite path is in $\overline{W}$. 

In order to prove $P_x\left( \overline{W}_+ \right)=1$, it is enough to show that if $(X(n))_{n \in \mathbb{N}}$ is a simple random walk on $G$ and if we define $Y(n):=\ln(\mu(X(n))), n \in \mathbb{N}$, then $(Y(n))_{n \in \mathbb{N}}$ is a random walk on $\mathbb{R}$ with a negative drift, i.e., a process with i.i.d.\ increments such that the expectation of one increment is strictly negative. The increments of the process $(Y(n))_{n \in \mathbb{N}}$ are indeed i.i.d.\ by the definition of simple random walk on the transitive graph $G$ and \eqref{mu_ratio_automorph}.

It remains to show that the expectation of one increment of $(Y(n))_{n \in \mathbb{N}}$ is negative, i.e.,
\begin{equation}\label{increment_negative_expect}
 \mathbb{E}\big[Y(1)-Y(0)\big]=  E_o\big[\ln(\mu(X(1)))-\ln(\mu(o)) \big] =\frac{1}{d} \sum_{x \in N(o)} \ln\left( \frac{\mu(x)}{\mu(o)} \right)
 <0.
\end{equation}
In order to prove this, we recall from \cite[Corollary 8.8]{LP16} a
 generalization of the mass transport principle \eqref{mass_transport_principle_transitive} which can be applied to any locally finite connected transitive graph $G$ (but unimodularity of $G$ is not required): if $f: V \times V \to [0,+\infty)$ is invariant under the diagonal action of $\Gamma$ (i.e., if $f$ is a mass transport function) then we have
\begin{equation}\label{nonunimod_mass_transport}
    \sum_{x \in V} f(o,x)=\sum_{x \in V} f(x,o)\frac{|S(x)o|}{|S(o)x|}.
\end{equation}
Applying this identity to the function $f(x,y):=\mathds{1}[\, \{x,y \} \in E \, ]$, where $E$ is the edge set of $G$, we obtain $d= \sum_{x \in N(o)} \frac{|S(x)o|}{|S(o)x|}$, thus the desired inequality \eqref{increment_negative_expect} follows:
\begin{equation}
    \frac{1}{d} \sum_{x \in N(o)} \ln\left( \frac{\mu(x)}{\mu(o)} \right) \stackrel{ \eqref{def_of_mu_haar} }{=}
    \frac{1}{d} \sum_{x \in N(o)} \ln\left( \frac{|S(x)o|}{|S(o)x|} \right) \stackrel{(*)}{<} 
    \ln\left( \frac{1}{d} \sum_{x \in N(o)} \frac{|S(x)o|}{|S(o)x|}  \right)=\ln(1)=0,
\end{equation}
where  $(*)$ holds by Jensen's inequality, noting that the inequality is indeed strict by \eqref{mu_not_constant}.
\end{proof}
 
\section{Further notation and auxiliary results}

In Section \ref{subsection_fopm} 
we introduce some notation 
related to the ``shearing'' of finite-length trajectories and prove an upper bound on the probability that a point process on $W_T$ with certain good properties  hits a finite set $K$ of vertices.
In Section \ref{subcetion_topology}
we introduce the topology of local convergence on the space $ \wormspm(\wormspace \times [0,1])$ of labeled trajectories and show that the resulting space is Polish (i.e., separable and completely metrizable). Finally, we prove that the law of the finite-length random interlacements process $\mathcal{Q}_{T}$
       (cf.\ Definition \ref{def_PPP_of_labelled_worms}) weakly converges to the law of the random interlacements process $\mathcal{Q}_\infty$  (cf.\ Definition \ref{def:RI_PP}) as $T \to \infty$ with respect to topology of local convergence.

\subsection{Functions of point measures}\label{subsection_fopm}

Now we introduce some functions on the space of trajectories which can be naturally extended to the case of point measures on the space of trajectories.

Given $T \in \N_+$ let us denote the vertex visited by a trajectory  $w \in  W_T$ in the $n$'th step by
\begin{equation}
    \label{notation:nth_step_visited}
    \step_n(w) := w(n), \qquad  0 \leq n \leq T-1.
\end{equation}
Let us introduce a separate notion for the initial point and the endpoint of $w \in W_T$:
\begin{equation}
    \label{notation:initial_terminal}
    \initial(w) := \step_0(w) = w(0) \qquad \text{ and } \qquad \terminal(w) := \step_{T-1}(w) = w(T-1).
\end{equation}
If $\omega = \sum_{i \in I} \delta_{w_j} \in \wormspm( W_T )$, let  $\initial(\omega)=\sum_{i \in I} \delta_{\initial(w_i)}$ and $\terminal(\omega) = \sum_{i \in I} \delta_{\terminal(w_i)}$ denote the point measure of initial points and endpoints of $\omega$, respectively and $\step_n(\omega) := \sum_{i \in I} \delta_{\step_n(w_i)}$ for general step $0 \leq n \leq T-1$. Note that
$\initial(\omega)$, $\terminal(\omega)$ and $\step_n(\omega)$, $n = 0, 1, \ldots T-1$ are  elements of $\wormspm(V)$.

\medskip

Given some $T \leq T' \in \mathbb{N}$, let us define the maps
$\initialinterval_T: W_{T'} \to W_T$ and $\terminalinterval_T: W_{T'} \to W_T$ by
\begin{equation}\label{initial_terminal_intervals}
\initialinterval_T(w)=(w(0),\dots,w(T-1)), \qquad
\terminalinterval_T(w)=(w(T'-T), \dots, w(T'-1)).
    \end{equation}
In words: $\initialinterval_T(w)$ is the initial sub-trajectory of $w$ of length $T$ and $\terminalinterval_T(w)$ is the terminal sub-trajectory of $w$ of length $T$.
We can extend these notions for any point measure $\omega = \sum_{i \in I} \delta_{w_i} \in \wormspm(W_{T'})$  as follows:
\begin{equation}\label{initial_terminal_interval_pp_def}
\initialinterval_{T}(\omega) := \sum_{i \in I} \delta_{ \initialinterval_{T}(w_i)}, \qquad
 \terminalinterval_{T}(\omega) := \sum_{i \in I} \delta_{ \terminalinterval_{T}(w_i)}.
\end{equation}

\begin{claim}[Shearing the trajectories of $\mathcal{P}_{v, T'}$]
    \label{claim:same_law_after_RW_shifts}
    If $v \in \mathbb{R}_+$, $T \leq T' \in \N_+$ and $\mathcal{X}^{v,T'} \sim \mathcal{P}_{v, T'}$ then the following statements hold.
    \begin{enumerate}[(i)]
        \item\label{nyissz} We have $\initialinterval_T( \mathcal{X}^{v,T'} ) \sim \mathcal{P}_{v, T}$ and $\terminalinterval_T( \mathcal{X}^{v,T'}) \sim \mathcal{P}_{v, T}$.
        \item\label{nyassz}  For any $n = 0, \ldots, T'-1$ we have $\step_n(\mathcal{X}^{v,T'}) \sim \mathcal{P}_{v,1}$.
    \end{enumerate}
\end{claim}

\begin{proof} Recall from Definition \ref{def_PPP_of_worms} that 	$\nu_{v,T'}$ denotes the intensity measure of the PPP  $\mathcal{X}^{v,T'}$.

First we prove \eqref{nyissz}.
For any $w \in W_T$, the number of elements $w'$ of $W_{T'}$ satisfying $\initialinterval_T(w' )=w$ is $d^{T'-T}$,  thus
$\nu_{v,T'}(\, (\initialinterval_T)^{-1}(w) \, )=v \cdot d^{1-T}$ by \eqref{intensity_worms_PPP}. Similarly, we have $\nu_{v,T'}(\, (\terminalinterval_T)^{-1}(w) \, )=v \cdot d^{1-T}$.  The proof of \eqref{nyissz}  is complete by the mapping property of Poisson point processes (cf.\ \cite[Section 5.2]{DRS14}).
The proof of \eqref{nyassz} can be deduced if we apply \eqref{nyissz} twice.
 \end{proof}

\begin{lemma}[Bound on the probability of hitting a set $K$ of vertices]\label{lemma_gen_pp_density_worm_hit_bound}  Let $\beta \in \mathbb{R}_+$, $T \in \N_+$ and $K \subset \subset V$. Let $\mathcal{X}$ denote a random element of $\wormspm(   W_T)$.
Let us assume that for any $x \in V$ the inequality $\mathbb{E}[ \initial(\mathcal{X})(\{x\}) ]\leq \beta$ holds. Let us also assume that if we condition on $\initial(\mathcal{X})$ then
 the trajectories of $\mathcal{X}$  are
 distributed as the first $T-1$ steps of a simple random walk  on $G$ with the points of the point process $\initial(\mathcal{X})$  as starting points.
The probability of the event that
a trajectory from $\mathcal{X}$ hits the set $K$ can be bounded as follows:
\begin{equation}\label{eq:prob_visit_KTv}
\mathbb{P}\left[\, \mathcal{X}(W_T(K)) \neq 0 \, \right] \leq \beta \cdot |K|\cdot T.
\end{equation}
 \end{lemma}
 Before we prove Lemma \ref{lemma_gen_pp_density_worm_hit_bound}, let us state a corollary which follows from it using Claim \ref{claim_worms_random_walks}.
 \begin{corollary}\label{corollary:prob_visit_KTv}  In particular, if $\mathcal{X} \sim \mathcal{P}_{\beta,T}$ (cf.\ Definition \ref{def_PPP_of_worms}) then \eqref{eq:prob_visit_KTv} holds.
\end{corollary}

\begin{proof}[Proof of Lemma \ref{lemma_gen_pp_density_worm_hit_bound}] For any $n=0,\dots,T-1$ let us denote $\mathcal{X}_n:=\step_n(\mathcal{X})$. We have
\begin{multline*}
  \mathbb{P}\left[\, \mathcal{X}(W_T(K)) \neq 0 \, \right]
  \stackrel{(*)}{\leq}
  \mathbb{E} \left[ \mathcal{X}(W_T(K))  \right]
  \leq
     \sum_{n=0}^{T-1} \mathbb{E} \left[ \mathcal{X}_n(K)  \right]      \stackrel{(**)}{\leq}  \sum_{n=0}^{T-1} \sum_{x \in V} \sum_{y \in K}  \beta \cdot p_n(x,y) \stackrel{\eqref{heat_kernel_symmetric}}{=}\\
     \sum_{n=0}^{T-1}  \sum_{y \in K} \sum_{x \in V} \beta \cdot p_n(y,x)=
     \sum_{n=0}^{T-1}  \sum_{y \in K}  \beta =
       \sum_{n=0}^{T-1} \beta \cdot |K|= \beta \cdot |K|\cdot T,
\end{multline*}
  where $(*)$  is Markov's inequality and $(**)$ follows from the assumptions of Lemma \ref{lemma_gen_pp_density_worm_hit_bound}.
\end{proof}


	

\subsection{Topology, convergence, completeness}\label{subcetion_topology}

Recall the notion of $\mathcal{H}$, $W_H(K)$, $ \fwormspace(K)$ and $\wormspace(K)$ from Section \ref{sebsection:wstar}.

\begin{definition}[Localization map]
    \label{projection_map_K} $ $
    
\begin{enumerate}[(i)]    
\item	For any $K \subset \subset V$ and $H \in \mathcal{H}$, we define the map
		$\firstlastvis_K: W_H(K) \to \fwormspace(K)$ as follows.
		\begin{equation}\label{firstlastvisit}
	    \firstlastvis_K (w):=\Big(w\big(H_K(w)+n\big)\Big)_{ 0 \leq n\leq L_K(w)-H_K(w)  }=\Big(w\big(H_K(w)\big), \dots, w\big(L_K(w)\big) \Big).
	\end{equation}
We define
	$\firstlastvis_K: W^* \to \fwormspace(K)$ by letting
	$ \firstlastvis_K (w^*):= \firstlastvis_K (w)$ for any $w\in(\pi^*)^{-1}(w^*)$ (noting that this definition is unambiguous).		
By the above definitions, $\firstlastvis_K(\lozengeworm)$ is defined for any  $\lozengeworm \in \wormspace(K)$ (cf.\ \eqref{def_eq_W_Wplus_Wminus}, \eqref{def_eq_WT_frormspace} and \eqref{wormspaces}), i.e.,
		$\firstlastvis_K: \wormspace(K) \to \fwormspace(K)$ is now well-defined.
We call  $\firstlastvis_K(\lozengeworm)$ the local image of  $\lozengeworm$ on $K$.

\item We  define the local image of a point measure  $\omega = \sum_{i \in I} \delta_{(\lozengeworm_i,t_i)}  \in \wormspm(\wormspace \times [0,1] )$
on a finite set $K \subset \subset V$ by letting
	\begin{equation}
	    \firstlastvis_K (\omega):= \sum_{i \in I} \delta_{(\firstlastvis_K(\lozengeworm_i),t_i )} \, \ind[ \lozengeworm_i \in \wormspace(K) ].
	\end{equation}
	The notion of $\firstlastvis_K (\omega)$ can be defined analogously for any
	$\omega \in \wormspm(\wormspace)$ as well.
\end{enumerate}
\end{definition}

\noindent In words: $ \firstlastvis_K (\lozengeworm)$ is the finite  sub-trajectory of $\lozengeworm$ which starts at the first visit of $\lozengeworm$ to $K$, ends at the last visit of $\lozengeworm$ to $K$, and the indexing of $ \firstlastvis_K (\lozengeworm)$  starts from zero.

If $K \subseteq K' \subset \subset V$ then the following compatibility relation holds for any $ \omega  \in  \wormspm(\wormspace \times [0,1] )$:
\begin{equation}
    \label{eq:compatibility_of_local_image}
    \firstlastvis_K (\omega)=\firstlastvis_K (  \firstlastvis_{K'} (\omega)  ).
\end{equation}

\begin{claim}[Reconstruction from local images]\label{claim_reconstruction}
Any $ \omega  \in  \wormspm(\wormspace \times [0,1])$ can be uniquely reconstructed if we know $ \firstlastvis_K (\omega)$ for all $K \subset \subset V$.
\end{claim}
\begin{proof} It is enough to show that if $\lozengeworm \in \wormspace$ then
$\lozengeworm$ can be uniquely reconstructed by looking at $ \firstlastvis_K (\lozengeworm)$ for all $K \subset \subset V$. More specifically, we will consider an exhaustion of $V$, i.e., an increasing sequence $K_1 \subseteq K_2 \subseteq \dots$ of finite subsets of $V$ such that $\cup_{k=1}^{\infty} K_k =V$, and reconstruct $\lozengeworm$ from $\firstlastvis_{K_1}(\lozengeworm), \firstlastvis_{K_2}(\lozengeworm), \dots$.

 Recalling the definition of $\wormspace$ from \eqref{wormspaces}, it suffices to verify our reconstruction claim
 for elements of $\fwormspace$,    $W_+$,  $W_-$ and  $W^*$ separately.
In all of these cases, it is straightforward to reconstruct $\lozengeworm$ up to time-shift equivalence (cf.\ Definition \ref{time_shift_equivalence}). If  $\lozengeworm \in  \fwormspace \cup W_+ \cup W_-$ then the time-parametrization is determined since either the starting point (if $\lozengeworm \in  \fwormspace \cup W_+$) or the endpoint (if $\lozengeworm \in  W_-$) of the trajectory is indexed by zero. On the other hand, if $\lozengeworm \in  W^*$ then it is enough to reconstruct $\lozengeworm$ up to time-shift equivalence.
\end{proof}

\begin{definition}[Point measure of  labels of copies of a trajectory in the local image]\label{def_aux_meas}
For any $K \subset \subset V$, $w \in \fwormspace(K)$  and $\omega \in \wormspm(\wormspace \times [0,1])$ let us denote by $ \firstlastvis_K^{w} (\omega)$ the   point measure on $[0,1]$ defined by
\begin{equation}
    \firstlastvis_K^{w} (\omega)(A):=   \firstlastvis_K (\omega)(\{w\} \times A), \qquad A \subseteq [0,1].
\end{equation}
\end{definition}
\noindent
The total mass $|\firstlastvis_K^{w}(\omega)|$ of the measure $\firstlastvis_K^{w}(\omega)$  is equal to the number of labeled copies of $w$ in $\firstlastvis_K (\omega)$ with any label. Note that this number is finite by Definition \ref{def_space_of_worm_point_measures}. If $|\firstlastvis_K^{w}(\omega)|=k$ then $\firstlastvis_K^{w}(\omega)$ can be viewed as the multiset of labels of the $k$ copies of $w$ in $\firstlastvis_K (\omega)$.

\begin{definition}[Local pseudometric on the space of point measures]
\label{def:local_metric}
Let $K \subset \subset V$. Let us define the pseudometric $\mathrm{d}_{K}(\cdot,\cdot)$ on
 $\wormspm(\wormspace \times [0,1])$ as follows.  Let $\omega, \omega' \in \wormspm(\wormspace \times [0,1] )$.
 If there exists a $w \in \fwormspace(K)$ for which $|\firstlastvis_K^{w}(\omega)| \neq |\firstlastvis_K^{w}(\omega')|$ then we define
 $\mathrm{d}_{K}(\omega,\omega'):=1$. On the other hand, if
 $|\firstlastvis_K^{w}(\omega)| = |\firstlastvis_K^{w}(\omega')|$ for every $w \in \fwormspace(K)$,
 let us define
 \begin{equation}\label{wass}
   \mathrm{d}_{K}(\omega,\omega'):= \max_{\substack{w \in \fwormspace(K) \\ |\firstlastvis_K^{w}(\omega)|\neq 0 } } 
    \mathrm{d}_{\mathrm{Wass}}\left( \frac{\firstlastvis_K^{w}(\omega)}{|\firstlastvis_K^{w}(\omega)|}, \frac{\firstlastvis_K^{w}(\omega')}{|\firstlastvis_K^{w}(\omega')|} \right),
 \end{equation}
where $ \mathrm{d}_{\mathrm{Wass}}(\cdot,\cdot)$ denotes the $1^{\mathrm{st}}$ Wasserstein distance (also known as earth mover's distance) of probability measures on the real line.
\end{definition}
The reason we can write $\max$ instead of $\sup$ in \eqref{wass} is that by Definition \ref{def_space_of_worm_point_measures}, for any $\omega, \omega' \in \wormspm(\wormspace \times [0,1])$ there are only finitely many trajectories $w \in \fwormspace(K)$ for which either $ |\firstlastvis_K^{w}(\omega)|\neq 0$ or $|\firstlastvis_K^{w}(\omega')|\neq 0 $.
 Note that the expression on the r.h.s.\ of \eqref{wass} is at most $1$, since both probability measures are supported on $[0,1]$. One can easily check that $\mathrm{d}_{K}(\cdot,\cdot)$ is
indeed a pseudometric on $\wormspm(\wormspace \times [0,1])$.

\begin{definition}[Locally Cauchy sequences of point measures]
    \label{def:local_conv}
    Let $\omega_n \in \wormspm(\wormspace \times [0,1]), n \in \mathbb{N}$. We say that the sequence $(\omega_n)_{n \in \mathbb{N}}$ is locally Cauchy if for every $K \fsubset V$ the sequence
    $(\omega_n)_{n \in \mathbb{N}}$ is Cauchy with respect to the pseudometric $d_K(\cdot,\cdot)$.
\end{definition}
Note that by Definition \ref{def:local_metric} the sequence 
$(\omega_n)_{n \in \mathbb{N}}$ is locally Cauchy if and only if for every $K \fsubset V$
there exists an $n_0\in \mathbb{N}_+$ such that for every $w \in \fwormspace(K)$ the number $|\firstlastvis_K^{w}(\omega_n)|$   stays constant if $n \geq n_0$, moreover for any $w \in \fwormspace(K)$ the sequence of measures $\firstlastvis_K^{w} (\omega_n), n \in \mathbb{N}$ weakly converges (since the topology of weak convergence  and the topology induced by the $1^{\mathrm{st}}$ Wasserstein metric are equivalent if we only consider probability measures supported on $[0,1]$).

\begin{claim}[Completeness]\label{claim_completeness}
    If the sequence of point measures $\omega_n \in  \wormspm(\wormspace \times [0,1])$, $n \in \mathbb{N}$ is locally Cauchy then there exists	$\omega \in \wormspm(\wormspace \times [0,1])$ such that for every $K \fsubset V$ we have $\lim_{n \to \infty}\mathrm{d}_{K}(\omega_n,\omega)=0$.
\end{claim}

\begin{proof} Let us assume that $(\omega_n)_{n \in \mathbb{N}}$ is locally Cauchy.
For any $K \fsubset V$ and $w \in \fwormspace(K)$ let us denote by  $\widetilde{\firstlastvis}_K^{w}$ the weak limit of $\firstlastvis_K^{w} (\omega_n)$. Let us define
$\widetilde{\firstlastvis}_K \in \wormspm(\fwormspace(K) \times [0,1]) $ by letting 
\[ \widetilde{\firstlastvis}_K(\{ w \} \times A ):=\widetilde{\firstlastvis}_K^{w}(A), \qquad w \in \fwormspace(K), \qquad A \subseteq [0,1].\] 
Observe that  \eqref{eq:compatibility_of_local_image} holds for $\omega_n$ for each $n$, thus for any
     $K \subseteq K' \subset \subset V$  the  compatibility relation
      $\widetilde{\firstlastvis}_K=\firstlastvis_K (  \widetilde{\firstlastvis}_{K'})$ also holds. Similarly to Claim \ref{claim_reconstruction}, it is easy to check that
      this  implies that there exists a unique $\omega \in \wormspm(\wormspace \times [0,1])$
      such that $\widetilde{\firstlastvis}_K=\firstlastvis_K(\omega)$ for all $K$.
\end{proof}

\begin{claim}[Polish space]
    There is a way to equip $\wormspm(\wormspace \times [0,1])$ with a metric $\widetilde{\mathrm{d}}(\cdot,\cdot)$  so that the resulting metric space is (i) complete and separable, moreover (ii) it is a metrization of the convergence introduced in Definition \ref{def:local_conv}.
\end{claim}
\begin{proof} Let us choose $b_K \in (0,+\infty)$ for each
$K \fsubset V$ so that $\sum_{K \fsubset V } b_K <+\infty$. It is a standard exercise to see that if we define $\widetilde{\mathrm{d}}(\omega,\omega'):= \sum_{K \fsubset V} b_K \cdot \mathrm{d}_K(\omega,\omega')$ then $\widetilde{\mathrm{d}}(\cdot,\cdot)$  is indeed a metric on $\wormspm(\wormspace \times [0,1])$ that satisfies properties (i) and (ii).
\end{proof}

Now that we established a notion of convergence on the space $\wormspm(\wormspace \times [0,1])$, we can talk about weak convergence of probability measures on $\wormspm(\wormspace \times [0,1])$, or more precisely the weak convergence of the law of the finite-length random interlacements process $\mathcal{Q}_{T}$
       (cf.\ Definition \ref{def_PPP_of_labelled_worms}) to the law of the random interlacements process $\mathcal{Q}_\infty$  (cf.\ Definition \ref{def:RI_PP}) as $T \to \infty$.

The idea of the proof the next lemma is similar to the proof of \cite[Theorem A.2]{B19}, \cite[Proposition 3.3]{H18}, or the proof in  \cite[Theorem 3.1]{DRS14}. However, none of these results imply the result of our next lemma, so we include its proof for completeness.
\begin{lemma}[Convergence in law]
    \label{lemma:worms_converge_in_distribution} The sequence of probability measures
     $\mathcal{Q}_{T}, T \in \mathbb{N}$
        weakly converges to  $\mathcal{Q}_\infty$   w.r.t.\ the notion of local convergence introduced in Definition \ref{def:local_conv} as $T \to \infty$.
\end{lemma}

\begin{proof} Let $\mathcal{Z}_T \sim \mathcal{Q}_{ T }$
and $\mathcal{Z}_\infty \sim \mathcal{Q}_\infty$.
By Definition \ref{def:local_conv} it is enough
 to show that for any $K \subset \subset V$  one can couple  $\firstlastvis_K(\mathcal{Z}_T) $ and  $\firstlastvis_K(\mathcal{Z}_\infty)$ in a way that
 \begin{equation}\label{coupling_high_prob_equal}
 \mathbb{P}\left(\firstlastvis_K(\mathcal{Z}_T) \neq \firstlastvis_K(\mathcal{Z}_\infty) \right) \to 0, \qquad T \to \infty.
 \end{equation}
 Let $W^{\circ}(K)$ denote the set of elements $w$ of $\fwormspace(K)$ that also satisfy $\initial(w),\terminal(w) \in K$. Note that for any
 $\lozengeworm \in \wormspace$ we have $\firstlastvis_K(\lozengeworm) \in W^{\circ}(K)$.

Note that both $\firstlastvis_K(\mathcal{Z}_T)$ and  $\firstlastvis_K(\mathcal{Z}_\infty)$ are Poisson point measures on $W^{\circ}(K) \times [0,1]$.
Let us denote by $\mu_T$ and $\mu_\infty$ their respective intensity measures. We have $\mu_T=a_T \times  \lambda \mathds{1}[\, [0,1]\, ]$ and $\mu_\infty=a_\infty \times  \lambda \mathds{1}[\, [0,1]\, ]$, where $a_T$ and $a_\infty$ are both measures on $W^{\circ}(K)$.
In order to describe these measures,
let us pick any $w \in W^{\circ}(K)$ and let us assume that $\initial(w)=x \in K$, $\terminal(w)=y\in K$ and $|w|=\ell \in \mathbb{N}_+$.  It follows from Definitions \ref{def_PPP_of_labelled_worms}, \ref{def:RI_PP}, \ref{projection_map_K} as well as  the Markov property and time-reversibility of simple random walk on $G$ that we have
\begin{align}
a_T(w)&:=	
		\frac{1}{T} \sum_{s=0}^{T-\ell} e_K^s(x) \cdot d^{1-\ell} \cdot e_K^{T-\ell-s}(y),  \\
	a_\infty(w)&:=	e_K(x) \cdot d^{1-\ell} \cdot e_K(y),
\end{align}
where $e_K(\cdot)$ and $e_K^s(\cdot)$ are defined in \eqref{def_eq_equlibrium_measure} and \eqref{equilib_s}, respectively.

Next we note that it follows from \eqref{equlib_s_conv} that  we have $\lim_{T \to \infty}a_T(w) =a_\infty(w)$ for any $w \in W^{\circ}(K) $,
moreover we also have
\begin{equation} \sum_{w \in W^{\circ}(K)} a_T(w)=  \frac{1}{T} \sum_{x \in K}
    \sum_{s=0}^{T-1} e_K^s(x) \to \mathrm{cap}(K)= \sum_{w \in W^{\circ}(K)} a_\infty(w), \qquad T \to \infty,
\end{equation}
where $\mathrm{cap}(K)$ is defined in \eqref{def_eq_equlibrium_measure}.
    From these observations and
 Scheff\'e's lemma we obtain
 \begin{equation}\label{sum_abs_diff_a_goes_to_zero}
 \lim_{T \to \infty} \sum_{w \in W^{\circ}(K)}|a_T(w)-a_\infty(w)| =0 .
 \end{equation}
\noindent
In order to prove \eqref{coupling_high_prob_equal}, we will construct a coupling of
$\firstlastvis_K(\mathcal{Z}_T) $ and  $\firstlastvis_K(\mathcal{Z}_\infty)$ that satisfies
 \begin{equation}\label{coupling_bound_a}
 \mathbb{P}\left(\firstlastvis_K(\mathcal{Z}_T) \neq \firstlastvis_K(\mathcal{Z}_\infty) \right)
 \leq \sum_{w \in W^{\circ}(K)}|a_T(w)-a_\infty(w)|.
 \end{equation}

	Let us define
	\begin{equation}
	W^\circ_1(K):= \{ w \in W^\circ(K) \, : \, a_T(w) <a_\infty(w) \}, \quad  W^\circ_2(K):=  W^\circ(K) \setminus W^\circ_1(K) .
	\end{equation}
Let us now define some Poisson point processes on  $W^\circ(K) \times [0,1]$.
	
\noindent 	Let $\mathcal{Z}^{w}_{\mathrm{min}}$ denote
	a PPP
	with intensity measure $(a_T(w)\wedge a_\infty(w))\mathds{1}[ \{ w \} ] \times \lambda \mathds{1}[\, [0,1]\, ] $,  $w \in 	W^\circ(K) $.

\noindent	Let $\mathcal{Z}^{w}_{\mathrm{diff}}$ denote
	a PPP
	with intensity measure $|a_\infty(w)-a_T(w)| \mathds{1}[ \{ w \} ] \times \lambda \mathds{1}[\, [0,1] \, ] $,  $w \in 	W^\circ(K) $.
	
\noindent Let us assume that all of these Poisson point processes are  independent.

\noindent Now let us define the Poisson point processes
\begin{equation}
    \tilde{Z}_T:= \sum_{w \in 	W^\circ(K) } \mathcal{Z}^{w}_{\mathrm{min}} + \sum_{w \in 	W^\circ_2(K) } \mathcal{Z}^{w}_{\mathrm{diff}}, \qquad
     \tilde{Z}_\infty:= \sum_{w \in 	W^\circ(K) } \mathcal{Z}^{w}_{\mathrm{min}} + \sum_{w \in 	W^\circ_1(K) } \mathcal{Z}^{w}_{\mathrm{diff}}.
\end{equation}

 By the above construction, the PPP $ \tilde{Z}_T$ has the same law as $\firstlastvis_K(\mathcal{Z}_T) $, while the PPP
 $ \tilde{Z}_\infty$ has the same law as $\firstlastvis_K(\mathcal{Z}_\infty)$. Moreover, we have
\begin{align*}
    \mathbb{P}\left(\firstlastvis_K(\mathcal{Z}_T) \neq \firstlastvis_K(\mathcal{Z}_\infty) \right) & =
    \mathbb{P}\left(\tilde{Z}_T \neq \tilde{Z}_\infty  \right) =
    \mathbb{P}\left( \sum_{w \in W^{\circ}(K)} |\mathcal{Z}^{w}_{\mathrm{diff}}| \neq 0 \right) \\ & \leq
    \sum_{w \in W^{\circ}(K)} \mathbb{P}( |\mathcal{Z}^{w}_{\mathrm{diff}}| \geq 1 ) \leq
	\sum_{w \in W^{\circ}(K)}|a_T(w)-a_\infty(w)|.
\end{align*}
This implies \eqref{coupling_bound_a}, which, together with \eqref{sum_abs_diff_a_goes_to_zero}, gives \eqref{coupling_high_prob_equal}. The proof of Lemma \ref{lemma:worms_converge_in_distribution} is complete.
\end{proof}

\section{Main result follows from coupling results}\label{section_main_follows_from_coupling}

In Section \ref{section_main_follows_from_coupling} we show that  Theorem \ref{main_thm_intro} follows from a variant (Theorem \ref{thm:main}) where the labels on the trajectories are restricted to $[0,1]$. We then deduce Theorem \ref{thm:main} from  Lemma \ref{lemma:labeled_coupling}, which states that we can couple a PPP with distribution  $\mathcal{Q}_{T}$ and a PPP with distribution $\mathcal{Q}_{2T}$ with small local error. We then deduce Lemma \ref{lemma:labeled_coupling} from Lemma \ref{lemma:coupling_lemma}, which states that we can couple a PPP with distribution  $\mathcal{P}_{v,T}$ and a PPP with distribution $\mathcal{P}_{v/2,2T}$ with small local error. The result stated in Lemma \ref{lemma:coupling_lemma} will be proved in Sections \ref{section_matching_softlocal} and \ref{section_mating_of_worms}.

\medskip 

 Recall from Definition \ref{def:RI_PP} the notion of
$\mathcal{Q}_\infty$. 
Recall how $\Gamma$ acts on $\wormspm(\wormspace \times \mathbb{R}_+)$ from Definition \ref{def_auto}. 
Recall that if $\underline{\eta} = ( \eta_{x} )_{x \in V} \in \Omega^V$ and $\varphi \in \Gamma$ then we denote
$\varphi (\underline{\eta}) = ( \eta_{ \varphi^{-1}(x) } )_{x \in V}$.

\begin{theorem}[Interlacement with restricted labels is a factor of i.i.d.]\label{thm:main} There exists a probability space $(\overline{\Omega}, \overline{\mathcal{A}}, \overline{\vartheta})$ and a measurable map $\Upsilon \, : \, \overline{\Omega}^V \rightarrow \wormspm ( W^* \times [0,1] )$  with the following properties.
\begin{enumerate}[(i)]
	\item\label{main_thm_i} If $\overline{\underline{\eta}}=( \overline{\eta}_x )_{x\in V}$ are i.i.d.\ with distribution $\overline{\vartheta}$ then $\Upsilon ( \overline{\underline{\eta}} )$ has law $\mathcal{Q}_\infty$.
	\item\label{main_thm_ii} For any graph automorphism $\varphi\in\Gamma$ we have $\Upsilon ( \varphi( \overline{\underline{\eta}} ) ) = \varphi\left( \Upsilon ( \overline{\underline{\eta}} ) \right)$.
\end{enumerate}
\end{theorem}
\noindent
In words: the PPP on $W^* \times [0,1]$ with intensity measure $\nu \times \lambda \mathds{1} [ \, [0,1] \,  ]$ is a factor of i.i.d.\ 

Before we prove Theorem \ref{thm:main}, let us deduce Theorem \ref{main_thm_intro} from it.

\begin{proof}[Proof of Theorem \ref{main_thm_intro}] 
We want to show that the PPP on $W^* \times \mathbb{R}_+$ with intensity measure $\nu \times  \lambda$ (i.e., the random interlacement point process, cf.\ Definition \ref{def:RI_PP}) is a factor of i.i.d. For any $k\in \mathbb{N}_0$, let us define the map $\psi_k: \wormspm(W^* \times [0,1] ) \to \wormspm(W^* \times [k,k+1] ) $ by letting $\psi_k\left( \sum_{i \in I} \delta_{(w_i, t_i))}  \right):=  \sum_{i \in I} \delta_{(w_i, t_i+k))}$. Let $\mathcal{Z}^k, k\in \mathbb{N} $ denote i.i.d.\ point processes with distribution $\mathcal{Q}_\infty$, noting that $(\mathcal{Z}_k)_{k \in \mathbb{N}}$ can be jointly realized as a factor of i.i.d.\
by Theorem \ref{thm:main}. Let $\mathcal{Z}:= \sum_{k=0}^{\infty} \psi_k(\mathcal{Z}^k)$. We constructed $\mathcal{Z}$ as a factor of i.i.d.\ and $\mathcal{Z}$
 is a PPP on $W^* \times \mathbb{R}_+$ with intensity measure $\nu \times  \lambda$.
\end{proof}

\begin{remark}\label{remark_after_main_thm} $ $

\begin{enumerate}[(i)]
\item It is well-known (see for instance Corollary 3.3 of \cite{L17}) that the factor of i.i.d\ property is not necessarily inherited by a distributional limit, hence  Theorem \ref{thm:main} does not follow automatically from Lemma \ref{lemma:worms_converge_in_distribution} and the fact that  $\mathcal{Q}_{T}$ is a  factor of i.i.d.\ for each $T \in \mathbb{N}_+$.

\item In our proof of Theorem \ref{thm:main} we only construct the output of the function $\Upsilon$ for  $\otimes_{x \in V}\vartheta$-almost surely 
all elements of the input space $\Omega^V$. However, it is easy to see that if we define the  $\Upsilon$-value for the remaining elements of $\Omega^V$ to be the point measure on $ W^* \times [0,1] $ with zero total mass  then both statements \eqref{main_thm_i} and \eqref{main_thm_ii} of Theorem \ref{thm:main} remain valid.

\item \label{remark_cannot_choose_time_zero} Note that one possible way of creating a PPP $\omega$ on $W^* \times [0,1]$ with law $\mathcal{Q}_\infty$ is (a) to create a PPP $\omega'=\sum_{i\in I} \delta_{w^*_i}$
on $W^*$ with intensity measure $\nu$ and (b) conditional on $\omega'$, create   i.i.d.\ random variables $U_i, \, i \in I$ with $\mathrm{UNI}[0,1]$ distribution and then one obtains that the PPP $\sum_{i \in I} \delta_{(w^*_i,U_i)}$ has distribution  $\mathcal{Q}_\infty$. Note that even if we could create $\omega'$ as a factor of i.i.d.\ in a cheaper way than our  construction used in the proof of Theorem \ref{thm:main}, we do not know how to perform step (b) alone in a factor of i.i.d.\ fashion if $G$ is unimodular, since for each $i \in I$, the doubly infinite trajectory $w^*_i$ is only identified up to time shift equivalence, thus  we do not know how to assign a single vertex of $G$  to $w^*_i$ (where it can pick its $\mathrm{UNI}[0,1]$ label up) in an equivariant fashion. However,  if $G$ is non-unimodular then this can be done
using the tricks that we also used the proof of direction \eqref{unimod_B}$\implies$\eqref{unimod_A} of Proposition \ref{prop_QW_unimod}.
\end{enumerate}
\end{remark}

Our proof of Theorem \ref{thm:main} involves a construction of a PPP $\mathcal{Z}$ with law $\mathcal{Q}_\infty$ by coupling a sequence of finite-length interlacement point processes with increasing length in a way that they almost surely converge to $\mathcal{Z}$. Our next lemma provides
the coupling between consecutive elements of the sequence.
Recall the notion of $\mathcal{Q}_{T}$ from Definition \ref{def_PPP_of_labelled_worms} and the notion of $\mathrm{d}_K(\cdot,\cdot)$ from Definition \ref{def:local_metric}.

\begin{lemma}[Doubling the length of labeled finite-length interlacements] \label{lemma:labeled_coupling}
For any $T,m \in \N_+$  satisfying $m \leq \sqrt{T}$ there exists a probability space $(\hat{\Omega}, \hat{\mathcal{A}}, \hat{\vartheta})$ and a measurable map
\begin{equation}
		\label{eq:coupling_map_labels}
		 \hat{\Psi}_{T, m} \, : \,
		\wormspm(W_T \times [0,1] ) \times \hat{\Omega}^{V}
		\longrightarrow
		\wormspm(W_{2T} \times [0,1] )
	\end{equation}
satisfying the following properties.
\begin{enumerate}[(i)]
		\item\label{mating_correct_distribution_2T_labels} If $\mathcal{Z}_T \sim \mathcal{Q}_{T}$ and $\hat{\underline{\eta}}=  (\hat{\eta}_x)_{x \in V}$ are i.i.d\ with distribution $\hat{\vartheta}$ (moreover $\mathcal{Z}_T$ and $\hat{\underline{\eta}}$ are independent) then
		\begin{equation}\label{mating_identity_labels}
			\mathcal{Z}_{2T} :=
			\hat{\Psi}_{T, m} \left( \mathcal{Z}_T, \hat{\underline{\eta}} \right)
		\end{equation}
		has law $\mathcal{Z}_{2T} \sim \mathcal{Q}_{2T}$.
		\item\label{mating_equivariant_labels} For  any realization  of $(\mathcal{Z}_{T},\hat{\underline{\eta}})$ and any $\varphi \in \Gamma$ we have
		\begin{equation}
			\hat{\Psi}_{T,m} \left( \varphi(\mathcal{Z}_{T}), \varphi\left(  \hat{\underline{\eta}} \right) \right) = \varphi\left(\hat{\Psi}_{T,m} \left( \mathcal{Z}_{T}
			,  \hat{\underline{\eta}} \right)\right).
		\end{equation}
		\item\label{mating_error_bound_labels} Using the notation introduced in \eqref{mating_identity_labels}, there exists a constant $C \in \mathbb{R}_+$ that only depends on $G$ such that for any $K \subset \subset V$ we have
		\begin{equation}
		    \label{eq:coupling_error_labels}
			\prob \left( \mathrm{d}_K\left(
			 \mathcal{Z}_{T},  \mathcal{Z}_{2T}\right) > \frac{1}{m}
			\right)
			\leq
			C \cdot |K| \cdot T^{-1/7} \cdot m^{2/7}  +
			C \cdot |K|^2  \cdot  \frac{1}{\sqrt{T}}.
		\end{equation}
	\end{enumerate}
\end{lemma}

Before we prove Lemma \ref{lemma:labeled_coupling}, let us deduce  Theorem \ref{thm:main} from it.

\begin{proof}[Proof of Theorem \ref{thm:main}] Let us start with  $\mathcal{Z}^0 \sim \mathcal{Q}_{1}$ (which can be realized as a factor of i.i.d.) and let us 
iteratively define the point processes $\mathcal{Z}^n, n \in \mathbb{N}$ by letting 
\begin{equation}\label{Z_n_plus_1_Tn_mn_def}
\mathcal{Z}^{n+1} := \hat{\Psi}_{T_n, m_n} \left( \mathcal{Z}^n, \hat{\underline{\eta}}^n \right), \quad \text{where} \quad 
T_n:=2^n, \quad m_n := \left\lfloor \sqrt[4]{T_n}  \right\rfloor, \quad n \in \mathbb{N},   
\end{equation}
 $\hat{\Psi}_{T_n, m_n}$ is the map defined in \eqref{eq:coupling_map_labels} and $\hat{\underline{\eta}}^n, n \in \mathbb{N}$ are independent with distribution as in Lemma \ref{lemma:labeled_coupling}. Note that $\mathcal{Z}^n \sim  \mathcal{Q}_{T_n}, \, n \in \mathbb{N} $ follows from Lemma \ref{lemma:labeled_coupling}\eqref{mating_correct_distribution_2T_labels} by induction on $n$.

We will now show that  $(\mathcal{Z}^n)_{ n \in \mathbb{N}}$ almost surely locally converges (cf.\ Definition \ref{def:local_conv} and Claim \ref{claim_completeness}). It is enough to show that for every $K \fsubset V$,  $(\mathcal{Z}^n)_{ n \in \mathbb{N}}$ is a Cauchy sequence w.r.t.\ $\mathrm{d}_K(\cdot,\cdot)$ with probability $1$.
In order to show this, it is enough to show that for every $K \fsubset V$ there exists an almost surely finite random variable $N$ such that if $n \geq N$ then we have $\mathrm{d}_K\left(
			 \mathcal{Z}^{n},  \mathcal{Z}^{n+1}\right) \leq 1/m_n $. However, this follows from  
Lemma \ref{lemma:labeled_coupling}\eqref{mating_error_bound_labels}
by Borel-Cantelli:
 \begin{equation*}
   \sum_{n=0}^{\infty} \mathbb{P}\left( \mathrm{d}_K\left(
			 \mathcal{Z}^{n},  \mathcal{Z}^{n+1}\right) > \frac{1}{m_n} \right) \stackrel{ \eqref{eq:coupling_error_labels} }{\leq}  \sum_{n=0}^{\infty} \left( C \cdot |K| \cdot T_n^{-1/7} \cdot m_n^{2/7}  +
			C \cdot |K|^2  \cdot  \frac{1}{\sqrt{T_n}} \right) \stackrel{ \eqref{Z_n_plus_1_Tn_mn_def} }{<}  +\infty.
 \end{equation*}
 Let $\mathcal{Z}$ denote the $\wormspm(\wormspace \times [0,1])$-valued random variable that arises as the almost sure local limit of the sequence $(\mathcal{Z}^n)_{ n \in \mathbb{N}}$.
Observe that it follows by a repeated application of Lemma \ref{lemma:labeled_coupling}\eqref{mating_equivariant_labels} that $(\mathcal{Z}^n)_{ n \in \mathbb{N}}$   can be jointly realized as a factor of i.i.d., thus $\mathcal{Z}$ is a factor of i.i.d. The only thing left to prove is that $\mathcal{Z} \sim \mathcal{Q}_\infty$, but this follows from $\mathcal{Z}^n \sim  \mathcal{Q}_{T_n}, \, n \in \mathbb{N} $, Lemma \ref{lemma:worms_converge_in_distribution} and the fact that almost sure convergence implies convergence in distribution.
\end{proof}

 Lemma \ref{lemma:labeled_coupling} involves the coupling of labeled finite-length interlacement point processes. We will prove it using the next lemma, which involves the coupling of unlabeled finite-length interlacement point processes.
Recall the notion of $\mathcal{P}_{v,T}$ from Definition \ref{def_PPP_of_worms}.

\begin{lemma}[Doubling the length of unlabeled finite-length interlacements]
	\label{lemma:coupling_lemma}
	Let us fix  $T \in \N_+$ and $v \in [T^{-3/2}, T^2]$. There exists a probability space $(\tilde{\Omega}, \tilde{\mathcal{A}}, \tilde{\vartheta})$ and a measurable map
	\begin{equation}
		\label{eq:coupling_map}
		 \Psi_{v,T} \, : \,
		\wormspm(W_T) \times \tilde{\Omega}^{V}
		\longrightarrow
		\wormspm(W_{2T})
	\end{equation}
	satisfying the following properties.
	\begin{enumerate}[(i)]
		\item\label{mating_correct_distribution_2T} If $\mathcal{X}^{v,T} \sim \mathcal{P}_{v,T}$ and $\tilde{\underline{\eta}}=  (\tilde{\eta}_x)_{x \in V}$ are i.i.d\ with distribution $\tilde{\vartheta}$ (moreover $\mathcal{X}^{v,T}$ and $\tilde{\underline{\eta}}$ are independent) then
		\begin{equation}\label{mating_identity}
			\mathcal{X}^{v/2,2T} :=
			\Psi_{v,T} \left( \mathcal{X}^{v,T}, \tilde{\underline{\eta}} \right)
		\end{equation}
		has law $\mathcal{X}^{v/2,2T} \sim \mathcal{P}_{v/2,2T}$.
		\item\label{mating_equivariant} For any $\varphi \in \Gamma$ we have
		\begin{equation}
			\Psi_{v,T} \left( \varphi(\mathcal{X}^{v,T}), \varphi\left(  \tilde{\underline{\eta}} \right) \right) = \varphi\left(\Psi_{v,T} \left( \mathcal{X}^{v,T}
			,  \tilde{\underline{\eta}} \right)\right).
		\end{equation}
		\item\label{mating_error_bound} Using the notation introduced in \eqref{mating_identity}, there exists a constant $C \in \mathbb{R}_+$ that only depends on $G$ such that for any $K \subset \subset V$ we have
		\begin{equation}
		    \label{eq:coupling_error}
			\prob \left(
			\firstlastvis_K( \mathcal{X}^{v,T}) \neq \firstlastvis_K( \mathcal{X}^{v/2,2T})
			\right)
			\leq
			C \cdot |K| \cdot T^{4/7} \cdot v^{5/7}  +
			C \cdot |K|^2  \cdot  \sqrt{T} \cdot v.
		\end{equation}
	\end{enumerate}
\end{lemma}

Before we prove Lemma \ref{lemma:coupling_lemma}, let us deduce Lemma \ref{lemma:labeled_coupling} from it.

\begin{proof}[Proof of Lemma \ref{lemma:labeled_coupling}] Let us fix  $T,m \in \N_+$  satisfying $m \leq \sqrt{T}$. Let $\mathcal{Z}_T=\sum_{i \in I} \delta_{(w_i,t_i)} \sim \mathcal{Q}_T$. 
For any $n \in [m] $  let 
\begin{equation}\label{XTn_def}
\mathcal{Z}_T^n:= \sum_{i \in I} \delta_{(w_i,t_i)} \mathds{1}\left[\, t_i \in \left[ \frac{n-1}{m}, \frac{n}{m} \right]  \,\right], \quad
\mathcal{X}_T^n:= \sum_{i \in I} \delta_{w_i} \mathds{1}\left[\, t_i \in \left[ \frac{n-1}{m}, \frac{n}{m} \right]  \,\right].
\end{equation}
From this definition we obtain that almost surely we have
\begin{equation}\label{useful_sum_Z}
 \mathcal{Z}_T=   \sum_{n =1}^m \mathcal{Z}_T^n.
\end{equation}
Note that it follows from Definitions \ref{def_PPP_of_worms} and  \ref{def_PPP_of_labelled_worms}  that 
$\mathcal{X}_T^n, \, n=1,\dots, m$ are i.i.d.\ with distribution  $\mathcal{P}_{1/(Tm),T}$.
Let us now define
\begin{equation}
\mathcal{X}_{2T}^n:= \Psi_{1/(Tm),T} \left( \mathcal{X}_{T}^n, \tilde{\underline{\eta}}^n \right),
\end{equation}
where $\Psi_{1/(Tm),T}$ is the map defined in \eqref{eq:coupling_map} and $\tilde{\underline{\eta}}^n, n=1,\dots,m$ are independent with distribution as in Lemma \ref{lemma:coupling_lemma}. Now $\mathcal{X}_{2T}^n, \, n=1,\dots, m$ are i.i.d.\ with distribution  $\mathcal{P}_{1/(2Tm),2T}$
by Lemma \ref{lemma:coupling_lemma}\eqref{mating_correct_distribution_2T}. Let us denote $\mathcal{X}_{2T}^n=\sum_{i \in I_n} \delta_{w_i}$.
Given $\mathcal{X}_{2T}^n$, let $U^n_i, i \in I_n$ denote conditionally i.i.d.\ random variables with $\mathrm{UNI}\left[ \frac{n-1}{m}, \frac{n}{m} \right]$ distribution.
Note that we can use auxiliary i.i.d.\ randomness on $V$ to generate $U^n_i, i \in I_n$ in a factor of i.i.d.\ fashion. For any $n \in [m]$, let   
\begin{equation}\label{Z2T_def}
 \mathcal{Z}_{2T}^n:=\sum_{i \in I_n} \delta_{(w_i, U^n_i)}, \qquad   \mathcal{Z}_{2T}:=\sum_{n=1}^m  \mathcal{Z}_{2T}^n.
\end{equation}
It follows from our construction and the coloring property of Poisson point processes (cf.\ \cite[Section 5.2]{DRS14}) that $\mathcal{Z}_{2T}\sim \mathcal{Q}_{2T}$ (cf.\ Definition \ref{def_PPP_of_labelled_worms}), thus Lemma \ref{lemma:labeled_coupling}\eqref{mating_correct_distribution_2T_labels} holds. Note that 
Lemma \ref{lemma:labeled_coupling}\eqref{mating_equivariant_labels} also holds since we constructed $\mathcal{Z}_{2T}$ from $\mathcal{Z}_T$ and some auxiliary i.i.d.\ randomness on $V$ in a factor of i.i.d.\ fashion (cf.\ Lemma \ref{lemma:coupling_lemma}\eqref{mating_equivariant}). 

In order to show that the error bound of Lemma \ref{lemma:labeled_coupling}\eqref{mating_error_bound_labels} holds, let us first note that
by Lemma \ref{lemma:coupling_lemma}\eqref{mating_error_bound} there exists a constant $C \in \mathbb{R}_+$ that only depends on $G$ such that for any $K \subset \subset V$ we have
\begin{multline}\label{one_non_match_union_T_m}
   \mathbb{P}\left(\, \exists\, n \in [m] \, : \, \firstlastvis_K( \mathcal{X}_{T}^n) \neq \firstlastvis_K( \mathcal{X}_{2T}^n )   \,\right) \leq 
  \sum_{n=1}^m \mathbb{P}\left(\,  \firstlastvis_K( \mathcal{X}_{T}^n) \neq \firstlastvis_K( \mathcal{X}_{2T}^n )   \,\right)  
  \stackrel{ (*)  }{\leq} \\
  m \cdot\left( \frac{C \cdot |K| \cdot T^{4/7}}{(Tm)^{5/7}}  +
			\frac{C \cdot |K|^2  \cdot  \sqrt{T}}{  Tm } \right) = 
  C \cdot |K| \cdot T^{-1/7} \cdot  m ^{2/7}  +
			C \cdot |K|^2  \cdot    \frac{1}{\sqrt{T}}, 
\end{multline}
where $(*)$ follows from \eqref{eq:coupling_error} with $v:=1/(Tm)$, noting that the condition $v \in [T^{-3/2}, T^2]$ of Lemma \ref{lemma:coupling_lemma} follows from the condition $m \leq \sqrt{T}$ of Lemma \ref{lemma:labeled_coupling}.
 The desired bound \eqref{eq:coupling_error_labels}
will follow from \eqref{one_non_match_union_T_m} as soon as we show that if the complement of the  event  on the l.h.s.\ of \eqref{one_non_match_union_T_m} occurs, i.e., if we have
\begin{equation}\label{good_event_m}
 \firstlastvis_K( \mathcal{X}_{T}^n) = \firstlastvis_K( \mathcal{X}_{2T}^n ), \qquad n \in [m], 
\end{equation}
then we have
 $\mathrm{d}_K\left(
			 \mathcal{Z}_{T},  \mathcal{Z}_{2T}\right) \leq \frac{1}{m}$. The rest of the proof of
 Lemma \ref{lemma:labeled_coupling} is devoted to the proof of this implication.

 Let us first
note that by Definition \ref{def_aux_meas}, for each $ w \in \fwormspace(K)$ and $n \in [m]$ we have
 \begin{align}
  |\firstlastvis_K^{w}(\mathcal{Z}_T^n)|  &\stackrel{\eqref{XTn_def} }{=}  \firstlastvis_K( \mathcal{X}_{T}^n)(\{w\}), & 
|\firstlastvis_K^{w}(\mathcal{Z}_{2T}^n)| 
&\stackrel{ \eqref{Z2T_def} }{=}  \firstlastvis_K( \mathcal{X}_{2T}^n)(\{w\}), \\ \label{add_it_up}
\sum_{n =1}^m \firstlastvis_K^{w}(\mathcal{Z}_T^n)  &\stackrel{\eqref{useful_sum_Z} }{=} 
\firstlastvis_K^{w}(\mathcal{Z}_T), &
\sum_{n =1}^m \firstlastvis_K^{w}(\mathcal{Z}_{2T}^n)  &\stackrel{ \eqref{Z2T_def} }{=} \firstlastvis_K^{w}(\mathcal{Z}_{2T}).
\end{align} 
Thus, if \eqref{good_event_m} holds then we have 
  \begin{align}\label{nw_eq}
  |\firstlastvis_K^{w}(\mathcal{Z}_T^n)| &= |\firstlastvis_K^{w}(\mathcal{Z}_{2T}^n)|, \qquad w \in \fwormspace(K), \quad n \in [m], \\ 
  \label{nw_summed_eq}
   |\firstlastvis_K^{w}(\mathcal{Z}_T)| &= |\firstlastvis_K^{w}(\mathcal{Z}_{2T})|, \qquad w \in \fwormspace(K).
  \end{align} For any $w \in \fwormspace(K)$, let $\mathcal{S}(w):=\{\, n \in [m]\, : \,  |\firstlastvis_K^{w}(\mathcal{Z}^n_T)|\neq 0\, \} $. 
   Note that if \eqref{good_event_m} holds then  we have $\mathcal{S}(w)=\{\, n \in [m]\, : \,  |\firstlastvis_K^{w}(\mathcal{Z}^n_{2T})|\neq 0\, \} $ and
   \begin{equation}\label{wass_interval_m_bound} \mathrm{d}_{\mathrm{Wass}}\left( \frac{\firstlastvis_K^{w}( \mathcal{Z}_{T}^n )}{|\firstlastvis_K^{w}( \mathcal{Z}_{T}^n )|}, \frac{\firstlastvis_K^{w}(\mathcal{Z}_{2T}^n )}{|\firstlastvis_K^{w}(\mathcal{Z}_{2T}^n )|} \right)\leq  \frac{1}{m}, \quad w \in \fwormspace(K), \quad n \in \mathcal{S}(w),
   \end{equation}
   because both of the point measures $\firstlastvis_K^{w}( \mathcal{Z}_{T}^n )$ and $\firstlastvis_K^{w}(\mathcal{Z}_{2T}^n )$
are supported on $\left[ \frac{n-1}{m}, \frac{n}{m} \right]$.

Assuming that  $w \in \fwormspace(K)$ satisfies $|\firstlastvis_K^{w}(\mathcal{Z}_T)|\neq 0$ (or, equivalently, $\mathcal{S}(w)\neq \emptyset$), we define
\begin{equation}
    \alpha_n:= \frac{|\firstlastvis_K^{w}(\mathcal{Z}_T^n)|}{|\firstlastvis_K^{w}(\mathcal{Z}_T)| }
    \stackrel{ \eqref{good_event_m}, \eqref{nw_eq}, \eqref{nw_summed_eq}    }{=}
    \frac{|\firstlastvis_K^{w}(\mathcal{Z}_{2T}^n)|}{|\firstlastvis_K^{w}(\mathcal{Z}_{2T})| }, \qquad n \in \mathcal{S}(w).
\end{equation}
Assuming that \eqref{good_event_m} holds, for any  $w \in \fwormspace(K)$ satisfying $|\firstlastvis_K^{w}(\mathcal{Z}_T)|\neq 0$ we have
        \begin{multline}\label{wass_ineq_m}
    \mathrm{d}_{\mathrm{Wass}}\left( \frac{\firstlastvis_K^{w}( \mathcal{Z}_{T} )}{|\firstlastvis_K^{w}( \mathcal{Z}_{T} )|}, \frac{\firstlastvis_K^{w}(\mathcal{Z}_{2T} )}{|\firstlastvis_K^{w}(\mathcal{Z}_{2T} )|} \right) \stackrel{ \eqref{add_it_up} }{=}
    \mathrm{d}_{\mathrm{Wass}}\left( \sum_{n\in \mathcal{S}(w)} \alpha_n \frac{\firstlastvis_K^{w}( \mathcal{Z}^n_{T} )}{|\firstlastvis_K^{w}( \mathcal{Z}^n_{T} )|}, \sum_{n\in \mathcal{S}(w)} \alpha_n \frac{\firstlastvis_K^{w}(\mathcal{Z}^n_{2T} )}{|\firstlastvis_K^{w}(\mathcal{Z}^n_{2T} )|} \right)
    \leq \\
\sum_{n\in \mathcal{S}(w)} \alpha_n  \cdot  \mathrm{d}_{\mathrm{Wass}}\left(  \frac{\firstlastvis_K^{w}( \mathcal{Z}^n_{T} )}{|\firstlastvis_K^{w}( \mathcal{Z}^n_{T} )|},  \frac{\firstlastvis_K^{w}(\mathcal{Z}^n_{2T} )}{|\firstlastvis_K^{w}(\mathcal{Z}^n_{2T} )|} \right)
\stackrel{ \eqref{wass_interval_m_bound} }{\leq} \sum_{n\in \mathcal{S}(w)} \alpha_n \cdot \frac{1}{m} \stackrel{ \eqref{add_it_up} }{=}
\frac{1}{m}.
    \end{multline}
If \eqref{good_event_m} holds then \eqref{nw_summed_eq} also holds, thus  by Definition \ref{def:local_metric}  we can use \eqref{wass} to calculate $\mathrm{d}_K\left( \mathcal{Z}_{T},  \mathcal{Z}_{2T}\right)$, consequently \eqref{wass_ineq_m} gives that if \eqref{good_event_m} holds then we also have $\mathrm{d}_K\left( \mathcal{Z}_{T},  \mathcal{Z}_{2T}\right) \leq \frac{1}{m}$. Using this, the proof of \eqref{eq:coupling_error_labels} follows from \eqref{one_non_match_union_T_m}. The proof of Lemma \ref{lemma:labeled_coupling} is complete.
\end{proof} 

It remains to prove Lemma \ref{lemma:coupling_lemma}.
Sections \ref{section_matching_softlocal} and \ref{section_mating_of_worms} are devoted to this proof.

\section{Matching i.i.d.\ Poisson point processes on the vertex set}\label{section_matching_softlocal}

The goal of Section \ref{section_matching_softlocal} is to state and prove Lemma \ref{lemma:matching_iid_PPP_on_V}. In Section
\ref{section_mating_of_worms} we will use  Lemma \ref{lemma:matching_iid_PPP_on_V} to prove
 Lemma \ref{lemma:coupling_lemma}. Let us informally explain the idea of the proof of Lemma \ref{lemma:coupling_lemma}:  we will take a point process $\mathcal{X}^{v,T}$ with distribution $\mathcal{P}_{v,T}$, split it into two i.i.d.\ point processes $\mathcal{X}_1$ and  $\mathcal{X}_2$ with distribution $\mathcal{P}_{v/2,T}$ and   we will try to match the terminal points of (most of) the trajectories of $\mathcal{X}_1$ with the starting points of nearby trajectories of $\mathcal{X}_2$ using an auxiliary random walk trajectory of length $L$, gluing these pairs of trajectories of length $T$ as well as the auxiliary middle part of length $L$ to create a point process of  trajectories of length $2T+L$ with distribution  $\mathcal{P}_{v/2,2T+L}$, and finally we cut off a portion of length $L$ from these trajectories to obtain the desired PPP  $\mathcal{X}^{v/2,2T}$ with distribution  $\mathcal{P}_{v/2,2T}$. Intuitively, making the parameter $L$ bigger helps us to reach further when we find a pair for a trajectory (which results in a higher fraction of matched trajectories), but making $L$ bigger  also makes the local images of $\mathcal{X}^{v,T}$ and $\mathcal{X}^{v/2,2T}$ more different. We will see later in Section \ref{section_mating_of_worms} that in some sense the optimal choice of $L$ is  $L= \left\lceil T^{4/7} \cdot v^{-2/7} \right\rceil$.
 
 Lemma \ref{lemma:matching_iid_PPP_on_V} provides us with  a (partial) matching of the endpoints of the trajectories of $\mathcal{X}_1$ and the starting points of the trajectories of  $\mathcal{X}_2$, i.e., two i.i.d.\ point processes on $V$ with distribution  $\mathcal{P}_{v/2,1}$. In order to state Lemma \ref{lemma:matching_iid_PPP_on_V}, we need some definitions.

\medskip 

Let us define the projections $\textbf{p}_1: V \times V \to V$ and   $\textbf{p}_2: V \times V \to V$ by letting
$\textbf{p}_1((x,y))=x$ and $\textbf{p}_2((x,y))=y$.
In words, $\textbf{p}_1$ and $\textbf{p}_2$ are the projections on the first and second coordinates.

\noindent If $\omega= \sum_{i \in I} \delta_{(x_i,y_i)} \in \wormspm(V \times V)$, let us denote by
$\textbf{p}_1(\omega)=\sum_{i \in I} \delta_{x_i} $ and $\textbf{p}_2(\omega)=\sum_{i \in I} \delta_{y_i}$.


If $\varphi \in \Gamma$ and  $\omega= \sum_{i \in I} \delta_{(x_i,y_i)} \in \wormspm(V \times V)$, let us denote $\varphi(\omega)=\sum_{i \in I} \delta_{(\varphi(x_i),\varphi(y_i))}$.

Recall the notion of $\mathcal{P}_{v,1}$ from Definition \ref{def_PPP_on_V}.  Recall how $\Gamma$ acts on $\wormspm(\wormspace)$ from Definition \ref{def_auto}. 
Also recall that if $\underline{\hat{\eta}} = ( \hat{\eta}_{x} )_{x \in V} \in \hat{\Omega}^V$ and $\varphi \in \Gamma$ then we denote
$\varphi (\underline{\hat{\eta}}) = ( \hat{\eta}_{ \varphi^{-1}(x) } )_{x \in V}$.

\begin{lemma}[Partial matching of i.i.d.\ PPPs on $V$]
	\label{lemma:matching_iid_PPP_on_V}
	Let us fix $\alpha \in \mathbb{R}_+$ and $L \in \N$. There exists a probability space
 $(\hat{\Omega}, \hat{\mathcal{A}}, \hat{\pi})$ and a measurable map
	\begin{equation}
		\label{eq:coupling_map_VV}
		 \Psi^*_{\alpha,L} \, : \,
		\wormspm(V) \times \wormspm(V) \times \hat{\Omega}^{V}
		\longrightarrow
		\wormspm(V \times V)
	\end{equation}
	satisfying the following properties.
	\begin{enumerate}[(i)]
		\item\label{output_ppp_pair} If $\mathcal{R}_1$ and $\mathcal{R}_2$ are i.i.d.\ Poisson point processes on $V$ with law $\mathcal{P}_{\alpha,1}$ 
 and $\hat{\underline{\eta}}=  (\hat{\eta}_x)_{x \in V}$ are i.i.d\ with distribution $\hat{\pi}$ (moreover $(\mathcal{R}_1,\mathcal{R}_2)$ and $\hat{\underline{\eta}}$ are independent) then
		\begin{equation}\label{matching_identity}
			\mathcal{R} :=
			\Psi^*_{\alpha,L} \left( \mathcal{R}_1, \mathcal{R}_2, \hat{\underline{\eta}} \right)
		\end{equation}
		is a PPP on $V \times V$ with intensity measure $\nu^{(2)}(\{ (x,y)\} ):= \alpha \cdot p_L(x,y)$ (cf.\ \eqref{random_walk_transitions}).
		\item\label{matching_ppp_equivariant} For  any $\varphi \in \Gamma$ we have
		\begin{equation}\label{triple_varphi_invariance}
			\Psi^*_{\alpha,L} \left( \varphi(\mathcal{R}_1),\varphi(\mathcal{R}_2),  \varphi\left(  \hat{\underline{\eta}} \right) \right) = \varphi\left(\Psi^*_{\alpha,L} \left(   \mathcal{R}_1, \mathcal{R}_2,  \hat{\underline{\eta}} \right)\right).
		\end{equation}
\item\label{first_projection_is_R1}  We have $\textbf{p}_1(\mathcal{R})=\mathcal{R}_1$.
\item \label{matching_ppp_error_bound_from_surface} 	 We have $\mathbb{E}\left[ \Big( \mathcal{R}_2(\{ o \} ) - \textbf{p}_2(\mathcal{R})(\{ o \}) \Big)_+   \right] \leq \sqrt{2\alpha \cdot p_{2L}(o,o) } $.
		\item	\label{matching_ppp_error_bound_from_remaining}  We have $\mathbb{E}\left[ \Big(   \textbf{p}_2(\mathcal{R})(\{ o \}) -\mathcal{R}_2(\{ o \} ) \Big)_+   \right] \leq \sqrt{2 \alpha \cdot p_{2L}(o,o) } $.
	\end{enumerate}
\end{lemma}
In words, $ \left( \mathcal{R}_2(\{ o \} ) - \textbf{p}_2(\mathcal{R})(\{ o \}) \right)_+$ is the number of unmatched points of $\mathcal{R}_2$ located at vertex $o$ and 
  $ \left(   \textbf{p}_2(\mathcal{R})(\{ o \}) -\mathcal{R}_2(\{ o \} ) \right)_+$
  is the number of unmatched points of $\mathcal{R}_1$ which have a ``phantom pair'' (cf.\ Remark \ref{remark_phantom})
  located at vertex $o$.
Section \ref{section_matching_softlocal} is devoted to the proof of Lemma \ref{lemma:matching_iid_PPP_on_V}.
In Section \ref{subsection:greedy_mathcing_using_SLT_method} we construct the partial matching $\mathcal{R}$ satisfying \eqref{output_ppp_pair}, \eqref{matching_ppp_equivariant} and \eqref{first_projection_is_R1}. In Section \ref{subsetion_matching_error_bounds} we prove that the error bounds \eqref{matching_ppp_error_bound_from_surface} and \eqref{matching_ppp_error_bound_from_remaining} hold.

\subsection{Matching PPPs using the soft local time method}
\label{subsection:greedy_mathcing_using_SLT_method}

Let us introduce the notation
\begin{equation}\label{R1_R2_labels_I1_I2}
  \mathcal{R}_1=\sum_{i \in I_1} \delta_{x_i}, \qquad  \mathcal{R}_2=\sum_{i \in I_2} \delta_{y_i}.
\end{equation}

\noindent Let us assign i.i.d.\ uniformly distributed labels on $[0,\alpha]$ to the points of $\mathcal{R}_2$, noting that this can be done in a factor of i.i.d.\ fashion.
 If $i \in I_2$, let $t_i$ denote the label assigned to $y_i$, thus
	 $ \mathcal{Y}^{\alpha}:= \sum_{i \in I_2} \delta_{(y_i,t_i)}$ is a PPP on $V \times [0,\alpha]$ with intensity measure $\mu^{V} \times \lambda \mathds{1}[\, [0,\alpha] \, ]$.

Then we use further i.i.d.\  randomness on $V$ to
 extend the PPP $\mathcal{Y}^\alpha$  in a factor of i.i.d.\ fashion  to a PPP $\mathcal{Y}$ on $V \times \mathbb{R}_+$ with intensity measure $\mu^{V} \times \lambda $. Let us denote
\begin{equation}
    \label{notation:extended_labelled_worms:Z}
	\mathcal{Y} =   \sum_{j \in J} \delta_{(y_j, t_j)}.
\end{equation}
 Note  that $\mathcal{Y}$ is still independent of $\mathcal{R}_1$, moreover we have
 \begin{equation}\label{Y_alpha_from_Y}
 \mathcal{Y}^{\alpha} =\mathcal{Y}\mathds{1}[\, V \times [0,\alpha]\, ] \qquad \text{and} \qquad \mathcal{R}_2=  \sum_{j \in J} \delta_{y_j} \mathds{1}[ t_j \leq \alpha ].
 \end{equation}
 Alternatively, one may view $\mathcal{Y}$ as follows: $\mathcal{Y}$ is made up of i.i.d.\ homogeneous Poisson point processes on $\mathbb{R}_+$ with unit intensity, one such point process for each $x \in V$.

 Also note that  $t_j$ uniquely identifies  the corresponding $j \in J$, since the values $t_j, j \in J$ are almost surely distinct.

\begin{remark}\label{remark_phantom}  In order to create the PPP $\mathcal{R}$ on $V \times V$ as in  Lemma \ref{lemma:matching_iid_PPP_on_V}, we will find exactly one pair in $\mathcal{Y}$ for each point in $\mathcal{R}_1$.
\begin{enumerate}[(a)]
\item\label{a_silly} Since $\mathcal{Y} \geq \mathcal{Y}^\alpha$, it might happen that  a point from $\mathcal{R}_1$ only has a phantom pair, that is, a point of $\mathcal{Y}$ that does not correspond to a point in $\mathcal{R}_2$.
\item \label{b_silly}
Moreover, it might also happen that a point of $\mathcal{Y}$ that corresponds to a point in $\mathcal{R}_2$ is not matched to any point in $\mathcal{R}_1$.
\end{enumerate}
Nevertheless, we will  show that these mismatches are infrequent if the parameter $L$ is big enough. More specifically, the error described in \eqref{a_silly} will be bounded in Lemma \ref{lemma:matching_iid_PPP_on_V}\eqref{matching_ppp_error_bound_from_remaining} and the error described in \eqref{b_silly}  will be bounded in Lemma \ref{lemma:matching_iid_PPP_on_V}\eqref{matching_ppp_error_bound_from_surface}.
\end{remark}

We are ready to define the factor of i.i.d.\ algorithm which assigns a pair in $\mathcal{Y}$  to every point in  $\mathcal{R}_1$.
More precisely, we will construct a matching resulting in a point process $\mathcal{R}$ of pairs which is a PPP with intensity measure
 $\nu^{(2)}(\{ (x,y)\} )= \alpha \cdot p_L(x,y)$, as required by Lemma \ref{lemma:matching_iid_PPP_on_V}\eqref{output_ppp_pair}.

The algorithm will perform the matching in rounds.
For $\ell \in \mathbb{N}$, let us denote by $S_\ell$ the set of indices of the points of $\mathcal{R}_1$ which are not yet matched by the end of round $\ell$.
Let $S_0 := I_1$.
We will have  $S_{\ell+1} \subseteq S_{\ell}$ for any $\ell \in \mathbb{N}$. We will see that $S_\ell$  converges point-wise to $\emptyset$ almost surely as $\ell \to \infty$, i.e., we have
\begin{equation}
	\label{eq:almost_surely_finds_a_pair}
	\mathbb{P} (\, \exists \, i \in I_1 \;\, \forall \, \ell \in \mathbb{N}\; :  \; i \in S_\ell\, ) = 0.
\end{equation}
	
Let us also define a randomly growing surface, encoded by a function $g: V \times \mathbb{N} \to \mathbb{R}_+ $. We say that $g(y,\ell)$ is the height of the surface at vertex $y \in V$ after the end of round $\ell$. We have $g(y,0)=0$ for each $y \in V$, and we will have $g(y,\ell)\leq g(y,\ell+1)$ for all $y \in V$ and all $\ell \in \mathbb{N}$.

As we will see, this height function will govern the matching in the sense that a point $(y_j, t_j)$ of the point process  $\mathcal{Y}$ will be already matched to some point in $\mathcal{R}_1$ by the end of round $\ell$ if and only if $t_j \leq g(y_j,\ell)$. Moreover, it will also follow from the algorithm that almost surely
\begin{equation}\label{g_infty}
	g(y,\infty) := \lim_{\ell \to \infty} g(y,\ell)<+\infty, \qquad y \in V
\end{equation}
 and we will see in Lemma \ref{lemma_g_expect_var} that   $g(y, \infty)$ is close to $\alpha$ if $L$ is big enough.
	
Now let us describe how to obtain $S_{\ell+1}$ from $S_\ell$ and $g(\cdot,\ell+1)$ from $g(\cdot, \ell)$, and how to find a pair for $x_i, i \in S_{\ell} \setminus S_{\ell+1}$.
This is a variant of the \emph{soft local time method}, introduced in \cite{PT15}.
Given $\mathcal{R}_1$, let $U_{\ell,i}, i \in I_1$ denote i.i.d.\ $\mathrm{UNI}[0,1]$ random variables,  realized in a factor of i.i.d.\ fashion. Let us define
\begin{align}\label{the_smallest_one}
	\widetilde{S}_{\ell}:=
	\left\{ \,
	i \in S_{\ell} \, : \,
	\parbox{14em}{
		$U_{\ell,i} < U_{\ell,i'}$  for all $i' \in S_{\ell} \setminus \{ i \}$ \\
		for which  $\distance(x_i,x_{i'} ) \leq 2L$
	}
	\, \right\},
	\qquad S_{\ell+1}:= S_\ell \setminus \widetilde{S}_{\ell},
\end{align}
where we recall that $\distance(\cdot,\cdot)$ denotes the graph distance on $G$. Note that $ \widetilde{S}_{\ell}$ is obtained from $S_\ell$ in a factor of i.i.d.\ fashion. Also note that we have
\begin{equation}
	\label{eq:far_apart}
	\forall \, i \neq i' \in  \widetilde{S}_{\ell} \; : \;  \distance(x_i, x_{i'} ) > 2L.
\end{equation}

In round $\ell+1$ we will find a pair for each point $x_i$, $i \in \widetilde{S}_{\ell}$ from the set of vertices of $\mathcal{Y}$ not yet matched using the following method. Having already constructed $g(.,\ell)$ and $\widetilde{S}_{\ell}$ we define
\begin{equation}
	\label{def:growth_of_height_function_round_ell}
	\eta_i:= \min\left\{\, t \geq 0 \, : \, \exists \, j \in J \text{ such that }
	t_j \in \big(g( y_j, \ell ),\, g( y_j, \ell ) + t \cdot  p_L( x_i, y_j)\big]  \, \right\},
	\quad
	i \in \widetilde{S}_{\ell}.
\end{equation}
Note that in \eqref{def:growth_of_height_function_round_ell} we could write $\min$ instead of $\inf$ because for each $i \in  \widetilde{S}_{\ell}$ the function $y \mapsto  p_L( x_i, y) $ is finitely supported (it is supported on the ball of radius $L$ around $x_i$).  
 For every $i \in \widetilde{S}_{\ell}$ there exist an almost surely unique index $\pi(i) \in J$ for which
\begin{equation}
	g( y_{\pi(i)}, \ell ) + \eta_i \cdot  p_L(x_i,y_{\pi(i)}) = t_{\pi(i)}
	\end{equation}
holds. For each $i \in \widetilde{S}_{\ell}$
\begin{equation}
    \text{we match the point $x_i$ of $\mathcal{R}_1$ to the (labeled) point $(y_{\pi(i)}, t_{\pi(i)})$ of $\mathcal{Y}$. }
\end{equation}
In words: we start from the function $g(\cdot, \ell)$ and  for each $i \in \widetilde{S}_{\ell}$ we grow $g$ inside the ball of radius $L$ centered at the point $x_i$ by increasing $t$ in the expression $g(\cdot, \ell)+ t \cdot  p_L( x_i, \cdot)$.
We grow the function around $x_i$ until the graph of the function consumes a new point $(y_{\pi(i)}, t_{\pi(i)})$ from the support of $\mathcal{Y}$, and this new point becomes the pair of $x_i$. Also note that we can do this simultaneously for all $x_i, i \in \widetilde{S}_\ell$ without ambiguities, since the supports of the functions $ p_L( x_i, \cdot), i \in \widetilde{S}_{\ell}$ are disjoint by \eqref{eq:far_apart}.

 In order to  finish round $\ell + 1$, we use the variables $\eta_i$, $i \in \widetilde{S}_{\ell}$ to define
\begin{equation}
	\label{def_of_height_function}
	g(y,\ell+1):=g(y,\ell)+\sum_{i \in \widetilde{S}_\ell } \eta_i \cdot p_L(x_i,y), \qquad y \in V.
\end{equation}
Intuitively, we explore the sub-region of  $V\times \mathbb{R}_+$ that lies below the graph of the function $g(\cdot, \ell+1)$ by the end of round $\ell+1$.
As we mentioned before,  a point $(y_j, t_j)$ from $\mathcal{Y}$ is matched to some point in $\mathcal{R}_1$ by the end of round $\ell+1$ if and only if $t_j \leq g(y_j,\ell+1)$, or, more formally, 
\begin{equation}\label{matched_under_surface_ell}
 \sum_{i  \in I_1 \setminus S_{\ell+1}} \delta_{(y_{\pi(i)}, t_{\pi(i)})} = \sum_{j \in J} \delta_{(y_j,t_j)} \mathds{1}[ t_j \leq g(y_j,\ell+1) ].
\end{equation}

\begin{definition}[Sigma-algebra $\mathcal{F}_\ell$]
	\label{def:SLT_sigma_algebras}
	Given $\ell \in \N$ let $\mathcal{F}_{\ell}$ denote the $\sigma$-field generated by $\mathcal{R}_1$, the  variables $U_{k,i}$, $k = 0, 1, \ldots, \ell$, $i \in I_1$, the random variables $\eta_i$, $i \in I_1 \setminus S_{\ell}$, the indices $\pi(i) \in J, i \in I_1 \setminus S_{\ell}$ and the points $(y_{\pi(i)}, t_{\pi(i)}),  i \in I_1 \setminus S_{\ell}$  of $\mathcal{Y}$ already matched to some points of $\mathcal{R}_1$ up to the end of the $\ell$'th round.
\end{definition}

\noindent
In words, $\mathcal{F}_{\ell}$ contains all the information about everything that has been constructed up to the end of round $\ell$ and the variables $U_{\ell,i}$, $i \in I_1$. For example, the index sets $\widetilde{S}_k$ and the functions $y \mapsto g(y,k)$ for $k = 0, \ldots, \ell$ are all measurable with respect to $\mathcal{F}_{\ell}$.

\begin{definition}[Point process  $\dot{\mathcal{Y}}_{\ell}$] For any $\ell \in \mathbb{N}$
     let us denote by $J_{\ell}$ the set of indices of the points of $\mathcal{Y}$ not matched until the end of the $\ell$'th round. Let us denote by
         \begin{equation}\dot{\mathcal{Y}}_{\ell} := \sum_{j \in J_{\ell}} \delta_{(y_j, t_j)}
         \end{equation}
         the point process of labeled points of $\mathcal{Y}$ not matched until the end of the $\ell$'th round.
\end{definition}
Note that with the above notation we have $\mathcal{Y}= \sum_{i \in I_1 \setminus S_\ell} \delta_{(y_{\pi(i)},t_{\pi(i)})} + \dot{\mathcal{Y}}_{\ell}$.

\begin{lemma}[Consequences of the soft local time method]
	\label{lemma:soft_local_time}
	For any $\ell \in \N$, given $\mathcal{F}_{\ell}$,
	\begin{enumerate}[(i)]
		\item the random variables $\eta_i, i \in \widetilde{S}_{\ell}$ are conditionally i.i.d.\ with $\mathrm{EXP}(1)$ distribution;
		\item\label{coloring_soft} the random variables $y_{\pi(i)}, i \in \widetilde{S}_{\ell}$ are conditionally independent with distribution
		\begin{equation}
		    \mathbb{P}\left(\, y_{\pi(i)} = y \, | \, \mathcal{F}_{\ell} \,  \right) = p_L(x_i, y),
		    \qquad
		    y \in V,
		\end{equation}
		\item the random variables
		 $\eta_i, i \in \widetilde{S}_{\ell}$ and  $y_{\pi(i)}, i \in \widetilde{S}_{\ell}$ are conditionally independent,
	\item\label{remaining_PPP_nice}
	the point process   $\dot{\mathcal{Y}}_\ell$ is a PPP on $V \times \mathbb{R}_+$ with intensity measure $\dot{\upsilon}^{\ell}$, where
	\begin{equation}
		\dot{\upsilon}^{\ell} (\{y\} \times [t, t+ \mathrm{d}t]) =  \ind \left[ t > g(y, \ell) \right] \, \mathrm{d} t.
	\end{equation}
\end{enumerate}
\end{lemma}

\begin{proof} One proves the statements of the lemma by induction on $\ell$.
	Considering any $\ell \in \N$, observe that by \eqref{eq:far_apart}, the supports of the functions $y \mapsto p_L(x_i, y)$, $i \in \widetilde{S}_{\ell}$ are disjoint. As a consequence, in the $\ell$'th round we can use Proposition $4.1.$ of \cite{PT15} simultaneously for all indices from $\widetilde{S}_{\ell}$ to conclude the proof of Lemma \ref{lemma:soft_local_time}.
\end{proof}

Now the proof of \eqref{eq:almost_surely_finds_a_pair} is straightforward. Let us condition on $\mathcal{R}_1$ and let us also fix $i \in I_1$.
Denote by $N$ the number of points in $\mathcal{R}_1$ that are closer than $2L$ to  $x_i$:
\begin{equation} N:= \sum_{i' \in I_1} \ind[\, \distance(x_i, x_{i'} ) \leq 2L \,].
\end{equation}
Observe that $N<+\infty$ holds almost surely.
Moreover, we have
\begin{equation}
\mathbb{P}\left(\,i \in \widetilde{S}_\ell\; \middle| \; \mathcal{F}_{\ell-1}\,  \right) \stackrel{ \eqref{the_smallest_one} }{\geq} \frac{1}{N} \mathds{1}[i \in S_\ell],
\end{equation}
 since the random variables $U_{\ell,i}, i \in I_1, \ell \in \mathbb{N}$ are i.i.d.
Thus the number of rounds it takes for us to find a pair for $i$ is stochastically dominated by a $\mathrm{GEO}(1/N)$ random variable, therefore it is almost surely finite and thus \eqref{eq:almost_surely_finds_a_pair} holds.

Recalling the definition of $g(y,\infty)$ from	\eqref{g_infty} we note that  if we let $\ell \to \infty$ in \eqref{matched_under_surface_ell} then we obtain that the  point process of points of $\mathcal{Y}$ that are matched to a point in $\mathcal{R}_1$ consists of those points of $\mathcal{Y}$ that are below the graph of $g(\cdot, \infty)$:
\begin{equation}\label{Y_under_the_surface}
   \sum_{i \in I_1} \delta_{(y_{\pi(i)}, t_{\pi(i)})} = \sum_{j \in J} \delta_{(y_j,t_j)} \mathds{1}[ t_j \leq g(y_j,\infty) ].
\end{equation}

Let us introduce the $\sigma$-field
\begin{equation}
    \mathcal{F} := \sigma \left( \bigcup_{\ell \geq 0} \mathcal{F}_{\ell} \right).
\end{equation}

\noindent	Let us denote by $\pi(I_1):=\{ \pi(i), \, i \in I_1\}$ the subset of the index set $J$ which consists of the pairs matched to some index in $I_1$. Let us introduce the point process of unmatched points $\dot{\mathcal{Y}}$:
	\begin{equation}
\label{Y_dot}
	\dot{\mathcal{Y}} :=  \sum_{j \in J \setminus \pi(I_1)} \delta_{(y_j, t_j)} \stackrel{\eqref{Y_under_the_surface}}{=} \sum_{j \in J} \delta_{(y_j,t_j)} \mathds{1}[ t_j > g(y_j,\infty) ].
	\end{equation}	

\noindent
Our next result follows from Lemma \ref{lemma:soft_local_time}.
\begin{corollary}
	\label{corollary:PPP_of_the_unmatched} $ $

\begin{enumerate}[(i)]	
	\item\label{maradek_PPP_infty}
	 Given $\mathcal{F}$, the point process $\dot{\mathcal{Y}}$ is a PPP on $V \times \R_+$ with intensity measure $\dot{\upsilon}$, where
	\begin{equation}
		\dot{\upsilon}(\{y\} \times [t, t+ \mathrm{d}t]) = \ind[t > g(y, \infty)] \, \mathrm{d} t.
	\end{equation}
\item\label{R_is_a_PPP_as_stated} 	The  point process
\begin{equation}
	\label{PPP_of_connected_pairs}
	\mathcal{R}:= \sum_{i \in I_1} \delta_{(x_i,y_{\pi(i)})}
\end{equation}
is a Poisson point process on $V \times V$ with intensity measure  $\nu^{(2)}(\{ (x,y)\} )= \alpha \cdot p_L(x,y)$.
	\end{enumerate}
\end{corollary}
\begin{proof} Using \eqref{g_infty}, the property stated in \eqref{maradek_PPP_infty} follows if we let $\ell \to \infty$ in  Lemma \ref{lemma:soft_local_time}\eqref{remaining_PPP_nice}.
In order to see that \eqref{R_is_a_PPP_as_stated} holds, we use our assumption that $\mathcal{R}_1  \sim \mathcal{P}_{\alpha,1}$,
Lemma \ref{lemma:soft_local_time}\eqref{coloring_soft}  and the coloring property of Poisson point processes (cf.\ \cite[Section 5.2]{DRS14}).
\end{proof}

This already shows that the statements \eqref{output_ppp_pair}, \eqref{matching_ppp_equivariant} and \eqref{first_projection_is_R1} of Lemma \ref{lemma:matching_iid_PPP_on_V} hold with our construction of $\mathcal{R}$. It remains to show that the error bounds
\eqref{matching_ppp_error_bound_from_surface} and \eqref{matching_ppp_error_bound_from_remaining} hold. In order to do so, we need some further preparations.

\subsection{Bounds on the number of unmatched points}
\label{subsetion_matching_error_bounds}
	
In the previous section (see \eqref{def_of_height_function} and Lemma \ref{lemma:soft_local_time}) we have shown that
\begin{equation}
	\label{eq:g_same_law_as}
\begin{array}{c}	g(y,\infty) = \sum_{i \in I_1} \eta_i \cdot p_L(x_i,y), \text{ where} \\ \text{$\eta_i, i \in I_1$ are conditionally i.i.d.\ with $\mathrm{EXP}(1)$ distribution given $\mathcal{R}_1$, } \end{array}
\end{equation}
where $\mathcal{R}_1= \sum_{i \in I_1} \delta_{x_i} \sim \mathcal{P}_{\alpha,1}$.

\begin{lemma}[Expectation and variance of the height function $g(\cdot,\infty)$]\label{lemma_g_expect_var}
We have
\begin{equation}
\mathbb{E}[ g(y,\infty)]=\alpha, \qquad \Var \left( g(y, \infty) \right)=	2\alpha \cdot p_{2L}(y,y)=2\alpha \cdot p_{2L}(o,o).
\end{equation}
    \end{lemma}	
\begin{proof}
 In the case of the expectation we have
\begin{equation}
	\label{expectation_of_g}
	\mathbb{E}[ g(y,\infty)]
	 =
	\mathbb{E} \left[ \sum_{i \in I_1} \eta_i \cdot p_L(x_i,y) \right] \stackrel{(*)}{=}
	\mathbb{E} \left[ \sum_{i \in I_1} p_L(x_i,y) \right]
	 \stackrel{(\circ) }{=}
	\alpha \cdot \sum_{x \in V}  p_L \left( x, y \right) \stackrel{\eqref{heat_kernel_symmetric} }{=}
	\alpha,
\end{equation}
where in $(*)$ we used the law of total expectation together with \eqref{eq:g_same_law_as}, and in $(\circ)$ we used
 that $\mathcal{R}_1 \sim \mathcal{P}_{\alpha,1}$. In the case of the variance we have
	\begin{multline}
\Var \left( g(y, \infty) \right)\stackrel{(**)}{=}
\mathbb{E}\left[ \Var \left( g(y, \infty) \, | \, \mathcal{R}_1 \right) \right]+
\Var\left( \mathbb{E}\left[ g(y, \infty) \, | \, \mathcal{R}_1 \right]  \right)
\stackrel{ \eqref{eq:g_same_law_as} }{=}
\mathbb{E}\left[ \sum_{i \in I_1}  p^2_L \left( x_i , y \right) \right] +\\
\Var\left[  \sum_{i \in I_1}  p_L \left( x_i , y \right)  \right] \stackrel{(\circ \circ )}{=}
2 \alpha \sum_{x \in V }    p^2_L \left( x , y \right) \stackrel{\eqref{heat_kernel_symmetric} }{=} 2 \alpha \sum_{x \in V }  p_L \left( y , x \right)    p_L \left( x , y \right) \stackrel{( \bullet \bullet)}{=} 2 \alpha \cdot  p_{2L}(y,y),
\end{multline}
where in $(**)$ we used the law of total variance, $(\circ \circ)$  follows from the fact that $\mathcal{R}_1 \sim \mathcal{P}_{\alpha,1}$  and $( \bullet \bullet)$ holds by the Chapman-Kolmogorov equations.
\end{proof}

\begin{proof}[Proof of Lemma \ref{lemma:matching_iid_PPP_on_V}\eqref{matching_ppp_error_bound_from_surface}]
Note that we have
\begin{equation}\label{R_2_p2R_and_Y}
\mathcal{R}_2(\{ o \}) \stackrel{ \eqref{Y_alpha_from_Y} }{=} \mathcal{Y}( \{o\} \times [0,\alpha] ),
\qquad
\textbf{p}_2(\mathcal{R})(\{ o \}) \stackrel{\eqref{Y_under_the_surface},\eqref{PPP_of_connected_pairs} }{=}\mathcal{Y}( \{o\} \times [0,g(o, \infty)] ),
\end{equation}
which implies that the number of unmatched points of $\mathcal{R}_2$ at vertex $o$ is equal to
\begin{equation}
  \Big( \mathcal{R}_2(\{ o \} ) - \textbf{p}_2(\mathcal{R})(\{ o \}) \Big)_+ =
  \mathcal{Y}( \{o\} \times ( g(o, \infty) \wedge \alpha ,  \alpha  ] )\stackrel{ \eqref{Y_dot} }{=} \dot{\mathcal{Y}}( \{ o \} \times [0,\alpha]).
\end{equation}
  We can thus write
\begin{multline}
  \mathbb{E}\left[ \Big( \mathcal{R}_2(\{ o \} ) - \textbf{p}_2(\mathcal{R})(\{ o \}) \Big)_+   \right] =
   \mathbb{E}\left[ \dot{\mathcal{Y}}( \{ o \} \times [0,\alpha] )     \right]=  \mathbb{E}\left[ \mathbb{E}\left(  \dot{\mathcal{Y}}( \{ o \} \times [0,\alpha])\; \middle| \; \mathcal{F} \right)  \right]
  \stackrel{(*)}{=}\\
	    \E \left[ \left( \alpha - g(o, \infty) \right)_+ \right]\leq
 \E \left[ | \alpha - g(o, \infty)| \right]
 \stackrel{(**)}{\leq}
	    \sqrt{ \E \left[ (\alpha - g(o, \infty))^2 \right] }
	    \stackrel{(\bullet) }{\leq}
	    \sqrt{2 \alpha \cdot p_{2L}(o,o)},
\end{multline}
where in $(*)$ we used Corollary \ref{corollary:PPP_of_the_unmatched}\eqref{maradek_PPP_infty}, in $(**)$ we used Jensen's inequality and in $(\bullet)$
we used Lemma \ref{lemma_g_expect_var}. The proof of Lemma \ref{lemma:matching_iid_PPP_on_V}\eqref{matching_ppp_error_bound_from_surface} is complete.
\end{proof}

\begin{proof}[Proof of Lemma \ref{lemma:matching_iid_PPP_on_V}\eqref{matching_ppp_error_bound_from_remaining}]
Let $ M:=  \mathcal{R}_2(\{ o \} ) \wedge \textbf{p}_2(\mathcal{R})(\{ o \}) $ denote the number of matched points of $\mathcal{R}_2$ at vertex $o$. We have
\begin{equation}\label{min_expect_lower}
\mathbb{E}[M]= \mathbb{E}\left[ \mathcal{R}_2(\{ o \} )- \Big( \mathcal{R}_2(\{ o \} ) - \textbf{p}_2(\mathcal{R})(\{ o \}) \Big)_+ \right] \stackrel{(\circ)}{\geq}
\alpha - \sqrt{2 \alpha \cdot p_{2L}(o,o)},
\end{equation}
where in $(\circ)$ we used that $ \mathcal{R}_2  \sim \mathcal{P}_{\alpha,1}$ and Lemma \ref{lemma:matching_iid_PPP_on_V}\eqref{matching_ppp_error_bound_from_surface}.

From this we obtain the desired upper bound on the number of unmatched points of $\mathcal{R}_1$ that have a phantom pair located at vertex $o$: 
\begin{multline*}
   \mathbb{E}\left[ \Big(   \textbf{p}_2(\mathcal{R})(\{ o \}) -\mathcal{R}_2(\{ o \} )  \Big)_+   \right]=
   \mathbb{E}\left[  \textbf{p}_2(\mathcal{R})(\{ o \}) -M  \right] \stackrel{ \eqref{min_expect_lower} }{\leq} \\ \mathbb{E}\left[  \textbf{p}_2(\mathcal{R})(\{ o \})  \right] - \left( \alpha - \sqrt{2 \alpha \cdot p_{2L}(o,o)} \right) \stackrel{(\circ \circ)}{=}
   \alpha-  \left( \alpha - \sqrt{2 \alpha \cdot p_{2L}(o,o)} \right)=
    \sqrt{2 \alpha \cdot p_{2L}(o,o)},
\end{multline*}
where $(\circ \circ)$ holds because $ \textbf{p}_2(\mathcal{R}) \sim \mathcal{P}_{\alpha,1}$, as we now explain.
We know from  Lemma \ref{lemma:matching_iid_PPP_on_V}\eqref{output_ppp_pair} that $\mathcal{R}$ is a PPP on $V \times V$ with intensity measure $\nu^{(2)}(\{ (x,y)\} ):= \alpha \cdot p_L(x,y)$, thus by the mapping property of Poisson point processes (cf.\ \cite[Section 5.2]{DRS14})  we obtain that $ \textbf{p}_2(\mathcal{R})$  is a PPP on $V$ with intensity measure $\nu^{(1)}$, where
\begin{equation*}
  \nu^{(1)}(\{ y \}  ):= \nu^{(2)}\big( (\textbf{p}_2)^{-1} (\{ y \})  \big)= \alpha \sum_{x \in V} p_L(x,y)  \stackrel{\eqref{heat_kernel_symmetric}}{=} \alpha \sum_{x \in V} p_L(y,x) = \alpha, \quad  y \in V.
\end{equation*}
 Thus  $ \textbf{p}_2(\mathcal{R}) \sim \mathcal{P}_{\alpha,1}$ follows from Definition \ref{def_PPP_on_V}.  The proof of Lemma \ref{lemma:matching_iid_PPP_on_V}\eqref{matching_ppp_error_bound_from_remaining} is complete.
\end{proof}
The proof of Lemma \ref{lemma:matching_iid_PPP_on_V} is complete.

\section{Doubling the length of unlabeled finite-length interlacements}\label{section_mating_of_worms}

The goal of this section is to  prove Lemma \ref{lemma:coupling_lemma}. In Section \ref{subsection:decorating_trajectories_with_labels} we construct a coupling of a PPP with distribution  $\mathcal{P}_{v,T}$ and a PPP with distribution $\mathcal{P}_{v/2,2T}$ that satisfies  properties \eqref{mating_correct_distribution_2T} and \eqref{mating_equivariant} of Lemma \ref{lemma:coupling_lemma}. In Section \ref{subsection_error_bound_of _local_map} we show that the coupling that we constructed also satisfies the error bound stated in Lemma \ref{lemma:coupling_lemma}\eqref{mating_error_bound}.

\subsection{Construction of the coupling}
\label{subsection:decorating_trajectories_with_labels}
	
We are given $\mathcal{X}^{v,T} \sim \mathcal{P}_{v,T} $ (cf.\ Definition \ref{def_PPP_of_worms}). Using $\mathcal{X}^{v,T}$ and some additional i.i.d.\ randomness on the vertex set $V$ of $G$ as ingredients, we will  create a point process  $\mathcal{X}^{v/2,2T}$ with distribution $\mathcal{P}_{v/2,2T} $ which is coupled to $\mathcal{X}^{v,T}$ in a way that the two point processes are close to each other (locally) with high enough probability.
	
\medskip

Given  the PPP $\mathcal{X}^{v,T}=\sum_{i \in I} \delta_{w_i}  \sim \mathcal{P}_{v,T}$, let us toss a fair coin for each finite trajectory $w_i, i \in I$ and note that this can be done in a factor of i.i.d.\ fashion.
Using this ``coloring'' we can write
\begin{equation}\label{X1_X2_sum_is_XvT}
\mathcal{X}^{v,T}= \mathcal{X}_1 + \mathcal{X}_2,
\end{equation}
where $\mathcal{X}_1$ and  $\mathcal{X}_2$ are i.i.d.\ with distribution $\mathcal{P}_{v/2, T}$. Let us introduce the notation
\begin{equation}
    \label{notation:copies_of_half_worms}
    \mathcal{X}_1 := \sum_{i \in I_1} \delta_{ w_i^1 }, \qquad
    \mathcal{X}_2 :=  \sum_{i \in I_2} \delta_{ w_i^2 }.
\end{equation}
We will  pair (most of) the trajectories of $\mathcal{X}_1$ to (most of) the trajectories of  $\mathcal{X}_2$ in a way that the matched pairs (mostly) look like the first and the second halves of a trajectory of a copy of $\mathcal{X}^{v/2,2T}$. Recalling the notation introduced in \eqref{notation:initial_terminal}, we define
\begin{equation}\label{R_1_R_2_from_X_1_X_2}
  \mathcal{R}_1:= \terminal\left(  \mathcal{X}_1 \right), \qquad \mathcal{R}_2:= \initial\left(  \mathcal{X}_2 \right).
\end{equation}
In words: $\mathcal{R}_1$ is the point process of the terminal points of the trajectories of $\mathcal{X}_1$ and $\mathcal{R}_2$ is the point process of the initial points of the trajectories of $\mathcal{X}_2$.  Let us also make the following observation, which follows from Claim \ref{claim_worms_random_walks}.

\begin{claim}[Pinning down one endpoint of trajectories] \label{claim_conditional_indep_of_X_12_given_R_12}
Conditional on $\mathcal{R}_1$ and $\mathcal{R}_2$ (cf.\ \eqref{R_1_R_2_from_X_1_X_2}),
\begin{enumerate}[(i)]
\item\label{balaton1}  the trajectories of $\mathcal{X}_2$   are
 distributed as the first $T-1$ steps of a simple random walk  on $G$ with the points of the point process $\mathcal{R}_2$  as starting points,
 \item\label{balaton2}  the time-reversals of the trajectories of $\mathcal{X}_1$   are
 distributed as the first $T-1$ steps of a simple random walk  on $G$ with the points of the point process $\mathcal{R}_1$  as starting points,
 \item\label{balaton3} all of these random walk trajectories are independent of each other.
 \end{enumerate}
\end{claim}

Both $\mathcal{R}_1$ and $\mathcal{R}_2$ (cf.\ \eqref{R_1_R_2_from_X_1_X_2}) have law $\mathcal{P}_{v/2,1}$ by Claim
    \ref{claim:same_law_after_RW_shifts}\eqref{nyassz}, moreover $\mathcal{R}_1$ and $\mathcal{R}_2$ are independent (since the same holds for $\mathcal{X}_1$ and $\mathcal{X}_2$). Thus, if we choose $\alpha:=v/2$ and
     \begin{equation}\label{optimal_L_choice}
       L:= \left\lceil T^{4/7} \cdot v^{-2/7} \right\rceil,
     \end{equation}
     then we can apply Lemma \ref{lemma:matching_iid_PPP_on_V}  using $\mathcal{R}_1$ and $\mathcal{R}_2$ (as well as some auxiliary i.i.d.\ randomness on  $V$) as inputs. The output is a Poisson point process $\mathcal{R}$ on $V \times V$ with intensity measure  $\nu^{(2)}(\{ (x,y)\} )= \frac{v}{2} \cdot p_L(x,y)$. As we will see, the seemingly arbitrary choice of $L$ in \eqref{optimal_L_choice} will turn out to minimize the value of an error term.

By the definition \eqref{random_walk_transitions} of $p_L(x,y)$,  one can use
auxiliary i.i.d.\ randomness on $V$ to create in a factor of i.i.d.\ fashion a Poisson point process
\begin{equation}\label{X_star_satisfies}
 \mathcal{X}' \sim \wormsppp_{v/2,L+1} \quad
\text{satisfying} \quad
  \initial\left( \mathcal{X}'\right) = \textbf{p}_1(\mathcal{R}), \qquad  \terminal\left( \mathcal{X}'\right) = \textbf{p}_2(\mathcal{R}).
    \end{equation}
 In words: we connect the  vertices $x_i$ and  $y_i$ of each pair $(x_i,y_i)$ of $\mathcal{R}$ with a random walk trajectory that performs $L$ steps (i.e.,  has length $L+1$), starts at $x_i$,  and is conditioned to be at $y_i$ at time $L$. We omit the technical details of this construction, which uses Lemma \ref{lemma_atlas_factor_of_iid} just like in the case of the proof of Claim \ref{claim_QW_exists_then_fiid_triv}.
 
\noindent Note that one can write $ \mathcal{X}'= \sum_{i \in I_1} \delta_{w'_i} $, where $I_1$ is the same index set as the one that appears in $\mathcal{X}_1=\sum_{i \in I_1} \delta_{ w_i^1 }$ (cf.\ \ref{notation:copies_of_half_worms}), since
\begin{equation}\label{X_prime_indexed_by_I1}
\initial\left( \mathcal{X}' \right) \stackrel{ \eqref{X_star_satisfies} }{=}  \textbf{p}_1(\mathcal{R}) \stackrel{ \text{Lemma \ref{lemma:matching_iid_PPP_on_V}\eqref{first_projection_is_R1}} }{=} \mathcal{R}_1
\stackrel{ \eqref{R_1_R_2_from_X_1_X_2}}{=} \terminal\left(  \mathcal{X}_1 \right).
\end{equation}
We have created $\mathcal{R}$ and then $\mathcal{X}'$ from $\mathcal{R}_1$ and $\mathcal{R}_2$ using auxiliary i.i.d.\ randomness, thus we can use Claim \ref{claim_conditional_indep_of_X_12_given_R_12} to make the following observation.

\begin{claim}\label{claim_XXX_condindep_given_RR}
 $\mathcal{X}_1$, $\mathcal{X}_2$ and $\mathcal{X}'$ are conditionally independent of each other given $\mathcal{R}_1$ and $\mathcal{R}_2$.
\end{claim}

One goal of Lemma \ref{lemma:matching_iid_PPP_on_V} was to create $\mathcal{R}$ in a way that the difference  between $\mathcal{R}_2$ and $\textbf{p}_2(\mathcal{R}) $ is small.
Let us define the point measure $\mathcal{R}_2^m \in \wormspm(V)$ by
\begin{equation}\label{common_point_measure}
\mathcal{R}_2^m(\{x\}):=\mathcal{R}_2(\{x\}) \wedge \textbf{p}_2(\mathcal{R})(\{x\}), \qquad x \in V.
\end{equation}
We will see that $\mathcal{R}_2^m(\{x\})$ is the number of trajectories of $\mathcal{X}_2$ with starting point $x$ that are matched to a trajectory from $\mathcal{X}_1$.
In the next lemma, the superscripts $m$ and $u$ stand for ``matched'' and ``unmatched'', respectively. More specifically, $\mathcal{X}_2^m$ and $\mathcal{X}_2^u$ respectively denote the point process of matched and unmatched trajectories of $\mathcal{X}_2$, while the point process $\widehat{\mathcal{X}}_2^u$ consists of new trajectories that serve as continuations of the unmatched trajectories of $\mathcal{X}_1$. 

\begin{lemma}[Rewiring]\label{lemma_rewiring}
 With the use of auxiliary i.i.d.\ randomness on $V$ one can create   point processes
$\mathcal{X}_2^u, \mathcal{X}_2^m, \widehat{\mathcal{X}}_2^u \in  \wormspm(W_T) $ in a factor of i.i.d.\ fashion that satisfy the following properties.
\begin{enumerate}[(i)]
\item\label{bereny1} $\initial\left( \mathcal{X}_2^u \right)= \mathcal{R}_2-\mathcal{R}_2^m$, and
conditional on $\mathcal{R}_2-\mathcal{R}_2^m$,  the trajectories of $\mathcal{X}_2^u$   are
 distributed as the first $T-1$ steps of a simple random walk  on $G$ with the points of the point process $\mathcal{R}_2-\mathcal{R}_2^m$  as starting points.
\item \label{bereny2} $\mathcal{X}_2=\mathcal{X}_2^u+\mathcal{X}_2^m$.
\item \label{bereny3} $\initial(\widehat{\mathcal{X}}_2^u)=\textbf{p}_2(\mathcal{R})-\mathcal{R}_2^m$, and conditional on $\textbf{p}_2(\mathcal{R})-\mathcal{R}_2^m$,  the trajectories of $\widehat{\mathcal{X}}_2^u$   are
 distributed as the first $T-1$ steps of a simple random walk  on $G$ with the points of the point process $\textbf{p}_2(\mathcal{R})-\mathcal{R}_2^m$  as starting points.
 \item \label{bereny4} If we define
 \begin{equation}\label{wormppp_hat_X2}
 \widehat{\mathcal{X}}_2:=\widehat{\mathcal{X}}_2^u+\mathcal{X}_2^m,
 \end{equation}
 then, conditional on $\mathcal{X}'$ and $\mathcal{X}_1$, the trajectories of $\widehat{\mathcal{X}}_2$ are
 distributed as the first $T-1$ steps of a  random walk  on $G$ with the points of the point process $\textbf{p}_2(\mathcal{R})$  as starting points.
\end{enumerate}
\end{lemma}

\begin{proof} We can use auxiliary i.i.d.\ randomness on $V$ to split the index set $I_2$ (cf.\ \eqref{notation:copies_of_half_worms}) in a factor of i.i.d.\ fashion
 into the disjoint union of $I_2^m$ and $I_2^u$ such that if we define the point processes
 \begin{equation}\label{split_ppp_X2_into_matced_unmatched}
 \mathcal{X}_2^m:=\sum_{i \in I_2^m} \delta_{w_i^2}, \qquad \mathcal{X}_2^u:= \sum_{i \in  I_2^u} \delta_{w_i^2},
 \end{equation}
 then we have $\initial\left(  \mathcal{X}_2^m \right)= \mathcal{R}_2^m$ and $\initial\left( \mathcal{X}_2^u \right)= \mathcal{R}_2-\mathcal{R}_2^m$ (cf.\ \eqref{R_1_R_2_from_X_1_X_2}). As a matter of fact, such a splitting of the index set $I_2$ has already been performed in the proof of Lemma \ref{lemma:matching_iid_PPP_on_V}.
 
 The statement \eqref{bereny1} follows from  Claims \ref{claim_conditional_indep_of_X_12_given_R_12} and \ref{claim_XXX_condindep_given_RR}.
 The identity \eqref{bereny2} follows from \eqref{split_ppp_X2_into_matced_unmatched}.

 Note that one can write $ \mathcal{R}= \sum_{i \in I_1} \delta_{(x_i,y_i)}$, where $I_1$ is the same index set as the one that appears in $\mathcal{X}_1=\sum_{i \in I_1} \delta_{ w_i^1 }$ (cf.\ \ref{notation:copies_of_half_worms}), since \begin{equation*}
 \textbf{p}_1(\mathcal{R})\stackrel{ \text{ Lemma \ref{lemma:matching_iid_PPP_on_V}\eqref{first_projection_is_R1}} }{=} 
 \mathcal{R}_1
 \stackrel{ \eqref{R_1_R_2_from_X_1_X_2} }{=}
 \terminal(\mathcal{X}_1).
 \end{equation*}
  We can use auxiliary i.i.d.\ randomness on $V$ to split the index set $I_1$  in a factor of i.i.d.\ fashion
 into the disjoint union of $I_1^m$ and $I_1^u$ in such a way that if we define the point processes
  \begin{equation}\label{split_ppp_X1_into_matced_unmatched}
 \mathcal{R}^m:=\sum_{i \in I_1^m} \delta_{(x_i,y_i)}, \qquad \mathcal{R}^u:= \sum_{i \in  I_1^u} \delta_{(x_i,y_i)},
 \end{equation}
then we have $\textbf{p}_2(\mathcal{R}^m)=\mathcal{R}_2^m$ and $\textbf{p}_2(\mathcal{R}^u)=\textbf{p}_2(\mathcal{R})-\mathcal{R}_2^m$. Again, such a splitting of the index set $I_1$ has in fact already been performed in the proof of Lemma \ref{lemma:matching_iid_PPP_on_V}.

Let us define the point process
$\widehat{\mathcal{X}}_2^u =\sum_{i \in I_1^u} \delta_{w^2_i} $ using auxiliary i.i.d.\ randomness on the vertices of $V$ so that
 $w^2_i$   is
 distributed as the first $T-1$ steps of a simple random walk  on $G$ satisfying $\initial(w^2_i)=y_i$ for any $i \in I_1^u$, moreover the trajectories
 $w^2_i, i \in I_1^u$ are conditionally independent given their starting points.  We omit the technical details of this construction, which is carried out using Lemma \ref{lemma_atlas_factor_of_iid} just like in the case of the proof of Claim \ref{claim_QW_exists_then_fiid_triv}.

 Statement \eqref{bereny3} holds with this construction,
 since $\textbf{p}_2(\mathcal{R})-\mathcal{R}_2^m=\textbf{p}_2(\mathcal{R}^u)=\sum_{i \in I_1^u} \delta_{y_i}=\initial(\widehat{\mathcal{X}}_2^u)$.
  If we put statement \eqref{bereny3} together with
  Claims  \ref{claim_conditional_indep_of_X_12_given_R_12} and \ref{claim_XXX_condindep_given_RR}, we obtain that statement  \eqref{bereny4} also holds, since
\begin{equation}\label{X_hat_2_indexed_by_I1}
\initial\left( \widehat{\mathcal{X}}_2\right)\stackrel{ \eqref{wormppp_hat_X2} }{=}
\initial\left(\widehat{\mathcal{X}}_2^u\right)+
\initial\left(\mathcal{X}_2^m\right)
\stackrel{\eqref{bereny1} }{=}
\left(\textbf{p}_2(\mathcal{R})-\mathcal{R}_2^m\right) +
 \mathcal{R}_2^m = \textbf{p}_2(\mathcal{R}).
\end{equation}
\end{proof}
We have already seen that the index set $I_1$ (that was introduced to index the points of $\mathcal{X}_1$ in \eqref{notation:copies_of_half_worms}) 
can also be used to index the point processes $\mathcal{X}'$ (cf.\ \eqref{X_star_satisfies} and \eqref{X_prime_indexed_by_I1}) and  $\widehat{\mathcal{X}}_2$ (cf.\ \eqref{X_hat_2_indexed_by_I1}).
Also note that $\terminal(\mathcal{X}_1)=\initial(\mathcal{X}')$ by \eqref{X_prime_indexed_by_I1}, moreover $\terminal(\mathcal{X}')=\initial( \widehat{\mathcal{X}}_2 )$ by \eqref{X_star_satisfies} and Lemma \ref{lemma_rewiring}\eqref{bereny4}. As a consequence, we can stitch together the three trajectories of $\mathcal{X}_1$, $\mathcal{X}'$ and $\widehat{\mathcal{X}}_2$ indexed by the same $i \in I_1$ to form a trajectory $w''_i$ in $W_{2T+L-1}$. The stitching results in a point process that we denote  by
\begin{equation}\label{triple_concatenation_def_eq}
\mathcal{X}'':=\sum_{i \in I_1} \delta_{w''_i}  \in \wormspm(W_{2T+L-1}).
\end{equation}
 Recalling the notation introduced in
\eqref{initial_terminal_intervals} and \eqref{initial_terminal_interval_pp_def}, we have
\begin{equation}\label{Xvv_eleje_kozepe_vege}
  \initialinterval_T\left( \mathcal{X}'' \right)=\mathcal{X}_1, \qquad
   \terminalinterval_T\left( \mathcal{X}'' \right)=\widehat{\mathcal{X}}_2, \qquad
 \terminalinterval_{L+1} \left(\, \initialinterval_{T+L} \left( \mathcal{X}'' \right)\, \right) = \mathcal{X}'.
\end{equation}
Also note that we have
\begin{equation}\label{triple_concatenation_law}
     \mathcal{X}^{''} \sim \wormsppp_{v/2,2T+L-1},
\end{equation}
since the ``middle part'' $\mathcal{X}'$ has distribution  $\wormsppp_{v/2,L+1}$ by \eqref{X_star_satisfies}, and if we condition on $\mathcal{X}'$ then the ``backward parts'' (i.e., the trajectories of $\mathcal{X}_1$) are conditionally independent (time-reversed) random walk trajectories of length $T$  by Claim \ref{claim_conditional_indep_of_X_12_given_R_12}\eqref{balaton2} and \eqref{balaton3}, moreover the ``forward parts'' (i.e., the trajectories of $\widehat{\mathcal{X}}_2$) are conditionally independent random walk trajectories of length $T$ given $\mathcal{X}_1$ and $\mathcal{X'}$ by Lemma \ref{lemma_rewiring}\eqref{bereny4}.

Finally, let us define
\begin{equation}\label{the_final_product_def}
  \mathcal{X}^{v/2,2T}:= \initialinterval_{2T}\left( \mathcal{X}'' \right).
\end{equation}
Our goal is to show that the statements of Lemma \ref{lemma:coupling_lemma} hold for this point process $\mathcal{X}^{v/2,2T}$.


Lemma \ref{lemma:coupling_lemma}\eqref{mating_correct_distribution_2T} holds, as we now explain. On the one hand, we started with $\mathcal{X}^{v,T} \sim \mathcal{P}_{v,T} $ and the auxiliary randomness that we used was i.i.d.\ on $V$ and it was also independent of $\mathcal{X}^{v,T}$. On the other hand, $\mathcal{X}^{v/2,2T} \sim \mathcal{P}_{v/2,2T}$ follows from \eqref{triple_concatenation_law}
and Claim \ref{claim:same_law_after_RW_shifts}\eqref{nyissz}.

Lemma \ref{lemma:coupling_lemma}\eqref{mating_equivariant} holds since we used our input (i.e., $\mathcal{X}^{v,T}$ and the auxiliary i.i.d.\ randomness) in an equivariant fashion: if the input is transformed by some $\varphi \in \Gamma$ then the output $\mathcal{X}^{v/2,2T}$ will also be transformed by $\varphi$.

 It remains to prove  Lemma \ref{lemma:coupling_lemma}\eqref{mating_error_bound}, i.e., the upper bound on the probability of the event $\{\firstlastvis_K( \mathcal{X}^{v,T}) \neq \firstlastvis_K( \mathcal{X}^{v/2,2T}) \}$ that the local image of input $\mathcal{X}^{v,T}$ on $K$ and the local image of the output $\mathcal{X}^{v/2,2T}$ on $K$ are different. This is what we will do in Section \ref{subsection_error_bound_of _local_map}.

\subsection{Bounding the probability of local discrepancies}\label{subsection_error_bound_of _local_map}

Recall the notion of the localization map $	\firstlastvis_K$ from Definition \ref{projection_map_K}.

\begin{lemma}[Bounds on the probabilities of bad events] \label{lemma_bad_event_bounds}
  There exists a constant $C=C(G)$ such that for any $T \in \N_+$, $K \subset \subset V$ and any  $v \in [T^{-3/2}, T^2]$, we have
\begin{enumerate}[(i)]
  \item\label{yaris1} $\prob \left(
			\firstlastvis_K( \mathcal{X}^{v,T}) \neq \firstlastvis_K( \mathcal{X}_1 ) + \firstlastvis_K( \mathcal{X}_2 )	\right)=0$,
\item\label{yaris2}  $\prob \left(
			  \firstlastvis_K( \mathcal{X}_2 ) \neq   \firstlastvis_K( \widehat{\mathcal{X}}_2 )
			  \right) \leq 2 \cdot |K| \cdot T \cdot \sqrt{v \cdot p_{2L}(o,o)} $,
   \item\label{yaris3}  $\prob \left(
			  |\firstlastvis_K( \mathcal{X}' )| \neq  0 
			  \right) \leq \frac{v}{2} \cdot |K|\cdot (L+1)  $,
   \item\label{yaris4}  $\prob \left( 
   |\firstlastvis_K( \mathcal{X}' )| =  0, \;
			\firstlastvis_K( \mathcal{X}_1 ) + \firstlastvis_K( \widehat{\mathcal{X}}_2 )  \neq   \firstlastvis_K( \mathcal{X}'' )	
			\right) \leq C \cdot v \cdot \sqrt{T} \cdot |K|^2 $,
\item\label{yaris5} $\prob \left(
			  \firstlastvis_K( \mathcal{X}'' ) \neq   \firstlastvis_K( \mathcal{X}^{v/2,2T} )	
			  \right) \leq \frac{v}{2} \cdot |K|\cdot (L-1)  $.
\end{enumerate}
\end{lemma}	

Before we prove Lemma \ref{lemma_bad_event_bounds}, let us deduce the proof of 	Lemma \ref{lemma:coupling_lemma}\eqref{mating_error_bound} from it.

\begin{proof}[Proof of Lemma \ref{lemma:coupling_lemma}\eqref{mating_error_bound}] 
Let us introduce the events 
$A_1:=\{ 
			\firstlastvis_K( \mathcal{X}^{v,T}) \neq \firstlastvis_K( \mathcal{X}_1 ) + \firstlastvis_K( \mathcal{X}_2 ) \}$, 
			$A_2:=\{  \firstlastvis_K( \mathcal{X}_2 ) \neq   \firstlastvis_K( \widehat{\mathcal{X}}_2 ) \}$, $A_3:=\{  |\firstlastvis_K( \mathcal{X}' )| \neq  0  \}$, $A_4:=\{ 	\firstlastvis_K( \mathcal{X}_1 ) + \firstlastvis_K( \widehat{\mathcal{X}}_2 )  \neq   \firstlastvis_K( \mathcal{X}'' )	 \}$, $A_5:=\{  \firstlastvis_K( \mathcal{X}'' ) \neq   \firstlastvis_K( \mathcal{X}^{v/2,2T} )	 \}$. We first argue that if the events $A_1^c,A_2^c, A_4^c,A_5^c$ all occur then
			$\{\firstlastvis_K( \mathcal{X}^{v,T}) = \firstlastvis_K( \mathcal{X}^{v/2,2T}) \}$ also occurs. Indeed:
\begin{equation}
	\firstlastvis_K( \mathcal{X}^{v,T}) \stackrel{A_1^c}{=}	
	\firstlastvis_K( \mathcal{X}_1 ) + \firstlastvis_K( \mathcal{X}_2 ) \stackrel{A_2^c}{=} 	\firstlastvis_K( \mathcal{X}_1 ) + \firstlastvis_K( \widehat{\mathcal{X}}_2 )\stackrel{A_4^c}{=}  \firstlastvis_K( \mathcal{X}'' )
	\stackrel{A_5^c}{=}\firstlastvis_K( \mathcal{X}^{v/2,2T}).
	\end{equation}
Thus by De Morgan's laws and the union bound we obtain	
\begin{equation}
\mathbb{P}\left(\firstlastvis_K( \mathcal{X}^{v,T}) \neq \firstlastvis_K( \mathcal{X}^{v/2,2T}) \right) \leq
\mathbb{P}(A_1)+\mathbb{P}(A_2)+\mathbb{P}(A_4)+\mathbb{P}(A_5).
\end{equation}
Noting that $\mathbb{P}(A_4) \leq \mathbb{P}(A_3)+ \mathbb{P}(A_3^c \cap A_4) $, we obtain that 
 $\prob\left( \firstlastvis_K( \mathcal{X}^{v,T}) \neq \firstlastvis_K( \mathcal{X}^{v/2,2T}) \right)$ can be upper bounded by the sum of the
 terms on the r.h.s.\ of \eqref{yaris1}--\eqref{yaris5}. Using that our assumption $v \leq T^2$ and the definition \eqref{optimal_L_choice} of $L$ together imply $L+1 \leq 2L$, we obtain that  the sum of the
 terms on the r.h.s.\ of \eqref{yaris1}--\eqref{yaris5} is less than or equal to
 \begin{equation}\label{error_sum_L_incl}
   2 \cdot |K| \cdot T \cdot \sqrt{v \cdot p_{2L}(o,o)}  + 2 \cdot |K| \cdot v \cdot L+  C \cdot |K|^2  \cdot v \cdot \sqrt{T}.
 \end{equation}
Noting that $p_{2L}(o,o)\leq C \cdot L^{-3/2}$ by Lemma \ref{lemma:heat_kernel_upper}, we can use the definition \eqref{optimal_L_choice} of $L$ to see that the first two terms of \eqref{error_sum_L_incl} are both upper bounded by a constant multiple of $|K| \cdot T^{4/7} \cdot v^{5/7} $,  thus the upper bound stated in \eqref{eq:coupling_error} indeed holds. The proof of Lemma \ref{lemma:coupling_lemma}\eqref{mating_error_bound} is complete.
 \end{proof}

The proof of Lemma \ref{lemma:coupling_lemma} is complete (given Lemma \ref{lemma_bad_event_bounds}).

\begin{proof}[Proof of Lemma \ref{lemma_bad_event_bounds}\eqref{yaris1}] We have $\mathcal{X}^{v,T}= \mathcal{X}_1 + \mathcal{X}_2$ by \eqref{X1_X2_sum_is_XvT}, thus
$\firstlastvis_K( \mathcal{X}^{v,T}) = \firstlastvis_K( \mathcal{X}_1 ) + \firstlastvis_K( \mathcal{X}_2 )$ follows from the definition of $\firstlastvis_K: \wormspm( \wormspace) \to \wormspm( \fwormspace(K))$, cf.\  Definition \ref{projection_map_K}.
\end{proof}

\begin{proof}[Proof of Lemma \ref{lemma_bad_event_bounds}\eqref{yaris2}] Recall from the statement of Lemma \ref{lemma_rewiring} that we have $\mathcal{X}_2=\mathcal{X}_2^u+\mathcal{X}_2^m$ and $\widehat{\mathcal{X}}_2=\widehat{\mathcal{X}}_2^u+\mathcal{X}_2^m$, thus by Definition \ref{projection_map_K} and the union bound we obtain
\begin{equation}\label{X2u_and_hatX_2_u_on_rhs}
  \prob \left(
			  \firstlastvis_K( \mathcal{X}_2 ) \neq   \firstlastvis_K( \widehat{\mathcal{X}}_2 )	\right) \leq
\prob \left( \mathcal{X}_2^u(W_T(K)) \neq 0
			  	\right) + \prob \left( \widehat{\mathcal{X}}_2^u(W_T(K)) \neq 0
			  	\right).
\end{equation}
In order to bound the first term on the r.h.s.\ of \eqref{X2u_and_hatX_2_u_on_rhs}, first observe that for any $ x \in V$ we have
\begin{equation}
 \initial(\mathcal{X}_2^u)(\{x\})  \stackrel{\text{Lemma \ref{lemma_rewiring}\eqref{bereny1}} }{=}  \mathcal{R}_2(\{x\})-\mathcal{R}_2^m(\{x\})  \stackrel{ \eqref{common_point_measure} }{=}
\Big( \mathcal{R}_2(\{x\}) - \textbf{p}_2(\mathcal{R})(\{x\})\Big)_+,
\end{equation}
thus we obtain that for any $x \in V$ we have $\mathbb{E}\left[ \initial(\mathcal{X}_2^u)(\{x\}) \right] \leq \sqrt{v \cdot p_{2L}(o,o) }$ using Lemma \ref{lemma:matching_iid_PPP_on_V}\eqref{matching_ppp_error_bound_from_surface} and the fact that the law of $\initial(\mathcal{X}_2^u)$ is invariant under the action of $\Gamma$. Putting this together with
Lemma \ref{lemma_gen_pp_density_worm_hit_bound} and Lemma \ref{lemma_rewiring}\eqref{bereny1},
 we obtain
$\prob \left( \mathcal{X}_2^u(W_T(K)) \neq 0
			  	\right) \leq |K| \cdot T \cdot \sqrt{v \cdot p_{2L}(o,o)} $.
The second term of on the r.h.s.\ of \eqref{X2u_and_hatX_2_u_on_rhs} can be bounded analogously using Lemma \ref{lemma_rewiring}\eqref{bereny3},
Lemma \ref{lemma:matching_iid_PPP_on_V}\eqref{matching_ppp_error_bound_from_remaining} and Lemma \ref{lemma_gen_pp_density_worm_hit_bound}.
Plugging these bounds into \eqref{X2u_and_hatX_2_u_on_rhs} we obtain Lemma \ref{lemma_bad_event_bounds}\eqref{yaris2}.
\end{proof}

\begin{proof}[Proof of Lemma \ref{lemma_bad_event_bounds}\eqref{yaris3}]
First observe that it follows from the definition of $\firstlastvis_K$ (cf.\ Definition \ref{projection_map_K}) that the event $\{ |\firstlastvis_K( \mathcal{X}' )| \neq  0 \}$ occurs if and only if a trajectory from $\mathcal{X}'$ hits $K$.
 Recalling from \eqref{X_star_satisfies}
that $\mathcal{X}' \sim \wormsppp_{v/2,L+1}$, the desired bound is just an application of Corollary \ref{corollary:prob_visit_KTv}.
\end{proof}

\begin{proof}[Proof of Lemma \ref{lemma_bad_event_bounds}\eqref{yaris4}]
Recall from \eqref{triple_concatenation_def_eq} the notation
$\mathcal{X}''=\sum_{i \in I_1} \delta_{w''_i}  \in \wormspm(W_{2T+L-1})$.
 Recall from \eqref{Xvv_eleje_kozepe_vege} that we have
  $\initialinterval_T\left( \mathcal{X}'' \right)=\mathcal{X}_1$,
  $\terminalinterval_T\left( \mathcal{X}'' \right)=\widehat{\mathcal{X}}_2$ and
 $\terminalinterval_{L+1} \left(\, \initialinterval_{T+L} \left( \mathcal{X}'' \right)\, \right) = \mathcal{X}'$.
 It follows from Definition \ref{projection_map_K} that we have
\begin{align}
\label{flXpp}
    \firstlastvis_K (\mathcal{X}'') &:= \sum_{i \in I_1} \delta_{\firstlastvis_K(w''_i) } \, \ind[\, w''_i \in \fwormspace(K)\, ],
    \\
    \firstlastvis_K (\mathcal{X}_1) & := \sum_{i \in I_1} \delta_{\firstlastvis_K(\initialinterval_T(w''_i)) } \, \ind[\, \initialinterval_T(w''_i) \in \fwormspace(K)\, ],
    \\
\label{flhatX2}
    \firstlastvis_K (\widehat{\mathcal{X}}_2 ) &:= \sum_{i \in I_1} \delta_{\firstlastvis_K(\terminalinterval_T(w''_i)) } \, \ind[\, \terminalinterval_T(w''_i) \in \fwormspace(K) \, ].
\end{align}
 In words, the event $\{ |\firstlastvis_K( \mathcal{X}' )| =  0\}$ that appears in \eqref{yaris4} means that the trajectories of $\mathcal{X}'$ are not in $\fwormspace(K)$.
 Now let us note that if $i \in I_1$ and the event
 \begin{equation}
 \Big\{ \terminalinterval_{L+1} \left(\, \initialinterval_{T+L} \left( w''_i \right)\, \right) \notin \fwormspace(K) \Big\} \cap
 \left( \Big\{ \initialinterval_T\left( w''_i \right) \in \fwormspace(K) \Big\} \cap  \Big\{ \terminalinterval_T\left( w''_i \right) \in \fwormspace(K) \Big\}  \right)^c
 \end{equation}
 occurs then we have
  \begin{multline}\label{delta_ind_identity}
   \delta_{\firstlastvis_K(w''_i) } \, \ind[\, w''_i \in \fwormspace(K)\, ]=\\
   \delta_{\firstlastvis_K(\initialinterval_T(w''_i)) } \, \ind[\, \initialinterval_T(w''_i) \in \fwormspace(K)\, ]
   +
   \delta_{\firstlastvis_K(\terminalinterval_T(w''_i)) } \, \ind[\, \terminalinterval_T(w''_i) \in \fwormspace(K) \, ],
  \end{multline}
 because  if both terms on the r.h.s.\ vanish then the l.h.s.\ also vanishes, but if exactly one term on the r.h.s.\ is nonzero then it is equal to the l.h.s.
 If \eqref{delta_ind_identity} holds for all $i \in I_1$ then by \eqref{flXpp}--\eqref{flhatX2}  we obtain $\firstlastvis_K( \mathcal{X}_1 ) + \firstlastvis_K( \widehat{\mathcal{X}}_2 )  =   \firstlastvis_K( \mathcal{X}'' )$.
 As a consequence, the bad event of \eqref{yaris4} can only occur if there is an $i \in I_1$ such that both $\initialinterval_T(w''_i)$ and $\terminalinterval_T(w''_i)$ hit $K$.
We can thus bound
\begin{multline} \prob \left(  \firstlastvis_K( \mathcal{X}' ) =  0, \;
			\firstlastvis_K( \mathcal{X}_1 ) + \firstlastvis_K( \widehat{\mathcal{X}}_2 )  \neq   \firstlastvis_K( \mathcal{X}'' )	\right) \leq
\prob \left[ \, \exists \, i \in I_1\, : \, \initialinterval_T(w''_i), \terminalinterval_T(w''_i) \in \fwormspace(K)\,  \right] \leq
 \\ \sum_{x,y \in K}
 \sum_{k=0}^{T-1} \sum_{\ell=T+L}^{2T+L-1}
		 \prob \left[ \, \exists\, i \in  I_1\, : \,  w''_i(k)=x, \, w''_i(\ell)=y \, \right]
\stackrel{(*)}{\leq}
\sum_{x,y \in K}
		  \sum_{k=0}^{T-1} \sum_{\ell=T+L}^{2T+L-1}
		 \frac{v}{2} \cdot p_{\ell-k}(x,y) \leq \\
\frac{v}{2} \cdot \sum_{x,y \in K} \sum_{n=1}^{2T+L} n \cdot p_n(x,y)
		\stackrel{ \eqref{eq:heat_kernel_upper}}{\leq}
		 \frac{v}{2} \cdot \sum_{x,y \in K} \sum_{n = 1}^{2T+L} C \cdot n^{- 1 \slash 2}
		\stackrel{(**)}{\leq}
		v \cdot  |K|^2 \cdot  C  \cdot \sqrt{T},
\end{multline}
where in $(*)$ we used Markov's inequality,
 $\mathcal{X}^{''} \sim \wormsppp_{v/2,2T+L-1}$ (cf.\ \eqref{triple_concatenation_law}) and Claim \ref{claim:same_law_after_RW_shifts}, and
 in $(**)$ we used that our assumption $v \geq T^{-3/2}$ and the definition \eqref{optimal_L_choice} of $L$ imply $L \leq T$.
\end{proof}

\begin{proof}[Proof of Lemma \ref{lemma_bad_event_bounds}\eqref{yaris5}]
Recall from \eqref{triple_concatenation_def_eq} the notation
$\mathcal{X}''=\sum_{i \in I_1} \delta_{w''_i}  \in \wormspm(W_{2T+L-1})$ and also recall from \eqref{the_final_product_def} that
  $\mathcal{X}^{v/2,2T}= \initialinterval_{2T}\left( \mathcal{X}'' \right)$. Let $i \in I_1$ and note that if $\terminalinterval_{L-1}(w''_i)$ does not hit
  $K$ then $\firstlastvis_K(w''_i)= \firstlastvis_K( \initialinterval_{2T}(w''_i) )$, thus $\prob \left(
			  \firstlastvis_K( \mathcal{X}'' ) \neq   \firstlastvis_K( \mathcal{X}^{v/2,2T} )	\right)$ is less than or equal to the probability  that a
  trajectory from $\terminalinterval_{L-1}(\mathcal{X}'')$ hits $K$. Observing that $\terminalinterval_{L-1}(\mathcal{X}'') \sim \wormsppp_{v/2,L-1} $ by
  \eqref{triple_concatenation_law} and Claim \ref{claim:same_law_after_RW_shifts}\eqref{nyissz}, the desired bound follows from Corollary \ref{corollary:prob_visit_KTv}.
\end{proof}

\appendix

\section{Heat kernel estimate}

Although the result of Lemma \ref{lemma:heat_kernel_upper} is well-known (as a mathematical folklore), since we were unable to find a written reference satisfying all of our needs, in this appendix we shall present a quick access to its proof, collecting all the necessary results together (in their appropriate forms) and filling in the missing gaps.

We will proceed as follows. In Subsection \ref{appendix:preliminaries} we will recall some notions and definitions from (geometric) group theory that we will use extensively during the argument.
After that, in Subsection \ref{appendix:growth_of_groups} we will take a short excursion in the structure theory of groups and will argue that any graph which we will encounter (i.e., satisfying the conditions of Lemma \ref{lemma:heat_kernel_upper}) has at least cubic growth rate.
We have put this argument in a separate subsection since it uses deep theorems from the structure theory of groups, which have little to do with probability theory.
The final steps of the proof are collected in Subsection \ref{appendix:isoperimetry_return_probabilities}:
First, using a general result of \cite{CSC93} we will show that such a cubic growth rate implies at least $3$-dimensional isoperimetry, then we will substitute this isoperimetric inequality into the machinery provided by \cite{MP05} to produce the desired upper bound on the transition probabilities.

From now on, if we don't say otherwise, $G$ will always denote a connected, locally finite infinite graph. Recall that such a graph is called (vertex-)transitive, if the group of its graph automorphisms acts transitively on its vertex set. Recall also that a graph is called transient if a simple random walk on it is transient (and is called recurrent otherwise).

\subsection{Preliminaries}
\label{appendix:preliminaries}

Consider two functions $f, g \, : \, \N \rightarrow \R_+$. Let us introduce the notation $f \preceq g$ meaning that there exist some constants $C, \alpha > 0$ such that $f(n) \leq C g(\alpha n)$ holds for all $n > 0$. We say that $f$ and $g$ are equivalent, denoted by $f \equivalent g$, if $f \preceq g$ and $f \succeq g$.

\begin{definition}
	Let $f \, : \, \N \rightarrow \R_+$. The function $f$ is called polynomial if there exists some $\beta > 0$ such that $f(n) \equivalent n^{\beta}$. Furthermore, it is called superpolynomial if $n^{\beta} \preceq f(n)$ for all $\beta > 0$.
\end{definition}

\noindent
For example, according to this definition, the function $n^{\pi}$ is considered to be polynomial and the function $n^{\log (\log (n))}$ is superpolynomial.

Most of the time graphs and groups are only considered up to quasi-isometries when their large scale geometric properties, such as the growth rate, are being examined. We recall this notion in the next definition, albeit its definition can be found in any textbook considering any aspect of asymptotic geometry (for example \cite{P22}, \cite{LP16}, \cite{W00} or \cite{G14}).

\begin{definition}
    Suppose that $(X_1, \distance_1)$ and $(X_2, \distance_2)$ are metric spaces. A map $\Phi \, : \, X_1 \rightarrow X_2$ is called a  quasi-isometry (or a rough isometry) if there exist positive constants $\alpha$ and $\beta$ such that the following two conditions are met:
    \begin{enumerate}[(i)]
        \item for all $x, y \in X_1$, we have $\alpha^{-1} \cdot \distance_1(x,y) - \beta \leq \distance_2( \Phi(x), \Phi(y) ) \leq \alpha \cdot \distance_1(x,y) + \beta$;
        \item for each $y \in X_2$, there is some $x \in X_1$ such that $\distance_2(y, \Phi(x)) < \beta$.
    \end{enumerate}
    If there is such a quasi-isometry between the two metric spaces then we say that they are quasi-isometric (or roughy isometric).
\end{definition}

\noindent
It is easy to check that being quasi-isometric is an equivalence relation.

As it was mentioned before, in the first step of the proof of Lemma \ref{lemma:heat_kernel_upper} we will use some interplay between groups and graphs. Therefore, it is worth it to recall the notion of Cayley graphs.

Given a finitely generated group $\Gamma$ with a symmetric generating set $S$, we define its (right) Cayley graph (with respect to the given generating set) $\Cayley(\Gamma, S)$ to be the graph whose vertices are the elements of $\Gamma$ and whose edge set is
\begin{equation}
	E( \Cayley(\Gamma, S) ) :=
	\left\{ (x,y) \in \Gamma \times \Gamma \, : \, y = x s \text{ for some } s \in S \right\}.
\end{equation}
Let us note that there is a natural (vertex-)transitive action of $\Gamma$ on $\Cayley(\Gamma, S)$ given by the multiplication from the left, whence Cayley graphs are always transitive. 
However, transitivity does not necessarily imply that the graph is a Cayley graph of some group (the standard example of this is the so-called Petersen graph).

Finally, let us recall some definitions from elementary group theory.
Given a group $\Gamma$, by the index of a subgroup $H \subgrp \Gamma$ we mean the number its (left or right) cosets
\begin{equation}
	[\Gamma:H] := \left| \left\{ x H \, : \, x \in \Gamma \right\} \right|=
	\left| \left\{ H x \, : \, x \in \Gamma \right\} \right|.
\end{equation}
If the two sets on the right-hand side are also equal (not just in size), then the subgroup $H$ is called normal, which is denoted by $H \triangleleft \Gamma$.

If we are given two subgroups $H, K \subgrp \Gamma$, then their commutator is the generated subgroup
\begin{equation}
	[H,K] := \langle h k h^{-1} k^{-1} \, : \, h \in H, \, k \in K \rangle.
\end{equation}

\begin{definition}
	A group $\Gamma$ is called nilpotent (or $s$-step nilpotent) if the lower central series $\gamma_0(\Gamma) := \Gamma$, $\gamma_{i+1}(\Gamma) := [ \gamma_i( \Gamma), \Gamma]$ terminates in $\gamma_s(\Gamma) = \{ 1 \}$ in $s \in \N$ steps.
\end{definition}

\noindent
Since an Abelian group is $1$-step nilpotent, one can think about the class of nilpotent groups as the groups that are almost Abelian.

For further reference, let us note here that for the elements of the lower central series we have $\gamma_{i+1}(\Gamma) \triangleleft \gamma_i(\Gamma)$ and that the quotient group $\gamma_i(\Gamma) / \gamma_{i+1}(\Gamma)$ is always Abelian (for example see page 113 of \cite{R95}).

\subsection{Growth of graphs and structure of groups}
\label{appendix:growth_of_groups}

If we endow the transitive graph $G$ with the usual graph metric, denoted by $\distance(.,.)$, then the ball of radius $n \in \N$ centered at $x \in V(G)$ is $\ball(x,n) := \left\{ y \in V(G) \, : \, \distance(x,y) \leq n \right\}$. The function $\growth_G(n) := \left| \ball(o, n) \right|$ is called the (volume) growth function, where $o \in V(G)$ is chosen arbitrarily (due to the transitivity). A graph has (super)polynomial growth if its growth function is (super)polynomial.

Using the so-called Nash-Williams criterion (see page 37, equation (2.14) of \cite{LP16}) from electric network theory, one can immediately give a necessary condition on the growth function of a transient graph.

\begin{lemma}[\cite{W00}, Lemma 3.12]
	\label{lemma:at_least_quadratic_growth}
	If $G$ is a transitive and transient graph, then 
	\begin{equation*}
		\liminf_{n \to \infty} \growth_G(n)/n^2 = \infty.
	\end{equation*}
\end{lemma}

\noindent
In words, Lemma \ref{lemma:at_least_quadratic_growth} says that a transient, transitive graph must have a growth rate faster than quadratic. 
Not surprisingly, to arrive at the aformentioned at least cubic growth rate, we need much stronger results.

Another crucial observation is the following. By definition we have that if our transitive graph $G$ has  superpolynomial growth then for all $\alpha >0$ there exists a constant $c >0$ (maybe depending on $\alpha$) such that the growth function satisfies $c n^{\alpha} \leq \growth_G(n)$ for all $n \in \N$. In particular, this means we already have the desired at least cubic growth rate in this case. Consequently, we only need to focus on graphs with polynomial growth, or more precisely, we only need to show that a transient, transitive graph with polynomial growth has at least cubic growth rate. Fortunately, as the next result shows, we have a good understanding of the structure of these graphs.

\begin{theorem}[\cite{T85}, Theorem 1.]
	\label{theorem:Trofimov}
	If a transitive graph has polynomial growth then it is roughly isometric to a Cayley graph of a finitely generated group that has a nilpotent subgroup of finite index.
\end{theorem}

\noindent
Maybe it is worth mentioning that Theorem \ref{theorem:Trofimov} is an extension of the famous theorem of Gromov from \cite{G81} (see page $54$, Main theorem therein) about the structure of groups with polynomial growth. Let us also note here that the way we stated Theorem \ref{theorem:Trofimov} is just the corollary of the original statement.

To convince the reader that Theorem \ref{theorem:Trofimov} is indeed all we can wish for, it is reasonable to state the following lemma about the growth of a graph being invariant under quasi-isometries. Moreover, this next lemma can also be used to define the notion of volume growth for groups, or more precisely, for finitely generated groups, supplemented with the definition of Cayley graphs from the previous subsection.

\begin{lemma}[\cite{W00}, Lemma 3.13]
	\label{lemma:growth_invariant_under_rough_isometry}
	If the graphs $G$ and $G'$ are roughly isometric (as metric spaces) then their growth functions are equivalent.
\end{lemma}

\noindent
Indeed, it is an easy exercise to show that the Cayley graphs of a finitely generated group with respect to different generating sets are roughly isometric. Hence, it is meaningful to define the (volume) growth function of a finitely generated group as the growth function of one of its Cayley graphs.

With the growth function being defined for finitely generated groups, the other step to convince the reader about the usefulness of Theorem \ref{theorem:Trofimov} is a consequence of the following property.

\begin{lemma}[\cite{W00}, Lemma 3.14]
	\label{lemma:growth_finite_index_subgroup}
    Let $\Gamma$ be a finitely generated group with a finite index subgroup $\Gamma_1$. Then $\Gamma$ and $\Gamma_1$ have equivalent growth functions.
\end{lemma}

\noindent
Indeed,  Theorem \ref{theorem:Trofimov} combined with Lemmas \ref{lemma:growth_invariant_under_rough_isometry} and \ref{lemma:growth_finite_index_subgroup} tells us that to describe all the transitive graphs of polynomial growth it is enough to consider the (asymptotic) structure of finitely generated nilpotent groups. Fortunately, the latter was done nearly 50 years ago in \cite{B72} and \cite{G73}.

Recall from elementary group theory (or see \cite{R95}, Theorem 10.20) the fundamental structure theorem of finitely generated Abelian groups. One of its consequences is the following observation: if we are given a finitely generated Abelian group $\Gamma$ then there exists a nonnegative integer $d \in \N$ such that $\Gamma$ contains a subgroup of finite index isomorphic with the free Abelian group $\Z^d$. In this case, this integer $d$ is called the rank (or the free-rank) of the Abelian group and is denoted by $\rank(\Gamma)$.

Now, if we are given a finitely generated nilpotent group $\Gamma$ with lower central series $( \gamma_i(\Gamma) )_{i = 0}^{s}$ then, as we mentioned previously, the quotient groups $\gamma_{i}(\Gamma) / \gamma_{i+1}(\Gamma)$ are (finitely generated) Abelian groups. Consequently, it is meaningful to define the nonnegative integer
\begin{equation}
    \mathbf{d}(\Gamma) := \sum_{i = 1}^{s} i \, \rank \left( \Gamma_{i-1} / \Gamma_i \right).
\end{equation}

As the following result, the so-called Bass-Guivarc'h formula, says, this number $\mathbf{d}(\Gamma)$ can be used to describe the structure of a finitely generated nilpotent group. Let us mention here that the name comes from the fact that in \cite{G73} (see Theorem II.4. therein) an identical result was proved for topological groups.

\begin{theorem}[\cite{B72}, Theorem 2.]
    If $\Gamma$ is a finitely generated nilpotent group then we have
    \begin{equation}
         \growth_{\Gamma}(n) \equivalent n^{\mathbf{d}(\Gamma)}.
    \end{equation}
\end{theorem}

An immediate consequence of this result, using Lemma \ref{lemma:growth_finite_index_subgroup}, is the following: if a finitely generated group has a finite-index nilpotent subgroup then it has polynomial growth. Note that this consequence (accompanied with Theorem \ref{theorem:Trofimov}) gives a characterization of transitive graphs with polynomial growth. Moreover, due to Lemma \ref{lemma:at_least_quadratic_growth}, this result also implies the desired at least cubic growth in the polynomial growth regime and hence we just arrived at the following corollary, which can be considered as the first step in the proof of Lemma \ref{lemma:heat_kernel_upper}.

\begin{corollary}
	\label{coro:transient_at_least_cubic}
	If $G$ is a transient transitive graph then we have $\growth_G(n) \succeq n^3$, i.e., such a graph has at least cubic growth.
\end{corollary}

\subsection{Isoperimetry and transition probabilities}
\label{appendix:isoperimetry_return_probabilities}

Although most of the authors use the notion of edge boundary in the following definition of isoperimetry, due to the local finiteness of the graph and the fact that we don't strive for the optimal constant, let us use something else. Namely, given a finite subset of vertices $K \fsubset V(G)$, we define its inner vertex boundary $\boundary K$ to be the set of vertices inside $K$ with at least one neighbour outside $K$.

\begin{definition}
	Let $\psi \, : \, \R_+ \rightarrow \R_+$ be a non-decreasing function. We say that $G$ satisfies the $\psi$-isoperimetric inequality $\isoperimetry_{\psi}$ if there exists a constant $c > 0$ such that
	\begin{equation}
	| \boundary K | \geq c \psi \left( | K| \right)
	\end{equation}
	holds for any finite subset of vertices $K \fsubset V(G)$. In particular, if $\psi(t) = t^{1 - 1/d}$ (with some $1 \leq d \leq \infty$), then we speak about $d$-dimensional isoperimetric inequality, denoted by $\isoperimetry_d$.
\end{definition}

Despite the fact that we are mostly interested in (at least) $3$-dimensional isoperimetry, it is worth mentioning that for example the $d$-dimensional lattice $\Z^d$ has $\isoperimetry_d$, which was originally proved in \cite{BL91} using compression by gravity.
However, a better known approach is to use the more general result of Theorem \ref{theorem:Coulhon_Saloff-Coste} below.

Now, according to the plan, we would like to use Corollary \ref{coro:transient_at_least_cubic} to obtain an at least $3$-dimensional isoperimetry for our graphs in question. This can be achieved by the next result, which was originally proved by Coulhon and Saloff-Coste in \cite{CSC93}. The referenced version of the statement is a combination of Lemma 10.46 and Proposition 8.14 from \cite{LP16}.

\begin{theorem}
	\label{theorem:Coulhon_Saloff-Coste}
	Define the inverse growth rate of the transitive graph $G$ by
	\begin{equation}
		\rho(n) := \min \left\{ r \, : \, \growth_G(r) \geq n \right\},
	\end{equation}
	that is, $\rho(n)$ is the smallest radius of a ball in $G$ that contains at least $n$ vertices. Then for all finite subset of vertices $K \fsubset V(G)$ we have
	\begin{equation}
		\frac{| \boundary K | }{ |K| } \geq \frac{1}{ 2 \rho \left( 2|K| \right) }.
	\end{equation}
\end{theorem}

Noting that $\growth_G(n) \succeq n^3$ implies $\rho(n) \leq C n^{1/3}$, this instantly yields the desired corollary, which can be considered as the second step towards to the proof of Lemma \ref{lemma:heat_kernel_upper}.

\begin{corollary}
	\label{coro:cubic_growth_implies_3dim_isoperimetry}
	If $G$ is a transitive graph which has at least cubic growth, then there exists a constant $c > 0$ such that for any finite subset of vertices $K \fsubset V(G)$ we have
	\begin{equation}
		|\boundary K | \geq c |K|^{ 1 - 1/3 },
	\end{equation}
	that is, $G$ satisfies the $3$-dimensional isoperimetric inequality $\isoperimetry_3$.
\end{corollary}

For the final step, we would like to use this isoperimetry to obtain the desired upper bound on the transition probabilities of random walks. Let us note that while results of this kind are originated in \cite{V85} (see also \cite{VSCC92}), where the authors used tools from functional analysis, what we will present here is the result of \cite{MP05}, which follows from a purely probabilistic argument.

Recall that we only consider a symmetric random walk on the given transitive, transient graph $G$, which is a reversible Markov chain. Moreover, since $G$ is assumed to be transitive, and whence all the degrees are the same - denoted by $D$ - the constant function $\pi(x) = D$ ($x \in V(G)$) is a reversible measure.

Now, for $x,y \in V(G)$ let us introduce $Q(x,y) := \pi(x) p_1(x,y) = D \cdot p_1(x,y)$, where recall from \eqref{random_walk_transitions} that $p_1(x,y)$ denotes the one-step transition probability from $x$ to $y$.
Since the random walk is simple, note that $Q$ is just the adjacency matrix of the graph.

This can be extended for any subsets $K, K' \subseteq V(G)$ as $Q(K,K') := \sum_{x \in K, y \in K'} Q(x,y)$. Since for a finite set $K \fsubset V(G)$ the quantities $\pi(K)$ and $Q(K, K^c)$ are just a more general way to measure the size and the size of the boundary respectively, the next notion can be considered as slight extension of the isoperimetric dimension.

\begin{definition}
	The isoperimetric profile of a reversible Markov chain on the transitive graph $G$ is
	\begin{equation}
		\isoprof(r) := \inf \left\{ \frac{Q(K, K^c)}{ \pi(K) } \, : \, K \fsubset V(G), \, \pi(K) \geq r \right\}, \qquad r \geq D,
	\end{equation}
	where $\pi$ is the reversible measure and $D$ is the degree of a vertex.
\end{definition}

\noindent
Observe that the $d$-dimensional isoperimetry $\isoperimetry_d$ implies that there exists some constant $c > 0$ such that $\isoprof(r) \geq c r^{-1/d}$.

As it was mentioned earlier, a version of the next theorem was first proved in \cite{MP05} (see Theorem 2. therein) using a specific spatial martingale called the evolving set process. The form below, which is easier to use to achieve our goals, is Theorem 6.31 from \cite{LP16}.

\begin{theorem}
	\label{theorem:Morris_Peres}
	Suppose that we are given a reversible Markov chain with reversible measure $\pi$ on the graph $G$. If $x,y \in V(G)$, $n \in \N$ and $\veps > 0$ satisfies
	\begin{equation}
		n \geq 1 + \int_{\min \{ \pi(x), \pi(y) \}}^{4 / \veps} \frac{16}{u \cdot \isoprof^2(u)} \, \mathrm{d}u,
	\end{equation}
	then we have
	\begin{equation}
		p_n(x,y) \leq \pi(x) \cdot \veps.
	\end{equation}
\end{theorem}

\noindent
As the following corollary shows, Theorem \ref{theorem:Morris_Peres} indeed gives us the required implication from lower bounds on the isoperimetric profile to upper bounds on transition probabilities and whence the final step in the proof of Lemma \ref{lemma:heat_kernel_upper}. Let us note here that the proof of Corollary \ref{coro:3dim_isoperimetry_implies_heat_kernel_decay} is almost identical to that of Corollary 6.32(ii) from \cite{LP16}.

\begin{corollary}
	\label{coro:3dim_isoperimetry_implies_heat_kernel_decay}
	Suppose that we have a simple random walk on a transitive graph $G$ satisfying $d$-dimensional isoperimetry. Then there exists a constant $C = C(G) < \infty$ such that
	\begin{equation}
		p_n(x,y) \leq C \cdot n^{-d / 2}
	\end{equation}
	holds for all $x,y \in V(G)$.
\end{corollary}

\begin{proof}
	As we mentioned previously, $d$-dimensional isoperimetry $\isoperimetry_d$ implies that there exists $c > 0$ such that $\isoprof(r) \geq c r^{-1/d}$. Therefore, given $\veps \in (0,1)$ let
	\begin{equation}
		n \geq  1 + C \int_0^{4/\veps} u^{2/d - 1} \, \mathrm{d} u = 1 + c \veps^{-2/d},
	\end{equation}
	where $c$ is a constant.
	However, by Theorem \ref{theorem:Morris_Peres} this means that the $n$-step transition probability is less than $D \cdot \veps$. Consequently, choosing the minimum $\veps$ in terms of $n$ yields $p_n(x,y) \leq C n^{-d / 2}$.
\end{proof}

To sum up, the desired proof now follows as an easy chain of implications. Indeed, applying Corollary \ref{coro:transient_at_least_cubic}, then Corollary \ref{coro:cubic_growth_implies_3dim_isoperimetry} and Corollary \ref{coro:3dim_isoperimetry_implies_heat_kernel_decay} for $d = 3$ yields Lemma \ref{lemma:heat_kernel_upper}.

\bigskip

{\bf Acknowledgements:}\\
We thank \'Ad\'am T\'im\'ar for suggesting the idea of the proof of  implication \eqref{unimod_B}$\implies$\eqref{unimod_A} of
 Proposition \ref{prop_QW_unimod} to us. We thank G\'abor Pete for sketching the proof of Lemma \ref{lemma:heat_kernel_upper} to us.

The work of M.~Borb\'enyi was partially supported by the UNKP-21-2 New National Excellence Program of the Ministry for Innovation and Technology from the source of the National Research, Development and Innovation Fund.

The work of B.~R\'ath was partially supported by grants NKFI-FK-123962 and NKFI-KKP-139502 of NKFI (National Research, Development and Innovation Office) and the ERC Synergy under Grant No. 810115 - DYNASNET. 

The work of S.~Rokob was partially supported by the ERC Consolidator Grant 772466 ``NOISE''.

\end{document}